\documentclass{article}
\usepackage{amsfonts,amsmath, amssymb,latexsym,comment}
\usepackage{graphicx}
\usepackage{xcolor}
\usepackage{color}
\usepackage{hyperref}

\newcommand{\nc}{\newcommand}
\nc{\wt}{\widetilde}
\nc{\wh}{\widehat}
\definecolor{ao(english)}{rgb}{0.0, 0.5, 0.0}

\def\binom#1#2{{#1\choose#2}}

\def\Z{{\mathbb Z}}
\def\ra{{\rightarrow}}

\def\SL{{\rm SL}}
\def\GL{{\rm GL}}
\def\SO{{\rm SO}}

\def\Stab{{\rm Stab}}

\def\Sym{{\rm Sym}}
\def\Gal{{\rm Gal}}
\def\Cl{{\rm Cl}}

\def\O{{\mathcal O}}
\def\P{{\mathbb P}}
\def\bG{{\mathbb G}}

\def\Bas{{\rm Bas}}
\def\PBas{{\rm PBas}}

\def\Aut{{\rm Aut}}
\def\Res{{\rm Res}}

\def\over{{\rm over}}
\def\sub{{\rm sub}}

\def\prim{{\rm prim}}
\def\bad{{\rm bad}}
\def\good{{\rm good}}

\def\diag{{\rm diag}}

\def\ind{{\rm ind}}
\def\nm{{\rm nm}}
\def\ker{{\rm ker}}

\def\Re{{\rm Re}}
\def\Vol{{\rm Vol}}
\def\covol{{\rm covol}}
\def\R{{\mathbb R}}
\def\cO{{\mathcal O}}
\def\cB{{\mathcal B}}

\def\F{{\mathbb F}}
\def\FF{{\mathcal F}}

\def\Q{{\mathbb Q}}

\def\V{{\mathcal V}}
\def\Z{{\mathbb Z}}
\def\P{{\mathbb P}}
\def\F{{\mathbb F}}
\def\Q{{\mathbb Q}}
\def\C{{\mathbb C}}

\def\I{{\mathcal I}}

\def\W{{\mathcal W}}

\def\wt{\widetilde}
\def\ol{\overline}
\def\ds{\displaystyle}

\newtheorem{Theorem}{Theorem}
\newtheorem{theorem}{Theorem}[section]
\newtheorem{corollary}[theorem]{Corollary}
\newtheorem{lemma}[theorem]{Lemma}
\newtheorem{remark}[theorem]{Remark}
\newtheorem{Remark}[Theorem]{Remark}
\newtheorem{proposition}[theorem]{Proposition}
\newtheorem{definition}[theorem]{Definition}
\newenvironment{proof}{\noindent {\bf Proof:}}{$\Box$ \vspace{2 ex}}
\usepackage{titlesec}
\titleformat{\part}
  {\normalfont\LARGE\bfseries} 
  {}
  {0pt} 
  {\partname~\thepart:\ }

\title{Secondary terms in the counting functions of quartic fields I}
\author{Arul Shankar and Jacob Tsimerman}

\begin{document}
\maketitle

\abstract{We prove that the smoothed counting function of the set of quartic fields, satisfying any finite set of local conditions, can be written as a linear combination of $X,X^{5/6}\log X,X^{5/6}$, upto an error term of $O(X^{13/16+o(1)})$. For certain sets of local conditions, namely, those cutting out ``$S_4$-families'' of quartic fields, we explicitly determine the leading constants of the secondary terms. We moreover express these constants in terms of secondary mass formulas associated to families of quartic fields.}

In our proof, we introduce a new method to count integer orbits on representations of reductive groups, one which allows for the recovery of lower order terms. This new method contains aspects of the tools used by Sato--Shintani to analyze zeta functions associated to prehomogeneous vector spaces and the geometry-of-numbers techniques pioneered by Bhargava.

\tableofcontents

\part{Introduction}

\section{Statements of the results}

A classical question in number theory is: how many number fields are there of degree $n>1$ and discriminant less than $X$ in absolute value? Denote this quantity by $N_n(X)$. It is a folklore conjecture that $N_n(X)\asymp X$ for all $n\geq 2$. Asymptotics for $N_2(X)$ are elementary to derive, but for higher $n$, asymptotics are known only in three cases. For $n=3$, this is due to Davenport--Heilbronn \cite{MR491593}; for $n=4$, by results of Cohen--Diaz y Diaz--Olivier \cite{Cohen_Diaz_Olivier} and Bhargava \cite{dodqf}; and for $n=5$, by work of Bhargava~\cite{dodpf}.

Finer information on the behavior of $N_n(X)$ is even more elusive. In the case of cubic fields, Roberts conjectured the existence of a secondary main term for $N_3(X)$ of size $\asymp X^{5/6}$. This conjecture was resolved by Bhargava and the two authors of this paper \cite{BST}, and simultaneously and independently by Taniguchi and Thorne \cite{MR3127806}. Moreover, these two works determine secondary main terms for the family of $S_3$-cubic fields satisfying any finite set of splitting conditions. However, secondary terms in the counting functions of degree-$n$ fields are unknown for all $n>3$. Even in the function fields case, where it is speculated (see \cite{BDPW}) that secondary terms correspond to secondary homological stability in the sense of \cite{E2_Map_Cl_Gps}, unconditional results are only known (to the authors knowledge) for cubic extensions of function fields by work of Zhao \cite{YZhao_thesis}, and the count of quartic $\F_p[t]$-algebras (along with a cubic resolvent algebra) by work of Chang \cite{KC}.

In this article, we consider families of quartic fields, with prescribed splitting conditions at finitely many places. To this end, let $S$ be a finite set of places, and let $\Sigma=(\Sigma_v)_{v\in S}$ be a {\it finite collection of local specifications for quartic fields}, where for each $v\in S$, $\Sigma_v$ is a finite set of \'etale quartic extensions of $\Q_v$. Let $F(\Sigma)$ denote the family of quartic fields $K$ such that $K\otimes\Q_v\in\Sigma_v$ for every $v\in S$.
We make an important simplification by considering smooth rather than sharp counts. Let $\psi:\R_{\geq 0}\to \R_{\geq 0}$ be a smooth function with compact support. We define the ``smooth count'' of quartic fields in $F(\Sigma)$ by
\begin{equation}\label{eq:fields_count}
N_{\Sigma}(\psi,X):=\sum_{K\in F(\Sigma)}\psi\Bigl(\frac{|\Delta(K)|}{X}\Bigr).
\end{equation}
Then we prove the following result:

\begin{Theorem}\label{thm:mainallfields}
Let $F(\Sigma)$ be the family of quartic fields corresponding to the finite collection of local specifications $\Sigma$. Let $\psi:\R_{\geq 0}\to\R_{\geq 0}$ be a smooth function with compact support. Then 
\begin{equation*}
N_{\Sigma}(\psi,X)=C_1(\Sigma,\psi)\cdot X+
C'_{5/6}(\Sigma,\psi)\cdot X^{5/6}\log X + C_{5/6}(\Sigma,\psi)\cdot X^{5/6} + O(X^{13/16+o(1)}),
\end{equation*}
for some constants $C_1(\Sigma,\psi)$, $C_{5/6}'(\Sigma,\psi)$, and $C_{5/6}(\Sigma,\psi)$.
\end{Theorem}
We may break the contribution to $N_\Sigma(\psi,X)$ according the the Galois groups of the fields in question. Foundational work of Cohen--Diaz y Diaz--Olivier \cite[Corollary 6.1]{Cohen_Diaz_Olivier} proves that the number of $D_4$-quartic fields, with discriminants bounded by $X$, grows like an explicit constant times $X$ with an error term of $O(X^{3/4+o(1)})$; meanwhile a breakthrough result of Bhargava \cite[Theorem 1]{dodqf} proves that the number of $S_4$-quartic fields, with discriminants bounded by $X$, is asymptotic to an explicit constant times $X$. The error term in the $S_4$-fields counting result was subsequently improved to 
a power saving of size $O(X^{23/24+o(1)})$ by Belabas--Bhargava--Pomorance~\cite{BBPEE}. 

The number of $V_4$-, $C_4$-, and $C_2\times C_2$-fields with discriminant less than $X$ bounded by $O(X^{1/2+o(1)})$ by work of Baily \cite{Baily}, and the number of $A_4$-quartic fields with discriminant less than $X$ is bounded by $O(X^{.778...})$ by \cite[Theorem 1.4]{BSTTTZ}. So the contribution to $N_\Sigma(\psi,X)$ from fields with these Galois groups is subsumed in the error term. The works \cite{Cohen_Diaz_Olivier,dodqf} readily generalize to counting families on which finitely many splitting conditions are imposed, and to smooth (rather than sharp) counts. As a consequence, the main term constant $C_1(\Sigma,\psi)$ can be determined from them. Moreover, since the error term in the $D_4$-fields count has been proven to be $O(X^{3/4+o(1)})$, the existence of the secondary term in the count of quartic fields is seen to be an $S_4$-fields phenomena.

For certain natural families $F(\Sigma)$, we give explicit description of the secondary term constants.
We say that $\Sigma$ is an {\it $S_4$-collection} and that $F(\Sigma)$ is an {\it $S_4$-family} if the conditions of $\Sigma$ automatically force every $K\in F(\Sigma)$ to be an $S_4$-quartic number field. We can construct $S_4$-families by imposing local conditions at two primes. For example, if $S$ consists of two primes $p_1$ and $p_2$, and $\Sigma_{p_1}=\{\Q_{p_1^4}\}$ (ensuring that the Galois group of any $K\in F(\Sigma)$ contains a $4$-cycle) and $\Sigma_{p_1}=\{\Q_{p_2^3}\oplus\Q_{p_2}\}$ (ensuring that the Galois group of any $K\in F(\Sigma)$ contains a $3$-cycle), then $F(\Sigma)$ is an $S_4$-family. For such families, we prove that the $X^{5/6}\log X$ term does not occur in the expansion of $N_\Sigma(\psi,X)$, and compute the leading constant of the $X^{5/6}$ term. To state the result, we need to introduce the following notation.

Let $R$ be a principal ideal domain. We say that a triple $(Q,C,r)$ is a {\it quartic triple over $R$} if $Q$ is a rank-$4$ ring over $R$, $C$ is a rank-$3$ ring over $R$ that is a cubic resolvent ring of $Q$, and $r:Q/R\to C/R$ is the (quadratic) resolvent map. For the definitions of cubic resolvent rings and the resolvent map, see Bhargava's landmark work \cite{BHCL3} parametrizing quartic rings. Given an element $x$ in $(C/R)^\vee$, we obtain (by composition with $r$) a quadratic form $Q/R\to R$, which we denote by $r_x$. Now let $K$ be a quartic \'etale extension of $\Q$ or $\Q_p$ for some $p$, and denote the ring of integers of $K$ by $\O_K$. Then, as proven by Bhargava \cite{BHCL3}, $\O_K$ has a unique cubic resolvent ring $C_K$. Bhargava proves this result for quartic maximal orders over $\Z$, but the same analysis holds when $\Z$ is replaced with $\Z_p$. We say that $(\O_K,C_K,r_K)$ is the {\it triple corresponding to $K$}. Denote the quadratic form corresponding to an element $x$ in $(C_K/\Z)^\vee$ or $(C_K/\Z_p)^\vee$ by $r_{K,x}$, and denote the set of primitive elements in $(C_K/Z_p)^\vee$ by $(C_K/\Z_p)^\vee_\prim$. We define the constants
\begin{equation}\label{eq:M_def}
\mathcal M:= \frac{2^{5/3}\Gamma(1/6)\Gamma(1/2)}{\sqrt{3}\pi\Gamma(2/3)};\;\;
\mathcal M_i:=\left\{
\begin{array}{rcl}
\mathcal M&{\rm if}&i\in\{0,2\};
\\
\sqrt{3}\cdot\mathcal M&{\rm if}&
i=1;
\end{array}
\right.
\;\mathcal M_i':=\left\{
\begin{array}{rcl}
\mathcal M&{\rm if }& i=0;\\
\sqrt{3}\cdot \mathcal M&{\rm if }& i=1;\\
\displaystyle\frac{\mathcal M}{3}&{\rm if }& i=2.
\end{array}
\right.
\end{equation}

Then we have the following result.

\begin{Theorem}\label{thm:mainS4fields}
Let $F(\Sigma)$ be an $S_4$-family of quartic fields corresponding to the finite $S_4$-collection of local specifications $\Sigma$. Let $\psi:\R_{\geq 0}\to\R_{\geq 0}$ be a smooth function with compact support. Then 
\begin{equation*}
N_{\Sigma}(\psi,X)=C_1(\Sigma)\wt{\psi}(1)\cdot X+ C_{5/6}(\Sigma)\wt{\psi}(5/6)\cdot X^{5/6} + O(X^{13/16+o(1)}),
\end{equation*}
where
\begin{equation}\label{eq:leadingcons}
C_1(\Sigma):=\frac{1}{2}\Bigl(\sum_{K\in\Sigma_\infty}\frac{1}{\#\Aut(K)}\Bigr)\prod_p
\Bigl(\sum_{K\in\Sigma_p}\frac{|\Delta(K)|_p}{\#\Aut(K)}\Bigr)\Bigl(1-\frac{1}{p}\Bigr),
\end{equation} and

\begin{align*}    
C_{5/6}(\Sigma)&:=\frac{\pi}{8}\Bigl(\mathcal M_{\Sigma}\cdot\zeta(1/3)\prod_{p}\Bigl(1-\frac{1}{p^{1/3}}\Bigr)\sum_{K\in\Sigma_p}\frac{|\Delta(K)|_p}{\#\Aut(K)}\cdot \int_{x\in (C_K/\Z_p)^{\vee}_{\prim}}  |\det{r_{K,x}}|_p^{-2/3}dx \\&\;\;\;+ \;\mathcal M_\Sigma'\cdot\zeta(2/3)\prod_{p}\Bigl(1-\frac{1}{p^{2/3}}\Bigr)\sum_{K\in\Sigma_p}\frac{|\Delta(K)|_p}{\#\Aut(K)}\cdot \int_{x\in (C_K/\Z_p)^{\vee}_{\prim}}  \epsilon_p(r_x)|\det{r_{K,x}}|_p^{-2/3}dx\Bigr).
\end{align*}
Above, $(\cO_K,C_K,r_K)$ is the triple corresponding to $K$, $\epsilon_p$ denotes the Hasse invariant, and 
\begin{equation*}
\mathcal M_\Sigma = \sum_{K=\R^{4-2i}times\C^i\in\Sigma_\infty}\frac{\mathcal M_i}{|\Aut(K)|};
\quad\quad
\mathcal M_\Sigma' = 
\sum_{K=\R^{4-2i}\times\C^i\in\Sigma_\infty}\frac{\mathcal M_i'}{|\Aut(K)|}.
\end{equation*}
\end{Theorem}
See \S\ref{sec:2ndmaintermvalues} for the explicit values in the product for various $\Sigma$.

In fact there is no need for us to restrict our count to maximal orders. We also prove an analogue of Theorem \ref{thm:mainS4fields}, where we count quartic rings along with their cubic resolvent rings. To this end, we say that $\Lambda$ is a {\it finite collection of local specifications for quartic rings} if $\Lambda=(\Lambda_v)_{v\in S}$, where for all $v$ in a finite set of places $S$, the set $\Lambda_v$ consists of a finite set of quartic triples over $\Z_v$ with non-zero discriminant. We define $R(\Lambda)$ to be the set of quartic triples $(Q,C,r)$ over $\Z$ whose base change to $\Z_v$ lies in $\Lambda_v$ for every $v\in S$. Let $\psi:\R_{\geq 0}\to\R$ be a smooth function with compact support. We define the smoothed count of quartic triples in $R(\Lambda)$ analogously to \eqref{eq:fields_count}:
\begin{equation}\label{eq:rings_count}
N_{\Lambda}(\psi,X):=\sum_{(Q,C,r)\in R(\Lambda)}\psi\Bigl(\frac{|\Delta(Q)|}{X}\Bigr).
\end{equation}
We say that $\Lambda$ is an $S_4$-collection and that $R(\Lambda)$ is an $S_4$-family if the collection $\Lambda$ forces every triple $(Q,C,r)$ to have the property that $Q$ is an order in a quartic $S_4$-field. Then we have the following result:

\begin{Theorem}\label{thm:mainS4rings}
Let $R(\Lambda)$ be an $S_4$-family of quartic triples corresponding to the finite $S_4$-collection $\Lambda$ of local specification for quartic rings. Let $\psi:\R_{\geq 0}\to\R_{\geq 0}$ be a smooth function with compact support. Then 
\begin{equation*}
N_{\Lambda}(\psi,X)=C_1(\Lambda)\wt{\psi}(1)\cdot X+ C_{5/6}(\Lambda)\wt{\psi}(5/6)\cdot X^{5/6} + O(X^{3/4}\log X),
\end{equation*}
where
\begin{equation*}\label{eq:leadingcons}
C_1(\Lambda)=\frac{1}{2}\Bigl(\sum_{K\in\Lambda_\infty}\frac{1}{\#\Aut(K)}\Bigr)\prod_p
\Bigl(\sum_{(Q,C,r)\in\Lambda_p}\frac{|\Delta(Q)|_p}{\#\Aut((Q,C,r)}\Bigr)\Bigl(1-\frac{1}{p}\Bigr),
\end{equation*} and
\begin{align*}\label{eq:secondcons}    
C_{5/6}(\Lambda)&=\frac{\pi}{8}\Bigl(\mathcal M_\Lambda\zeta(1/3)\prod_{p}\Bigl(1-\frac{1}{p^{1/3}}\Bigr)\cdot\sum_{(Q,C,r)\in\Lambda_p}\frac{|\Delta(Q)|_p}{\#\Aut(Q,C,r)}\cdot \int_{x\in (C/\Z_p)^{\vee}_{\prim}}  |\det{r_x}|_p^{-2/3}dx \\&\;\;\;+\; \mathcal M_\Lambda'\zeta(2/3)\prod_{p}\Bigl(1-\frac{1}{p^{2/3}}\Bigr)\cdot\sum_{(Q,C,r)\in\Lambda_p}\frac{|\Delta(Q)|_p}{\#\Aut(Q,C,r)}\cdot \int_{x\in (C/\Z_p)^{\vee}_{\prim}}  \epsilon_p(r_x)|\det{r_x}|_p^{-2/3}dx\Bigr).
\end{align*}
Above, notation is as in Theorem \ref{thm:mainS4fields}.
\end{Theorem}

\noindent We conclude the section with the following remarks regarding our main results.
\begin{Remark}{\rm 
\begin{itemize}
\item[{\rm (a)}] In the $S_4$-family case, asymptotics for the number of quartic fields and quartic rings along with cubic resolvent rings are both due to Bhargava \cite{dodpf}, and the proofs readily generalize to the smooth count case. Moreover, the ``mass formula'' expression of the leading constants $C_1(\Sigma)$ and $C_1(\Lambda)$ was used by Bhargava \cite{bhargava_mass} to derive heuristics for the number of $S_n$-fields with bounded discriminants for all $n$. Further evidence for these heuristics were given by work of the authors \cite{ST_heuristics}, in which this mass formula arises naturally from a heuristic count of $S_n$-fields.
\item[{\rm (b)}] There is a striking resemblance between the mass formula expression of the secondary term in the count of cubic rings and fields \cite[Theorem 7]{BST} and the first summands of $C_{5/6}(\Sigma)$ and $C_{5/6}(\Lambda)$, since, $|\det r_{K,x}|_p^{-1}$ is equal to $|1/4|_p i(x)$, where the $i(x)$ in \cite[Theorem 7]{BST} is the index of $\Z_p[x]$ in $C_K$. The second summand on the other hand does not have a clear analogue. Indeed, the quantity $\epsilon_p(r_x)$ is not possible to define by looking just at the cubic resolvent.
\item[{\rm (c)}] The difficulty in upgrading our main results from smooth counts to sharp is the following. As we will subsequently explain, our results are proved using lattice point counts. Since these lattices are in a 12-dimensional space, it is difficult to beat the error of $X^{5/6}$ using purely formal methods. For instance, consider the problem of approximating points in a $12$-dimensional region (as in a higher dimensional Gauss circle problem) whose boundary is cut out by a quadratic form $Q$,  and using the counting function $Q^{6}$ as a substitute for the discriminant. Then the corresponding zeta function would have a similar functional equation to the Shintani Zeta function, but now the the sharp count cannot possibly have an error term better than $X^{5/6}$ since $Q(x)=n$ has around $n^5$ solutions. 
\item[{\rm (d)}] In \cite{dodpf}, Bhargava also counts the average number of 2-torsion in class groups of cubic fields. Unfortunately, our result does not currently apply to this case for {\it any} congruence family of cubic fields, since the trivial element in $\Cl_2(K)$ corresponds to the quartic algebra $K\oplus\Q$, which is not $S_4$.
Indeed, even if we were to exclude the trivial element of $\Cl_2(K)$ from our count, there is no way to impose a congruence condition on the family of cubic fields that would force a $4$-cycle in the Galois groups of the corresponding quartic fields. In forthcoming work, we refine the methods developed in this paper so as to be able to recover secondary terms for the count of $2$-torsion in class groups of cubic fields as well.
\end{itemize}
   
}\end{Remark}

\section{Outline of the proofs}

Our proofs of Theorems \ref{thm:mainallfields}, \ref{thm:mainS4fields}, and \ref{thm:mainS4rings} rely on formal theory of Shintani zeta functions developed by Sato--Shintani \cite{SatoShintani}, applied in particular to Shintani zeta functions associated to the prehomogeneous representation $2\otimes\Sym^2(3)$ of $\GL_2\times\SL_3$. A landmark result of Bhargava \cite{BHCL3} proves that the integer orbits of this representation are in bijection with triples $(Q,C,r)$ over $\Z$. As a formal consequence, we see that Shintani zeta functions associated to this prehomogeneous representation have the form
\begin{equation*}
\xi_\Lambda(s)=\sum_{(Q,C,r)\in R(\Lambda)}\frac{|\Delta(Q)|^{-s}}{|\Aut(Q)|},
\end{equation*}
for a finite collection of local specifications $\Lambda$ (not necessarily $S_4$), with $\Lambda_\infty$ being a singleton set. As a consequence of the general theory of prehomogeneous vector spaces, it follows that these Shintani zeta functions satisfy a functional equation, and have analytic continuation to the whole complex plane with poles having location and multiplicity controlled by the zeros of the corresponding Bernstein polynomial. In our case: possible poles at $1$, $5/6$, and $3/4$ of at most order two each. 

It follows, again formally, that the smoothed count $N_\Lambda(\psi,X)$ has a power series expansion, with terms of magnitude $X^c\log X$ and $X^c$ for $c\in\{1,5/6,3/4\}$, with super polynomial error term. The leading constants of these main terms are given in terms of the residues of $\xi_\Lambda(s)$. The residues at $s=1$ are (at least in the $S_4$-family case)\footnote{For general (not necessarily $S_4$) families, the residues can probably be computed by the following procedure: asymptotics for the number of quartic \'etale algebras over $\Q$ are known (as described in the introduction). It will then be necessary to count orders inside these quartic algebras weighted by the number of their cubic resolvents, using the works of Nakagawa \cite{nakagawa} and Bhargava \cite{BHCL3}. However, there does not seem to be anywhere in the literature where all this is pieced together.} known by work of Bhargava \cite{dodqf}.
However, despite much foundational work on this subject, most notably by Yukie \cite{Yukie}, the residues at $s=5/6$ and $s=3/4$ are unknown.

To prove Theorem \ref{thm:mainallfields}, we need to execute a sieve allowing us to go from counting triples $(Q,C,r)$ to counting {\it maximal triples} (those for which $Q$ is a maximal order). Moreover, we need to do this without any explicit knowledge of the leading constants of the terms in the expansion of $N_\Lambda(\psi,X)$. We explain how we do this in more detail in \S2.1. Proving Theorem \ref{thm:mainS4rings} is equivalent to showing that when $\Lambda$ is an $S_4$-collection, $\xi_\Lambda(s)$ has only a simple pole at $s=5/6$, and explicitly determining the residue. We explain how we do this in more detail in \S2.2. Theorem \ref{thm:mainS4fields} follows by inputting the results of Theorem \ref{thm:mainS4rings} into the sieve in the proof of Theorem \ref{thm:mainallfields}.

\subsection{Proof Strategy : Theorem \ref{thm:mainallfields}}

Our proof of Theorem \ref{thm:mainallfields} relies heavily on the formalism of Shintani zeta functions\cite{SatoShintani}. Using this formalism, one immediately gets an asymptotic expansion for the count of all \textit{rings} satisfying any finite number of local conditions. We then follow the standard strategy of sieving down from rings to fields by imposing maximality conditions at more and more primes. Specifically, we have:
$$N_{\Sigma}(\psi;X)=\sum_n \mu(n)N_{\Sigma\cap \Sigma^{\rm nmax}_n}(\psi;X)$$
where $\Sigma^{\rm nmax}_n$ is the set of conditions insisting that our ring is non-maximal at primes dividing $n$.
There are four ingredients needed to carry out this sieve:

\begin{enumerate}
    \item For $n$ small, we need a power series expansion for $N_{\Sigma\cap \Sigma_n^{\rm nmax}}(\psi;X)$, recovering the first and second terms, with good $n$-dependence on the error term.
    \item For $n$ large, we need good uniformity estimates for the number of triples with bounded discriminant, which are non-maximal at all primes dividing $n$.
    \item We need to prove that the leading constants of the possible second main terms $c_n'X^{5/6}\log X$ and $c_nX^{5/6}$ of $N_{\Sigma\cap \Sigma_n^{\rm nmax}}(\psi;X)$, when summed over $n$, to converge.
    \item We need to prove that the contributions of the possible lower order terms $d_n'X^{3/4}\log X$ and $d_nX^{3/4}$ of $N_{\Sigma\cap \Sigma_n^{\rm nmax}}(\psi;X)$, when summed over $n$, to not dominate the secondary term of size $X^{5/6}$.
\end{enumerate}
The first problem is made more difficult in this setting than in that of cubic rings because the dimension of the corresponding prehomogeneous vector spaces jumps from $3$ to $12$. In the language of geometry-of-numbers, this means that we can only take $n$ up to $X^{\frac1{24}}$ before our balls don't `fit', instead of up to $X^{\frac1{8}}$ in the cubic case. In the language of Shintani zeta functions, this shows up in the discriminant factor gaining a factor of $n^{24}$ instead of $n^8$ in the functional equation. 

In the previous execution of this strategy in the cubic case \cite{BST,MR3127806}, one optimized the error in these methods by combining bounds on Fourier transforms of (non)maximal congruence conditions, together with a uniformity estimate of $\frac{X}{n^2}$ derived from studying cubic rings directly. An issue that arises is that the congruence condition defining non-maximality is quite complicated and has a large Fourier transform. We improve on this method by breaking the non-maximal congruence condition at $p$ into two (see \S\ref{sec:switching} for more details). Loosely speaking these are:

\begin{itemize}
    \item Those that correspond to being non-maximal of index $p$, and
    \item Everything else
\end{itemize}

The trick is that the first function has a very small Fourier transform, whereas the second has a better uniformity estimate of $\frac{X}{p^4}$. Combining these makes the method work. 

To solve the second problem of bounding the constant terms $c_n$, we simply use our asymptotic formula and plug in an optimal value of $X$. Combined with the improvements coming from the splitting derived above, this ends up being strong enough to make the sum converge.

\begin{Remark}{\rm 
    We use the language of \emph{Switching Correspondences} analogously to the authors previous paper with Bhargava \cite{BST} where we broke up the non-maximal condition \textit{precisely} into two simpler ones induced by correspondences. In the quartic case, while it is possible (though non-trivial, at least to the authors!) to break up the non-maximal conditions into a large sum of functions corresponding to `switching', doing so introduces large coefficients (both positive and negative, of course) which ends up making it difficult to control the error. Hence, we use this less precise version, where we use exactly one switching correspondence, and treat the (multi) set of remaining nonmaximal quartic fields as an error term to be bounded.
}\end{Remark}

\subsection{Proof Strategy: Counting}

If one wants to obtain actual asymptotics as opposed to just proving the existence of an asymptotic, one needs to actually compute the residues of the Shintani zeta function. Shintani's methods produce mysterious terms that are difficult to get a handle on. Instead, we use geometry of numbers methods to count in a fundamental domain directly, and derive a precise asymptotic up to an $o(X^{5/6})$. Note that this yields a count for triples $(Q,C,r)$, not for fields. However, for this approach one may impose finitely many congruence condition without making the problem any more difficult. Therefore, we use this method to compute the residues of the Shintani zeta function (to the right of 5/6), and then we turn around and plug these values into the sieving formalism described above.

As in Bhargava's work \cite{dodqf}, we start with needing to evaluate an integral of the form
$$\int_{g\in\FF} \#(L\cap g\cB)^{\Delta<X}dg,$$
for a fundamental domain $\FF$ for the action of $\GL_2(\Z)\times\SL_3(\Z)$ on $\GL_2(\R)\times\SL_3(\R)$, where $L\subset V(\Z)$ is an $S_4$-set defined by finitely many congruence conditions, and $\cB$ is a (in our case, a smoothed out) ball. Bhargava approximates the main term of this integral via the following procedure.
\begin{itemize}
    \item For $g$ in the `main body', one may simply approximate the number of points by the volume.
    \item For $g$ near cusps the box becomes skewed, and these contributions are bounded using Davenports lemma.
    \item For $g$ deep in the cusp, enough `co-ordinates' are forced to become 0 (by virtue of being integers) that the lattice points represent non $S_4$-rings and this region can be discarded.
\end{itemize}  
For our purposes, since we need higher order terms, we must compute the number of points in skewed balls precisely. We proceed as follows:
\begin{itemize}
    \item In each box, we divide the coordinates into one of three ranges: zero, small, and large.
    \item For the coefficients whose ranges are large, we may approximate their contribution by the volume in that direction, and `project' onto the remaining coefficients
    \item For the coordinates whose ranges are small, we do not approximate at all, and instead record the sum over these coordinates as a Mellin integral over an appropriate zeta function (possibly in several variables).
    \item We then combine the zeta integrals with the integral over the group to obtain higher-dimensional integrals with polar divisors, where certain polar contributions constitute the higher-order poles of the Shintani zeta function.
    \item Using the Iwasawa ($NAK$) decomposition, it turns out that only the toric contribution needs to be kept. The compact $K$ can be discarded by making our box $K$-invariant, and the unipotent $N$ only ever adds coordinates with zero range to other coordinates (not affecting the count), or coordinates with small ranges to ones with large range, not affecting the projection. This allows for a clean analysis of polar divisors.
    \item We use the Shintani formalism to rule out potential polar contributions at values of $s$ ruled out by the Bernstein polynomial.

\end{itemize}

\begin{Remark}{\rm 
\begin{enumerate}
\item In principle, our procedure should allows us to compute the residues without requiring the $S_4$-condition. However, this assumption simplifies our task greatly as it allows us not to venture into certain regions of the cusp (as explained later, these regions are: $a_{11}=b_{11}=0$ and $\det A=0$). This cuts down the number of possibilities and cusps we have to consider. We expect that our method can be used to tackle the question of smoothly counting all quartic fields, but this requires going deeper into the cusp. We undertake this in forthcoming work.
\item While we use the Shintani formalism to rule out certain poles, our method also rules out certain poles that are theoretically permitted by the Shintani formalism. In particular, for $S_4$-families, we show that there is only a simple pole at $s=5/6$, as opposed to a double pole.
\end{enumerate}
}\end{Remark}

\subsection*{Outline of the paper}

In Part II we set up the main tools used in the paper. First, in \S3, we review the parametrization setup for the prehomogeneous space corresponding to triples $(Q,C,r)$, the corresponding Shintani Zeta functions, and our algebraic groups. Section 4 introduces the Mellin transform and describes how we use it to translate questions about counting points smoothly to zeta functions. We also introduce convenient notation for zeta functions in multiple variables to streamline arguments we make such as `projecting away' large variables. It is important to do this carefully, as we can pick up the behavior of the smooth count as a variable transitions from a large range to a small range via certain residues of an appropriate zeta function. In \S5 we set up the ``global zeta integral'' formalism, which allows us to compute the residues of Shintani zeta functions from smooth counts.

In Part III we execute the count. \S6 and \S7 are similar, dealing with different parts of the integral corresponding to whether or not we are near the $\SL_2$-cusp. The second main term ends up coming from the cusp regions studied in \S7, and we obtain it as a special-value of a different Shintani zeta function corresponding to symmetric 3$\times$ 3 matrices. In \S8 we compute this special value using work of Ibukiyama--Sato. We also re-express the answer in language of rings, and compute it explicitly for maximal orders. This computation is merely to get an explicit answer, and we do not have any special insight about the final answer. Having done this, Theorem \ref{thm:mainS4rings} follows.

In Part IV we execute the sieve. This part black-boxes the results in Part III completely. We begin in \S10 by recalling Bhargava's results on non-maximal quartic rings, and also develop the ``switching correspondences'' that we need. Then in \S11, we prove uniformity estimates on various types of non-maximal rings, relying on work of Bhargava and Nakagawa. We also use heavily the functional equation of the Shintani zeta function, and as such need estimates on the {\it Fourier transforms} of the characteristic functions of non-maximal elements. Here we use the work of Hough \cite{hough}. Finally in \S12 we execute the sieve, which allows us to prove Theorems \ref{thm:mainallfields} and \ref{thm:mainS4fields}.

\subsection*{Acknowledgments}

It is a pleasure to thank Evan O'Dorney, Jordan Ellenberg, Fatemehzahra Janbazi, Will Sawin, Takashi Taniguchi, and Frank Thorne for helpful discussions and comments. The first-named author was supported by a Simons fellowship and an NSERC discovery grant.

\part{Setup}

\section{Preliminaries}

In this section, we collect resu.lts of Bhargava \cite{BHCL3} on the parametrization of quartic rings and fields and results of Sato--Shintani \cite{SatoShintani} on the Shintani zeta functions associated to $V$. 

\subsection{The parametrization of quartic rings and fields}

We begin with the parametrization of cubic rings. Let $U=\Sym^3(2)$ be the space of binary cubic forms. That is, if $R$ is any ring, $U(R)$ is the set of all elements $\{ax^3+bx^2y+cxy^2+dy^3:a,b,c,d\in R\}$. The group $\GL_2(R)$ acts on $U(R)$ via the ``twisted'' action $\gamma\cdot f(x,y)=f((x,y)\cdot \gamma)/\det(\gamma)$. We denote the discriminant of $f$ by $\Delta(f)$ and consider $a$, $b$, $c$, and $d$ to be functions on $U(R)$ in the obvious way. The following result is due to works of Levi \cite{Levi}, Delone--Fadeev \cite{DFCF}, and Gan--Gross--Savin \cite{GGSCF} for the case $R=\Z$, and due to Gross--Lucianovic \cite{MR2521487} for PIDs:
\begin{theorem}\label{th:cubic_param}
Let $R$ be a principal ideal domain. Then there is a natural bijection between isomorphism classes of cubic rings over $R$ and $\GL_2(R)$-orbits on $U(R)$, satisfying the following properties.
\begin{enumerate}
\item If the $\GL_2(R)$-orbit of $f\in U(R)$ corresponds to the cubic ring $R_f$, then $\Delta(R_f)=\Delta(f)$.
\item For $f$ and $R_f$ as above, $\Aut(R_f)$ is isomorphic to $\Stab_{\GL_2(R)}(f)$.
\end{enumerate}
\end{theorem}

Recall that we denote the space of pairs of ternary quadratic forms by $V=2\times\Sym^2(3)$. For a ring $R$, we represent elements in $V(R)$ by pairs $(A,B)$ of $3\times 3$ symmetric matrices, where $A$ and $B$ are the Gram matrices of the corresponding quadratic forms. Denote the coefficients of $(A,B)$ by $a_{ij}$ and $b_{ij}$ and write
\begin{equation*}
(A,B)=\left[\left(
\begin{array}{ccc}
a_{11} & \frac{a_{12}}{2} & \frac{a_{13}}{2}\\[.05in]
\frac{a_{12}}{2} & a_{22} & \frac{a_{23}}{2}\\[.05in]
\frac{a_{13}}{2} & \frac{a_{23}}{2} & a_{33}
\end{array}
\right),
\left(
\begin{array}{ccc}
b_{11} & \frac{b_{12}}{2} & \frac{b_{13}}{2}\\[.05in]
\frac{b_{12}}{2} & b_{22} & \frac{b_{23}}{2}\\[.05in]
\frac{b_{13}}{2} & \frac{b_{23}}{2} & b_{33}
\end{array}
\right)
\right],
\end{equation*}
where $a_{ij}$ and $b_{ij}\in R$. We will consider $a_{ij}$ and $b_{ij}$ to be functions from $V(R)$ to $R$ in the obvious way.

The group $\GL_2\times\GL_3$ acts on $V$ via a linear change of variables:
\begin{equation*}
(\gamma_2,\gamma_3)\cdot(A,B)=\Bigl(\gamma_2\Bigl(
\begin{array}{c} \gamma_3A\gamma_3^t\\\gamma_3B\gamma_3^t\end{array}\Bigr)\Bigr)^t.
\end{equation*}
We define the algebraic group $G$ to be the following subgroup of $\GL_2\times\GL_3$:
\begin{equation}\label{eq:G_definition}
G:=\{(g_2,g_3)\in\GL_2\times\GL_3:\det(g_2)\det(g_3)=1\}.
\end{equation}
Consider the determinant map
\begin{equation*}
p_2:G\to \GL_2\cong\{(\lambda,g_2)\in\mathbb{G}_m\times\GL_2 :\lambda\det(g_2)=1\},
\end{equation*}
which sends $(g_2,g_3)$ to $(\det(g_3),g_2)$. The kernel of this map is the (normal) subgroup $\SL_3$ of $G$.
The $\SL_3$-invariants of $V$ are the coefficients of the {\it cubic resolvent form}, where the cubic resolvent map on $V(R)$ is given by
\begin{equation*}
\Res:V\to U,\quad\quad\Res(A,B):=4\det(Ax-By).
\end{equation*}
The actions of $G$ on $V$ and $U=\Sym^3(2)$, the space of binary cubic forms are equivariant, in the sense that $\Res(g\cdot(A,B))=p_2(g)\cdot\Res(A,B)$. Therefore, the $G$-relative invariants of $V$ are the same as the $\GL_2$-relative invariants on $U$. The ring of relative invariants for the latter action is generated by the discriminant. Define the {\it discriminant} polynomial $\Delta\in\Z[V]$ to be defined by $\Delta(A,B):=\Delta(\Res(A,B))$, which is a degree $12$ homogeneous polynomial in the coefficients of $A$ and $B$. Then it follows that $\Delta$ generates the ring of relative invariants for the action of $\GL_2\times\GL_3$ on $V$.

The following result is due to Bhargava \cite{BHCL3} in the case $R=\Z$ and Wood \cite{WoodThesis,MR2948473} for the case when $R$ is a PID. (In fact, Wood's generalization is vastly more general, holding in cases when $\Z$ can be replaced with an aritrary base scheme. But we only need this for the PID case.)

\begin{theorem}\label{th:quartic_param}
Let $R$ be a principal ideal domain. There is a natural bijection between isomorphism classes of triples $(Q,C,r)$, where $Q$ is a quartic ring and $C$ is a cubic resolvent ring of $Q$, and $G(R)$-orbits on $V(R)$, satisfying the following properties.  under this bijection, then the following are true:
\begin{itemize}
\item[{\rm (a)}] If $(Q,C,r)$ corresponds to  $(A,B)$, then $C$ corresponds to the $\GL_2(R)$-orbit of $\Res(A,B)$ under the parametrization of Theorem \ref{th:cubic_param}. Moreover, we have $\Delta(Q)=\Delta(C)=\Delta(A,B)$.
\item[{\rm (b)}] For $(Q,C,r)$ and $(A,B)$ as above, $\Aut(Q,C,r)$ is isomorphic to $\Stab_{G(R)}(A,B)$.
\end{itemize}
\end{theorem}
See \cite[\S3.2, \S3.3]{BHCL3} for a complete description of basis' and multiplication tables for $Q$ and $C$ in terms of the coefficients of $A$ and $B$.

\subsection{Fundamental domains and measures}

In this subsection, we set notation for Iwasawa coordinates on $G(\R)$, describe a fundamental domain for the action of $G(\Z)$ on $G(\R)$, describe the real orbits for the action of $G(\R)$ on $V(\R)$, and prove a (standard) change of measures formula.

\subsubsection*{Iwasawa coordinates}
The Iwasawa decomposition allows us to write
\begin{equation*}
G(\R)=\Lambda NAK,
\end{equation*}
where $N$ is the subgroup $\{(u_2,u_3)\}$ of pairs of unipotent lower triangular matrices, $A$ is the subgroup of pairs of diagonal matrices, $K=\SO_2(\R)\times\SO_3(\R)$ is a maximal compact subgroup of $G(\R)$, and $\Lambda\cong\R^\times$ is the group of elements $(\lambda_2,\lambda_3)$, where $\lambda_2$ is the $2\times 2$ diagonal matrix with $\lambda^{-3}$ as it's coefficients, and $\lambda_3=$ is the $3\times 3$ diagonal matrix with $\lambda^2$ as its coefficients, for $\lambda\in\R^\times$. It is easy to see that $(\lambda_2,\lambda_3)\in\R^\times$ acts on elements in $V(\R)$ by scalar multiplication by $\lambda$, and so we will denote elements in $\Lambda$ simply by $\lambda\in\R^\times$. We write elements in $N$ as $u=(u_{12},u'_{21},u'_{31},u'_{32})$, where $u_{ij}$ (resp.\ $u'_{ij}$ denote the $ij$th coefficient of the lower triangular unipotent matrix $u_2$ (resp.\ $u_3$), and elements of $A$ as $s=(t,s_1,s_2)$, where the $2\times 2$ matrix corresponding to $s$ has $t^{-1}$ and $t$ as its diagonal coefficients while the $3\times 3$ matrix corresponding to $s$ has $s_1^{-2}s_2^{-1}$, $s_1s_2^{-1}$, and $s_1s_2^2$ as its diagonal coefficients. In these coordinates, 
\begin{equation*}
dg=t^{-2}s_1^{-6}s_2^{-6}d^\times\lambda dud^\times sdk
\end{equation*}
is a Haar-measure on $G(\R)$, where $du$ is Haar-measure on $N(\R)$ normalized so that $N(\Z)$ has covolume-$1$ in $N(\R)$, $dk$ is Haar-measure on $K$ normalized so that $K$ has volume $1$, and $d^\times\theta$ denotes $\theta^{-1}d\theta$ for any~$\theta$.

This induces a natural measure for $\SL_2$ and $\SL_3$ as well. We refer to these as $dg_2$ and $dg_3$.

\subsubsection*{A fundamental domain for $G(\Z)\backslash G(\R)$}
We next describe a fundamental domain $\FF$ for the action of
$G(\Z)$ on $G(\R)$. Such an $\FF$ is expressible in the form $\R_{>0}\times\FF_2\times\FF_3$, where $\FF_i$ is a fundamental domain for the action of $\SL_i(\Z)$ on $\SL_i(\R)$. Let $\FF_2$ denote Gauss' fundamental domain for the action of $\SL_2(\Z)$ on $\SL_2(\R)$ (see \cite[\S5.1]{BST} for an explicit description). The domain $\FF_3$ can be sandwiched between two Siegel domains $\mathcal{S}_1\subset\FF\subset\mathcal{S}_2$, where
\begin{equation*}
    \mathcal{S}_1=\overline{N_3}\{(s_1,s_2):s_1,s_2>C\}\SO_3(\R),\quad
    \mathcal{S}_2=\overline{N_3}\{(s_1,s_2):s_1,s_2>c\}\SO_3(\R),
\end{equation*}
for positive real numbers $c<C$ and a fundamental domain $\overline{N_3}$ for the action of $N_3(\Z)$ on $N_3(\R)$. We choose $\FF_3$ to be the explicit {\it box shaped at infinity} fundamental domain constructed in work of Grenier \cite[\S6]{MR934172}.
We define $\delta_{\FF_3}(s_1,s_2)$ to be the measure of $$N_{s_1,s_2}:=\{n\subset \ol{N}\mid n(s_1,s_2)\in \FF_3\}.$$ This will come up as we frequently deal with functions that are $N$-invariant.  We define $\delta_{\FF_2}(t)$ analogously. Since $\FF_3$ is box shaped at infinity, it follows that for $s_1$ large enough, $\delta_{\FF_3}(s_1,s_2)$ only depends on $s_2$, for $s_2$ large enough, $\delta_{\FF_3}(s_1,s_2)$ only depends on $s_1$, and for $s_1$ and $s_2$ large enough, 
$\delta_{\FF_3}(s_1,s_2)=1$.

\subsubsection*{The action of $G(\R)$ on $V(\R)$}
Since $(G,V)$ is prehomogeneous, it follows that the $G(\C)$-action on $V(\C)$ has one open orbit. Indeed, the set of elements with nonzero discriminant form a single $G(\C)$-orbit, and the stabilizer in $G(\C)$ of any such element is isomorphic to $S_4=\Aut(\C^4)$, in accordance with Theorem \ref{th:quartic_param}. The situation over $\R$ is only slightly more complicated: the set of elements in $V(\R)$ having nonzero discriminant break up into three open orbits, one each corresponding to the \'etale quartic algebras $\R^4$, $\R^2\times\C$, and $\C^2$ over $\R$. We denote the set of elements in $V(\R)$ corresponding to these three orbits by $V(\R)^{(0)}$, $V(\R)^{(1)}$, and $V(\R)^{(2)}$, respectively. The stabilizers in $G(\R)$ of elements in these orbits are respectively isomorphic to $S_4\cong\Aut(\R^4)$, $V_4\cong\Aut(\R^2\times\C)$, and $D_4\cong\Aut(\C^2)$.

The torus $\Lambda A$ acts on $V$, and scales each coefficient by an amount that we call the {\it weight} of that coefficient. We denote this weight function by $w$, and explicitly write the weights of the 12 coefficients $(a_{ij}))_{1\leq i\leq j\leq 3}$ and $(b_{ij}))_{1\leq i\leq j\leq 3}$ to be
\begin{equation*}
\left(
\begin{array}{ccc}
\lambda t^{-1}s_1^{-4}s_2^{-2} & \lambda  t^{-1}s_1^{-1}s_2^{-2} & \lambda t^{-1}s_1^{-1}s_2\\[.05in]
 & \lambda t^{-1}s_1^{2}s_2^{-2} & \lambda t^{-1}s_1^{2}s_2\\[.05in]
 &  & \lambda t^{-1}s_1^{2}s_2^{4}
\end{array}
\right),
\left(
\begin{array}{ccc}
\lambda ts_1^{-4}s_2^{-2} & \lambda ts_1^{-1}s_2^{-2} & \lambda ts_1^{-1}s_2\\[.05in]
 &\lambda ts_1^{2}s_2^{-2} &\lambda ts_1^{2}s_2\\[.05in]
 &  &\lambda ts_1^{2}s_2^{4}
\end{array}
\right).
\end{equation*}
Then for $c\in\{a_{ij},b_{ij}\}$ and $g=\lambda n(t,s_1,s_2)k$ written in Iwasawa coordinates, we write $w_c(g)$ for the weight (which of course only depends on $\lambda$, $t$, $s_1$, and $s_2$.

\subsection{Smoothness and choices of measure}\label{sec:measures}

We begin by proving some smoothness results on our group actions. Note that the resolvent map $\Res:V\to U$ is $G$-equivariant and $\SL_3$-invariant. We have the following result.
\begin{lemma}\label{lem:smoothness}
The actions of $G$ on $V^{\Delta\neq 0}$ and $\GL_2$ on $U^{\Delta\neq 0}$ are smooth. Moreover, the resolvent map $\Res: V\to U$ identifies $U^{\Delta\neq 0}$ with $\SL_3\backslash V^{\Delta\neq 0}$.
\end{lemma}
\begin{proof}
We first show that $G$ acts smoothly on $V^{\Delta\neq 0}$. Since $V^{\Delta\neq0}$ is a single $G$-orbit, it suffices to show that the stabilizer is \'etale at a single $\Z$-point. Pick $x\in V^{\Delta\neq 0}(\Z)$ representing the quartic algebra $\Z^4$. Then by Theorem \ref{th:quartic_param}, we have $\#\Stab_x(\ol{\F_p})=\#\Stab_x(\ol\Q)=24$ for all primes $p$. Therefore $\Stab_x$ is etale over $\Z$, proving the claim. The same proof works for the action of $\GL_2$ on $U^{\Delta\neq 0}$.

We move on to the second claim of the lemma. Since $G$ acts smoothly on $V^{\Delta\neq 0}$ and on $U$, it follows that $\Res:\SL_3\backslash V^{\Delta\neq 0}\to U^{\Delta\neq 0}$ is \'etale. Comparing stabilizers using Theorem \ref{th:quartic_param} shows that $\Res$ is degree~$1$. Since $G$ acts transitively on both the source and then target, it follows that $\Res$ is an isomorphism as claimed.
\end{proof}

For a smooth and connected group scheme $H/\Z$, we define $\omega_H$ to be a (unique up to sign) top degree left-invariant differential form over $\Z$. For $R=\R$ or $\Z_p$ we denote the corresponding measures on $H(R)$ by $\nu_H$ - the choice of $R$ will always be clear from context. Note that for an exact sequence of smooth groups $1\to H_1\to H\to H_2$, then $\nu_{H_2}$ is the quotient measure of $\nu_H$ by $(H_1,\nu_{H_1})$.

Let $\omega_V$ (resp.\ $\omega_U$) be the top-degree differential form on $V$ (resp.\ $U$), such that the corresponding measure on $V(\R)$ (resp.\ $U(\R)$) is normalized so that $V(\Z)$ (resp.\ $U(\Z)$) has covolume $1$. Note that this corresponds with the above definition if we give $V$ (resp.\ $U$) the $\Z$-structure corresponding to $V(\Z)$ (resp.\ $U(\Z)$). 
We have the following consequences of the above lemma.

\begin{proposition}\label{prop:first_change_of_measures}
Let $R$ be $\R$, or $\Z_p$ for some $p$, and let $f\in U(R)$ and $x\in V(R)$ be any elements having nonzero discriminant. Let $\phi_f:\GL_2(R)\to U(R)$ and $\phi_x:G(R)\to V(R)$ be the maps sending $g_2\mapsto g_2\cdot f$ and $g\mapsto g\cdot x$, respectively. Then we have 
\begin{equation*}
\frac{\nu_U\mid_{\GL_2(R)\cdot f}}{|\Delta(f)|}=\frac{(\phi_f)_*\nu_{\GL_2}}{\#\Stab_{\GL_2(R)}(f)} ,
\quad\quad
\frac{\nu_V\mid_{G(R)\cdot x}}{|\Delta(x)|}=\frac{(\phi_f)_*\nu_{G}}{\#\Stab_{G(R)}(x)},
\end{equation*}
where $||$ denotes absolute value when $R=\R$, and $||=||_p$ when $R=\Z_p$. Above, we denote $\nu\mid_{S}$ to mean the restriction of $\nu$ to the set $S$.
\end{proposition}
\begin{proof}
The proofs of the two claimed equalities are identical so we only consider the first case. Since the measures are both given by differential forms, it is enough to show that $\phi_f^*\frac{\omega_U}{\Delta} =\pm\Delta(f)\cdot\omega_{\GL_2}$. Now for each $f$ both forms are left $G$-invariant, and so their quotient is a regular function $C(f)$ on $V$. Moreover, this function is clearly $G$-invariant, and $V$ has an open $G$-orbit, $C(f)$ must be a constant. Finally, since we've shown that the action is smooth at the point $x$ above, it follows that $C(f)\in\Z^{\times}$, as desired.
\end{proof}

Over $\R$, since the measures $dg$ and $\omega_G$ differ by a constant, may write $\omega_G=Jdg$ for some $J\in\R_{>0}$. Then we have 
\begin{equation}\label{eq:J_def}
J:=\frac{\nu_G}{dg}=6\cdot\frac{\nu_{\SL_2}}{dg_2}\cdot\frac{\nu_{\SL_3}}{dg_3}
\end{equation}
where the equality follows from the degree 6 isogeny $\SL_2\times\SL_3\times\mathbb{G}_m\to G$.
Next we have the following result, which is a consequence of the second claim of Lemma \ref{lem:smoothness}.
\begin{proposition}\label{prop:jac_sl_3}
Let $R$ be $\R$ or $\Z_p$, and let $\phi:V(R)\to\R$ be a measurable function. Then we have
\begin{equation*}
\int_{x\in V(R)}\phi(x)\nu_V(x)=\int_{f\in U(R)}
\sum_{x\in \frac{V(R)\cap\Res^{-1}(f)}{\SL_3(R)}}
\frac{1}{\#\Stab_{\SL_3(R)}(x)}\int_{g\in\SL_3(R)}\phi(g\cdot x)\nu_{\SL_3}(g)\nu_U(f).
\end{equation*}
\end{proposition}

\subsection{Shintani zeta functions associated to $(G,V)$}

In this subsection we introduce and set up notation for the theory of Shintani zeta functions associated to the representation $V$ of $G$. We only need results from the general theory of prehomogeneous vector spaces, due to Sato--Shintani, specialized to our case. See Kimura's book on the topic \cite{Kimura_book} for a clear exposition. It is worth noting that much foundational work has previously done by Yukie \cite{Yukie} by analyzing the corresponding Shintani zeta functions.

\subsubsection*{Dual spaces and nondegenerate $G(\R)$-orbits} 
We identify the dual space $V(\R)^\vee$ with $V(\R)$ using the standard inner product 
\begin{equation*}
[(A,B),(A',B')]:=\sum_{i\leq j}(a_{ij}a'_{ij}+b_{ij}b'_{ij}).
\end{equation*}
This inner product is $G(\R)$-equivariant in that we have $[g\cdot x,g^{-T}x']=[x,x']$, where $(g_2,g_3)^{-T}:=(g_2^{-T},g_3^{-T})$ and we use the superscript $-T$ to denote inverse transpose. Let $\Delta^*\in\Z[V]$ be the polynomial $\Delta^*(A,B):=\Delta(\det(Ax-By))=2^6\Delta(A,B)$.
We use $V(\Q)^\vee$ to denote the set of elements in $V(\R)^\vee$ with rational coefficients, note that $V(\Q)^\vee$ is identified with the dual of $V(\Q)$. Let $L\subset V(\Z)$ be any lattice. Define the dual lattice of $L$ by $L^\vee:=\{y\in V(\Q)^\vee\mid [y,V(\Z)]\subset\Z\}$. Then $V(\Z)^\vee$, the dual of $V(\Z)$, can be naturally identified with the set of pairs of integral $3\times 3$ symmetric matrices. For any ring $R$, we let $V(R)^\vee$ denote the set of pairs of $3\times 3$ symmetric matrices with coefficients in $R$.
Let $N$ be a positive integer and let $\phi:V(\Z/N\Z)\to\C$ be a function. We define the Fourier transform $\widehat{\phi}:V(\Z/N\Z)^\vee\to\C$ by
\begin{equation*}
\widehat{\phi}(y):=\frac{1}{N^{12}}\sum_{x\in V(\Z/N\Z)} e\Big(\frac{[x,y]}{N}\Big)\phi(x).
\end{equation*}
We will note for future use that the Fourier transform of $g\phi$ can be easily computed to be $\widehat{g\phi}=g^{-T}\widehat{\phi}$.

\subsubsection*{Analytic continuation, poles, and functional equations for Shintani zeta functions} Let $(A,B)\in V(\R)$ be an element with nonzero discriminant. Then the conics cut out by $A$ and $B$ intersect in four distinct $\Gal(\C/\R)$-invariant points in $\P^2(\C)$. For $i\in\{0,1,2\}$, let $V(\R)^{(i)}\subset V(\R)$ be the set of elements where the associated four points in $\P^2(\C)$ consists of $i$-pairs of complex conjugate points and $4-2i$ real points. Equivalently, $V(\R)^{(i)}$ consists of the elements in $V(\R)$ corresponding to the $\R$-algebra $R_i:=\R^{4-2i}\times \C^i$ under the parametrization of Theorem \ref{th:quartic_param}. For any subset $S$ of $V(\R)$, we use $S^{(i)}$ to denote $S\cap V(\R)^{(i)}$. For each $i$, the set $V(\R)^{(i)}$ is a single $G(\R)$-orbit, and the stabilizer of any element in $V(\R)^{(i)}$ in $G(\R)$ is isomorphic to $\Aut(R_i)$ and has size $\sigma_i$, where $\sigma_0=24$, $\sigma_1=4$, and $\sigma_2=8$.
Then for $i\in\{0,1,2\}$, the Shintani zeta functions associated to $L$ and $L^\vee$ are:
\begin{equation*}
\begin{array}{rcl}
\xi_{i,L}(s)&:=&\displaystyle\sum_{x\in G(\Z)\backslash L^{(i)}}
\frac{1}{|\Delta(x)|^s|\Stab_{G(\Z)}(x)|},
\\[.15in]
\xi_{i,L^\vee}^{*}(s)&:=&\displaystyle\sum_{x\in G(\Z)\backslash L^{\vee,(i)}}
\frac{1}{|\Delta^*(x)|^s|\Stab_{G(\Z)}(x)|}.
\end{array}
\end{equation*}
Then we have the following result, primarily due to Sato--Shintani \cite{SatoShintani}.
\begin{theorem}\label{th:SZF_v1}
The functions $\xi_{i,L}(s)$ and $\xi_{i,L^\vee}^{*}(s)$ continue to meromorphic functions on $\C$ with only possible double poles at $1$, $5/6$, and $3/4$. Moreover, they satisfy the function equation
\begin{equation*}
\xi_{i,L}(1-s)=\covol(L)^{-1}\gamma(s-1)\sum_{j\in\{0,1,2\}}c_{ji}(s)\xi^*_{j,L^\vee}(s),
\end{equation*}
where $\gamma(s)=\Gamma(s+5/4)^2\Gamma(s+7/6)^2\Gamma(s+1)^4\Gamma(s+5/6)^2\Gamma(s+3/4)^2$ and the $c_{ij}(s)$ are entire functions not depending on $L$. Moreover, the functions $\gamma(s-1)c_{ji}(s)$ are polynomially bounded in vertical strips. I.e. in a region with bounded real part, they are uniformly bounded at $\sigma+it$ by $|t|^{O(1)}$ as $|t|\ra\infty$. 
\end{theorem}
\begin{proof}
From Sato and Shintani's general theory of prehomogeneous vector spaces (see, for example, \cite[Theorem 5.18]{Kimura_book}), it follows that the functions $\xi_{i,L}(s)$ and $\xi_{i,L^\vee}^{*}(s)$ continue to meromorphic functions and satisfy a functional equation of the given form. The boundedness claim in vertical strips follows from the fact that $|\Gamma(\sigma+it)|\leq e^{-\pi|t|/2}\cdot |t|^{(O(1)}$ as $|t|\ra\infty$.

The explicit form of $\gamma(s)$ as well as the location of the poles of these zeta functions are controlled by the Bernstein--Sato polynomial associated to $V$. This polynomial is known to be $b(s)=[(s+5/4)(s+7/6)(s+1)^2(s+5/6)(s+3/4)]^2$ (see \cite{MR648417}). Moreover, the number of triples $(Q,C,r)$, where $Q$ is a nondegenerate quartic ring and $C$ is a cubic resolvent ring of $Q$, with $|\Delta(Q)|<X$ is bounded by $O(X\log X)$. Indeed, the number of \'etale quartic extensions $K_4$ of $\Q$ is bounded by $O(X\log X)$ by work of Baily \cite{Baily}, Yukie \cite{Yukie}, and Bhargava \cite{dodqf}. A bound of the same strength then follows by counting orders within each quartic algebra by work of Nakagawa \cite{MR1342021}, and controlling the number of possible cubic resolvent rings by work of Bhargava \cite{BHCL3}. It therefore follows that these Shintani zeta functions do not have poles at $s=5/4$ and $s=7/6$, and that the order of the pole at $s=1$ is bounded by $2$.
\end{proof}

Let $N$ be a positive integer, and let $\phi:V(\Z/N\Z)\to \C$ be a $G(\Z/N\Z)$-invariant function. We define the Shintani zeta functions associated to $\phi$ by:
\begin{equation*}
\begin{array}{rcl}
\xi_{i}(\phi;s)&:=&\displaystyle\sum_{x\in G(\Z)\backslash V(\Z)^{(i)}}\frac{\phi(x)}{|\Aut(x)||\Delta(x)|^{s}};
\\[.15in]
\xi_i^{*}(\widehat{\phi};s)&:=&\displaystyle\sum_{x\in G(\Z)\backslash V(\Z)^{\vee,(i)}}\frac{\widehat{\phi}(x)}{|\Aut(x)||\Delta(x)|^{s}}.
\end{array}
\end{equation*}
Note that the lift to $V(\Z)$ of every $G(\Z/N\Z)$-invariant set in $V(\Z/N\Z)$ can be written as a (weighted) union of lattices. (This is true because a $G(\Z/N\Z)$-invariant set is invariant under scaling by $(\Z/N\Z)^\times$, since for any element $\lambda\in (\Z/N\Z)^\times$, the element $\diag(\lambda^{-3},\lambda^{-3}),\diag(\lambda^2,\lambda^2,\lambda^2)$ acts on $V(\Z/N\Z)$ by scaling by $\lambda$.) Therefore, Theorem \ref{th:SZF_v1} implies the following result.
\begin{theorem}\label{th:SZF_v2}
The functions $\xi_i(\phi;s)$ and $\xi_i^{*}(\widehat{\phi};s)$ continue to meromorphic functions on $\C$ with only possible double poles at $1$, $5/6$, and $3/4$. Moreover, they satisfy the function equation
\begin{equation*}
\xi_i(\phi;1-s)=N^{12s}\gamma(s-1)\sum_{j\in\{0,1,2\}}c_{ji}(s)\xi^*_i(\widehat{\phi};s),
\end{equation*}
where $\gamma(s)=\Gamma(s+5/4)^2\Gamma(s+7/6)^2\Gamma(s+1)^4\Gamma(s+5/6)^2\Gamma(s+3/4)^2$ and the $c_{ij}(s)$ are entire functions not depending on $\phi$.
\end{theorem}

The Shintani zeta function can be used to obtain smoothed counts of triples $(Q,C,r)$. Let $\phi$ be a function on $V(\Z/N\Z)$ (and on $V(\Z)$) as above. We define $\phi$ to be a function on the set of all triples $(Q,C,r)$ by setting $\phi(Q,C,r):=\phi(x)$, where $x$ is the $G(\Z)$-orbit corresponding to $(Q,C,r)$ under Bhargava's parametrization result. Let $\psi:\R_{\geq 0}\to\R$ be a smooth and compactly supported function. We define $N_\psi(\phi,X)$ to be
\begin{equation*}
N_\psi(\phi,X):= \sum_{(Q,C,r)}\frac{\phi(Q,C,r)}{|\Aut(Q,C,r)|}\psi\Bigl(\frac{|\Delta(Q)|}{X}\Bigr).
\end{equation*}
Then we have the following result.
\begin{theorem}\label{th:countingbyzeta}
We have
\begin{equation}\label{eq:countingbyzeta}
N_\psi(\phi,X)=\sum_{c\in\{1,5/6,3/4\}}X^c\Bigl(\wt{\psi}(c)r_2(\phi;c)\log X+\wt{\psi}(c)r(\phi;c)+\wt{\psi}'(c)r_2(\phi;c)\Bigr)+O_A(X^{-A}),
\end{equation}
where the expansion of $\xi(\phi;s)$ around $c=1$, $5/6$, and $3/4$ is given by $\xi(\phi;s)=r_2(\phi;c)/(s-c)^2+r(\phi;c)/(s-c)+O(1)$.
\end{theorem}
\begin{proof}
This is a standard result which follows by using Mellin inversion to write
\begin{equation*}
N_\psi(\phi,X)=\frac{1}{2\pi i}\int_{\Re(s)=2}\xi(\phi;s)\wt{\psi}(s)X^s ds,
\end{equation*}
and shifting left to pick up the poles and a super-polynomially small error term.
\end{proof}

In Section 7, we prove the following result regarding these Shintani zeta functions.
\begin{theorem}\label{thm:Shintani_residues}
Let notation be as above, and assume that the support of $\phi$ in $V(\Z)$ is an $S_4$-congruence family. Then $\xi_i(\phi;s)$ has simple poles at $s=1$ and $s=5/6$.
\end{theorem}

Theorem \ref{thm:mainS4rings} is a direct consequence of the above two results, along with a computation of the residues, carried out in \S8.

\section{Introducing the counting tools}

In this section, we set up the notation and preliminary results for the techniques needed for the proof of Theorem \ref{thm:mainS4fields}; specifically, results on multiple zeta functions, counting points with Poisson summation, and the Mellin transform.

\subsection{Signed multiple zeta functions}
We say $f:\Z^n\to\C$ is a periodic function if it is defined by congruence conditions modulo some positive integer. We use $a_1,\dots,a_n$ to denote the coordinates on $\Z^n$. Let $t\in\{\pm 1\}^n$. Writing $\vec{s}$ for $(s_1,\ldots,s_n)$, we define the multiple zeta function $\zeta_{f,t}(s_1,\ldots,s_n)$ associated to $f$ and $t$ by
$$\zeta_{f,t}(\vec{s}):= \sum_{t\cdot \vec{a}\in\Z_{>0}^n} f(\vec{a})\prod_{i=1}^n  |a_i|^{-s_i}.$$ 
Note that since $f$ is periodic, its values are absolutely bounded. Hence $\zeta_{f,t}(\vec{s})$ converges absolutely for $(s_1,\ldots,s_n)\in\C^n$ with $\Re(s_i)>1$ for each $i$.

\begin{definition}
Let $S,T\subset\{1,2,\dots,n\}$ be disjoint subsets. Denote the complement of $S\cup T$ by $R$.  For each element $v\subset \Z^T$ we define $f_{S;T}(v)$ to be the average value of  
$f$ on the set $\{0_S\}\times\{v\}\times\Z^R\subset\Z^n$. Here, by $\{0_S\}\times\{v\}\times\Z^R$, we mean the subset of elements $w\in\Z^n$ such that $a_i(w)=0$ for $i\in S$ and $a_i(w)=a_i(v)$ for $i\in T$. For $t\in\{\pm1\}^T$ and $s_T\in\C^T$, we then define
$$\zeta_{f,t}(S=0;T)(s_T):= \sum_{t\cdot \vec{a}\in \Z_{>0}^T} f_{S,T}(\vec{a})\prod_{i\in T} |a_i|^{-s_i}.$$ 
By convention, we will write $\zeta_{f,t}(T)(s_T)$ for $\zeta_{f,t}(\emptyset=0;T)(s_T)$ when $S$ is empty. When $T=S^c$, we write $\zeta_{f,t}(S=0)(s_T)$ for $\zeta_{f,t}(S=0;S^c)(s_T)$. Note that $\zeta_{f,t}(S=0;\emptyset)$ is simply a complex number, namely, the density of $f$ on the set $\{0_S\}\times\Z^{S^c}$. We will denote this density by $\nu(f|_S)$. If $L$ is a set whose characteristic function $\chi_L$ is periodic we shall write $\zeta_L$ for $\zeta_{\chi_L}$ and $\nu(L|_S)$ for $\nu(\chi_L|_S)$. We define $\zeta_f(S=0;T)$ to be the vector indexed by $(\pm 1)^{T}.$
\end{definition}

We shall need the following results.
\begin{lemma}\label{lem:zeroden}
Let $f:\Z^n\ra\C$ be a periodic function. Then we have
 $$\zeta_{f,t_0}(S=0;T)(s_T) =(-1)^{|S|}\cdot\displaystyle\sum_{t\ra t_0} \zeta_{f,t}(S\cup T)(0_S\times s_T),$$
where $t_0\in\{\pm1\}^T$ and the sum is over every $t\in\{\pm1\}^{S\cup T}$ agreeing with $t_0$ in all the $T$-coordinates.
\end{lemma}
\begin{proof}
We may replace $R$ by $\emptyset$ and $f$ by $f_{S,T}:\Z^{S\cup T}\ra \C$  without changing either side of the equation. We thus assume that $S\cup T = \{1,\dots,n\}$.
Since both sides are linear in $f$, we may also assume $f$ is a product of separate functions in each coordinate, which reduces us to case $T=\emptyset$ and $S=\{1\}$. 

We may thus assume that $f$ is the characteristic function of numbers congruent to $a$ modulo $n$. If $\gcd(a,n)$ is not 1 then we may divide both sides of the equation by $\gcd(a,n)$, reducing to the case when $\gcd(a,n)=1$. If $a=n=1$, then the LHS is the constant $1$ and the RHS is $(-1)(\zeta(0)+\zeta(0)$. The claim then follows from $\zeta(0)=-\frac12$. If $n>1$, then the left hand side is the constant $1/n$. To evaluate the right hand side, note that $\zeta_{f,t}(s)$ is a linear combination of $L$-functions $L(\chi,s)$, where every $\chi$ is an even Dirichlet
character. The statement follows since for even nontrivial characters $\chi$, we have $L(0,\chi)=0$ (as can be seen from the functional equation, for example). Hence the only contribution is from the trivial character, which is weighted by $1/n$ as required.
\end{proof}

Next we have the following lemma.

\begin{lemma}\label{lem:oneres}
If $1\not\in S\cup T$, then 
\begin{equation*}
\zeta_{f,\ol t}(S=0;T)=
{\rm Res}_{s_1=1}\zeta_{f,t}(S=0;T\cup{1}).
\end{equation*}
\end{lemma}
\begin{proof}
We have
\begin{align*}
\Res_{s_1=1}\zeta_{f,t}(S=0;T\cup\{1\})(s_T,s_1)&= \Res_{s_1=1}\sum_{t\cdot \vec{a}\in \Z_{>0}^T}\prod_{i\in T} |a_i|^{-s_i}\sum_{m>0} \mu_{S,T\cup\{1\}}(\vec{a}\times\{m\})|m|^{-s_1}\\
&=\sum_{\vec{a}\in \Z_{\neq0}^T}\prod_{i\in T} |a_i|^{-s_i}\cdot f_{S,T}(\vec{a})\\
&=\zeta_{f,\ol t}(S=0;T)\\
\end{align*}
as desired, where we have used the fact that the natural density of the function $\mu_{S,T\cup\{1\}}(\vec{a}\times \{m\})$ over $m>0$ is $f_{S,T}(\vec{a})$.
\end{proof}

\subsection{Counting points smoothly using Poisson summation}

It is a well-known fact in analysis that smooth point counts (as opposed to sharp counts) can be evaluated with a super-polynomial error term. We prove a version of this which is applicable to our setting.

Let $f$ be a periodic function on $\Z^n$ of modulus $Q$, and let $\cB:\R^n\to \R$ be a fixed smooth function of compact support. Let $g=ud\in \GL_n(\R)$ be a upper triangular matrix with positive diagonal entries, where $u$ is unipotent with entries of size $O(1)$, and $d$ is diagonal. We further assume that the diagonal entries are non-decreasing. We use the standard inner product on $\Z^n$, and recall the transformation formula for the Fourier transform:

$$\widehat{g\cB}(y) = \det(g) (g^{-T} \widehat{\cB})(y).$$
For $1\leq r\leq s\leq n$ we define $g_{r,s}$ to be the induced action on $\R^{s-r}$ thought of the sub-quotient of $\R^n$ where we restrict to the subspace $\langle e_1,\dots,e_{s}\rangle$ and quotient out by the vector space $\langle e_1\dots,e_r\rangle.$ We similarly consider $B_{r,s}$ as a function on $\R^{s-r}$ by first restricting to the subspace $\langle e_1,\dots,e_{s}\rangle$ and then projecting to the quotient space. We write $f_{r,s}$ to mean $f_{R,S}$ in the notation of the previous section, for $R=\{1,\dots,r\}$ and $S=\{s+1,\dots,n\}$. 

We shall often use the following theorem to simplify our smooth counts by projecting away the variables which get `stretched', and restricting to $0$ those variables which get `compressed'.

\begin{theorem}\label{thm:projectlargecoords}
    Let the notation be as above, with $f,B$ fixed and all other parameters varying. Suppose for some parameter $Y$ we have $d_1,\dots,d_r\geq Y$ and $d_{s+1},\dots,d_n\leq Y^{-1}$. Then
\begin{equation*}
\sum_{\ell\in \Z^n}(gB)(\ell)f(\ell) = \prod_{i=1}^r d_i  \sum_{\ell_0\in \Z^{s-r}}(g_{r,s}B_{r,s})(\ell_0)f_{r,s}(\ell_0) + O_A(Y^{-A}).
\end{equation*}
\end{theorem}

\begin{proof}
Since $B$ has compact support, for $Y\gg 1$ we see that $gB(\ell)\neq 0$ implies that $\ell\subset \langle e_1,\dots,e_s\rangle$. Hence we may restrict the sum to that subspace, which we denote $\R^s$. We restrict and apply Poisson summation:

\begin{align*}
    \sum_{\ell\in \Z^n}(gB)(\ell)f(\ell)&=\sum_{\ell\in \Z^s}(gB)(\ell)f(\ell)
    \\&=\det(g)M^{-1}\sum_{\ell^*\in Q^{-1}\Z^s}(g^{-T}\hat{B})(\ell^*)\hat{f}(\ell^*)\\
    &=\det(g)M^{-1}\sum_{\ell^*\in Q^{-1}\Z^s}(\hat{B})(d^Tu^T\ell^*)\hat{f}(\ell^*).\\
\end{align*}
Now since $B$ is smooth of compact support we see have that $\hat{B}(\ell) = O_A(|\ell|^{-A})$. It follows that the contribution of all the terms in the above sum which are not contained in $\langle e^*_{r+1},\dots,e^*_{s}\rangle$ is $O_A(Y^{-A})$.  Next, note that for $\ell^*\in\langle e^*_{r+1},\dots,e^*_{s}\rangle$ we have
$\hat{B}(\ell^*)=\widehat{B_{r,s}}(\ol{\ell^*}),\hat{f}(\ell^*)=\widehat{f_{r,s}}(\ol{\ell^*})$. Hence, we have:

\begin{align*}
    \sum_{\ell\in \Z^n}(gB)(\ell)f(\ell)&=\det(g)M^{-1}\sum_{\ell^*\in Q^{-1}\langle e^*_{r+1},\dots,e^*_{s}\rangle}(\hat{B})(d^Tu^T\ell^*)\hat{f}(\ell^*) + O_A(Y^{-A})\\
    &=\det(g)M^{-1}\sum_{\ell_0^*\in Q^{-1}\Z^{s-r}}(\hat{B_{r,s}})(g^T\ell_0^*)\widehat{f_{r,s}}(\ell_0^*) + O_A(Y^{-A})\\
    &=\det(g)\det(g_{r,s}^{-1})\sum_{\ell_0\in \Z^{s-r}}(g_{r,s}B_{r,s})(\ell_0)f_{r,s}(\ell_0) + O_A(Y^{-A}),\\
\end{align*}
which completes the proof.
\end{proof}

\subsection{The Mellin transform}

We shall heavily employ the Mellin transform, often in many variables at once. To this end, recall that if $f:\R_{> 0}\to\R$ is a function on the positive real line, we define its Mellin transform $\wt{f}(s)$ via $$\wt{f}(s) = \int_{\R_{> 0}}f(x)|x|^{s}d^{\times}x$$ when the integral converges. The integral will converge in a strip $\Re(a)<c<\Re(b)$, and for any such $c$ we have the Mellin inversion formula:
\begin{equation*}
f(x) = \frac{1}{2\pi i}\int_{\Re(s)=c} \wt{f}(s) x^{-s}ds=\int_c \wt{f}(s) x^{-s}ds,
\end{equation*}
where $\int_c$ denotes $\frac{1}{2\pi i}\int_{\Res(s)=c}$.
We also recall the following facts: 
\begin{enumerate}
\item The Mellin transform satisfies the identity $\wt{x^af(x)}(s)=\wt{f}(s+a)$ for $a\in\R$.
\item For differentiable functions $f$, we have $\wt{f'}(s+1)=-s\wt{f}(s)$.
\item If $f$ is a smooth function, then $\wt{f}(s)$ has super polynomial decay along vertical strips, uniformly in any compact region in $(a,b)$. 
\item If $f$ has compact support on $\R_{>0}$, then $\wt{f}$ is entire. 

\item We have the special value $\tilde{f}(1) = \int_{\R_{>0}}f(x)dx= \hat{f}(0)$.
\end{enumerate}
Suppose now that $\lim_{x\to 0}$ $f(x)$ exists. Then the Mellin transform can have poles.
\begin{lemma}
Let $f:\R_{\geq 0}\to\R$ be a smooth function with bounded support. Then $\wt{f}(s)$ has an analytic continuation to the entire complex plane with at most simple poles at $\{0,-1,-2,\ldots\}$. Moreover, we have $\Res_{s=0}\wt{f}(s)=f(0)$.
\end{lemma}
\begin{proof}
It is clear that the integral defining $\wt{f}$ converges for $\Re(s)>0$. In fact, the same is true for all the derivatives $f^{(n)}$ of $f$. Point 2 above gives a functional equation for $\wt{f}$, in terms of $\wt{f^{(n)}}$, allowing analytic continuation to all of $\C$, with at most simple poles at $0$ and the negative integers. Finally, we compute the residue of $\wt{f}$ at $0$ to be
\begin{equation*}
\Res_{s=0}\wt{f}=\lim_{s\to 0}s\wt{f}(s)=-\wt{f'}(1)=-\int_{\R_{>0}}f'(x)dx=f(0),
\end{equation*}
as desired.
\end{proof}

We shall require the following lemmas about the functions $\delta_{\FF_3}(s_1,s_2)$ and $\delta_{\FF_2}(t)$

\begin{lemma}\label{lem:mellinbounddelta}
\begin{itemize}
    \item The Mellin transform $\wt{\delta_{\FF_3}}(v_1,w_1)$ is holomorphic except for poles at $v_1=0,w_1=0$.  The function $\min(1,|v_1|)\min(1,|v_2|)\wt{\delta_{\FF_3}}(v_1,w_1)$ is bounded on any right half-plane.
    \item  The Mellin transform $\wt{\delta_{\FF_2}}(t)$is holomorphic except for poles at $t=0$.  The function $\min(1,|t|)\wt{\delta_{\FF_3}}(t)$ is bounded on any right half-plane.
\end{itemize}
\end{lemma}

\begin{proof}
    We only handle the case of $\FF_3$ as the other case is handled the same way and is easier. Note that we may write $\delta_{\FF_3}(s_1,s_2) = C\delta_{s_1>C_1}\delta_{s_2>C_2}+R_1(s_1)\delta_{s_1>C_1}+R_2(s_2)\delta_{s_2>C_2}+H(s_1,s_2)$ where $R_1,R_2,H$ are compactly supported functions away from $0$ (for either of the co-ordinates) and constants $C,C_1,C_2$ with $C_1,C_2>0$. It is clear that $\tilde{R_i}$ is holomorphic everywhere and bounded on right half-planes. The theorem now follows from the fact that $\tilde{\delta}_{s_
    1>C_1}(s)=\frac{C_1^{-s}}{s}.$
\end{proof}

\subsection{Counting points smoothly using the Mellin transform}\label{sec:poissonexample}

We will use the Mellin transform in order to sum smooth functions over integer points. We will first do an example in the one variable case.
Suppose $\phi:\R_{\geq 0}\to\R$ is a smooth function with compact support, and that we would like to evaluate $\sum_{n\in\Z} \phi(n)$. 
As discussed previously, if $\phi$ is spread out, then this can be done effectively using Poisson summation. However, if $\phi$ is not very well spread out then some subtleties arise which can be difficult to detect with the Fourier transform, and we find the Mellin transform to be much more suitable. We use Mellin inversion to write
\begin{equation*}
\begin{array}{rcl}
\displaystyle\sum_{n\in\Z_{>0}}\phi\Bigl(\frac{n}{X}\Bigr) &=&\displaystyle
\sum_{n>0}\int_{2}\wt{\phi}(s)X^sn^{-s}ds
\\[.1in]&=&\displaystyle
\int_{2} \zeta(s)\wt{\phi}(s)X^sds
\\[.1in]&=&\displaystyle 
X\widehat{\phi}(0)+O(X^{o(1)}).
\end{array}
\end{equation*}

In the situation when $\phi:\R\to\R$ is a smooth function, and we are summing over all the integers, we can do much better. In fact, we can use this method to recover the same super-polynomial decay obtainable using Poisson summation.
Indeed, define the function $\phi_0:\R_{\geq 0}\to \R$ by $\phi_0(x):=\phi(x)+\phi(-x)$. Then the poles of $\wt{\phi_0}$ are only at the negative even integers. This can be seen by the fact that the Taylor expansion of $\phi_0$ is supported on even powers. Then we write
\begin{equation*}
\begin{array}{rcl}
\displaystyle\sum_{n\in\Z}\phi\Bigl(\frac{n}{X}\Bigr) &=&\displaystyle \phi(0)+\sum_{n>0} \phi_0\Bigl(\frac{n}{X}\Bigr)
\\[.1in]&=&\displaystyle
\phi(0) + \sum_{n\in\Z_+}\int_{2} \tilde{\phi_0}(s)X^sn^{-s}ds
\\[.1in]&=&\displaystyle
\phi(0) + \sum_{-A<n\leq 1} X^n \textrm{Res}_{s=n} (\zeta(s)\wt{\phi_0}(s))+ O_A(X^{-A}) 
\\[.1in]&=&\displaystyle
\phi(0)+X\hat{\phi}(0) + \zeta(0)\phi_0(0) + O_A(X^{-A})
\\[.1in]&=&\displaystyle
X\hat{\phi}(0) + O_A(X^{-A}),
\end{array}
\end{equation*}
where we have used the fact that $\Res_{s=0}\wt{\phi_0}(s)=\phi_0(0),\zeta(0)=-\frac{1}{2}$ and $\zeta(n)=0$ when $n$ is a negative even integer.

\medskip

When we have a function $B$ in $n$ variables $x_1,\dots,x_n$, for each $\vec{t}\in(\pm)^n$ we set 
$$\tilde{B}_{\vec{t}}(\vec{s}) = \int_{\R_+} B(t_1x_1,\dots,t_nx_n)\prod_i x_i^{s_i} d^{\times}\vec{x}$$ and $\tilde{B}(\vec{s})$ to be the element in $\C^{(\pm)^n}$ whose co-ordinates are 
$\tilde{B}_{\vec{t}}(\vec{s})$.

\begin{definition}
Given a function $B:\R^n\ra \R$ and disjoint sets $S,T\subset [n]$ we write $B_{S;T}:\R^T\ra\R$ to be the function 
$$B_{S;T}(\vec s_T):=\int_{\R^R}B(\vec 0_S,\vec s_T,\vec s_R)d\vec s_R$$ for $R=(S\cup T)^c$. In other words, we restrict the $S$ co-ordinates to be $0$ and integrate over the remaining co-ordinates except for $T$. We also write
$B_S$ to denote $B_{S;\emptyset}$.
\end{definition}

We shall require the following multi-variable version.

\begin{theorem}\label{thm:mellinsum}
    Given a smooth bounded function $B$ on $\R^n$, and a periodic function $f$ on $\Z^n$ we have
    $$\ds\sum_{\vec{a}\in \Z_{\neq 0}^n}  f(\vec a)B(\vec{a})= \frac{1}{2\pi i}\int _{\Re(\vec{s})=2} \wt{B}(\vec s)\cdot \zeta_f(\vec{s})d\vec s $$  
    and
    $$\ds\sum_{\vec{a}\in \Z^n}  f(\vec a)B(\vec{a})= \frac{1}{2\pi i}\ds\sum_{S\subset [n]}\int_{\Re(\vec{s})=2} \wt{B_S}(\vec s)\cdot \zeta_f(S=0)(\vec{s})d\vec s $$
\end{theorem}

\begin{proof}
The first claim is a direct application of Mellin inversion. The second claim follows from shifting the integral on the right hand side left, and using Lemma \ref{lem:oneres}. 
\end{proof}

Finally, the following lemma will be convenient when we begin shifting our integrals and analyzing poles:

\begin{lemma}\label{lem:symlatticenopole}
Let $f$ be a periodic function on $\Z^n$ which is invariant under negating a subset of the variables, so that $f(x_1,\dots,x_n)=f(|x_1|,\dots,|x_n|)$. Let $B$ be a smooth bounded function on $\R^n$. Then the only polar divisors of 
$\wt{B}(\vec s)\cdot \zeta_f(\vec s)$ occur at $s_i=0$ and $s_i=1$ for some $i\in[n]$.
\end{lemma}

\begin{proof}
Note that a-priori the only potential polar divisors of $\zeta_f(\vec s)$ are at $s_i=1$ and those of $\tilde{B}(\vec{s})$ are at non-positive integers. Therefore, by symmetry, it is sufficient to prove that there are no poles at $s_1=m$ for any negative integer $m$. Let $f_0(a_1,\dots,a_n):=f(-a_1,\dots,a_n)$ and $B_0(x_1,\dots,x_n)=B(-x_1,\dots,x_n)$. For $\epsilon=\pm 1$, we set $B_{\epsilon}=B+\epsilon B_0, f_{\epsilon}=f+\epsilon f_0$ so that 
$$2\wt{B}(\vec s)\cdot \zeta_f(\vec s) =\sum_{\epsilon=\pm 1}\wt{B_\epsilon}(\vec s)\cdot \zeta_{f_{\epsilon}}(\vec s) .$$
Now $B_{+1}$ is symmetric around the origin, and hence its taylor series around $0$ only has even coefficients, and so its Mellin transform only has (simple) poles at non-positive even integers. Moreover, $f_{+1}$ is symmetric around the origin, and so $\zeta_{f_{+1}}(\vec{s})$ can be written as a sum of products of Dirichlet L-functions in one variable, with the Dirichlet characters $\chi$ that occur of the variable $s_1$ satisfying $\chi(-1)=1$. So  $\zeta_{f_{+1}}(\vec{s})$ has zeroes along $s_1=m$ for an even integer $m<0$, since for such $\chi$ we have $L(m,\chi)=0$. 

Likewise, $B_{-1}$ is anti-symmetric around the origin, and hence its taylor series around $0$ only has odd coefficients, and so its Mellin transform only has (simple) poles at non-positive odd integers. Moreover, $f_{+1}$ is anti-symmetric around the origin, and so $\zeta_{f_{+1}}(\vec{s})$ can be written as a sum of products of Dirichlet L-functions in one variable, with the $s_1$ Dirichlet characters that occur satisfying $\chi(-1)=-1$. So  $\zeta_{f_{+1}}(\vec{s})$ has zeroes along $s_1=m$ for odd integers $m<0$, since for such $\chi$ we have $L(m,\chi)=0$.
The lemma follows.
\end{proof}

\section{Setting up the count}

\subsection{The global zeta integral}

Fix $i\in\{0,1,2\}$, and let $\cB:V(\R)^{(i)}\to\R$ be a smooth function with compact support away from the discriminant zero locus.  Let $A_i= (i)!(4-2i)!2^{2i}=\#\Aut(\R^{2i}\times \C^i)$. Let $N$ be a positive integer and let $\phi:V(\Z/N\Z)\to\R$ be a $G(\Z/N\Z)$-invariant function. The {\it global zeta integral} $Z(\cB,\phi;s)$ is defined to be
\begin{equation}\label{eq:global_zeta_int_V}
Z(\cB,\phi;s):=\int_{g\in \FF}(\lambda(g))^{-12s}\Bigl(\sum_{x\in V(\Z)}\phi(x)(g\cB)(x)\Bigr)\nu_G(g).
\end{equation}
If $\phi$ is the characteristic function of a set $L$, we will use $Z_L(\cB,s)$ to denote $Z(\cB,\phi;s)$. The following result, relating global zeta integrals to Shintani zeta functions, is well known (see, for example \cite[\S5]{Kimura_book}).
\begin{proposition}\label{prop_global_zeta_shintani_G}
With notation as above, we have
\begin{equation*}
Z(\cB,\phi;s)=A_i\xi_i(\phi,s)\int_{x\in V(\R)^{(i)}}|\Delta(x)|^{s-1}\cB(x)\nu_V(x).
\end{equation*}
\end{proposition}
It is clear from our assumptions on $\cB$ that the integral above gives an entire function in $s$. In the sequel, we will compute the residues of the poles of Shintani zeta functions at $5/6$ by computing the residues of the poles of the global zeta integrals.

\subsection{Expressing residues in terms of point counts}
Let $i\in \{0,1,2\}$ be fixed, let $N$ be a positive integer, and let $\phi:V(\Z/N\Z)\to\R$ be a $G(\Z/N\Z)$-invariant function. Assume that the lift of $\phi$ to $V(\Z)$ is the characteristic function of a set corresponding to an $S_4$-family of rings. The counting results of Bhargava \cite{dodqf}, determining asymptotics on $G(\Z)$-orbits on the support of $\phi$, with discriminant less than $X$, implies that $\xi_i(\phi,s)$ has a simple pole at $s=1$. Moreover, the residue of $\xi_i(\phi,s)$ at $s=1$ is given by the leading constant of this asymptotic.
Over the next few sections, we will prove that $\xi_i(\phi,s)$ has a simple pole at $s=5/6$, and determine the residue of this pole. Our strategy for doing so is as follows.

We simplify our notation by working with $\xi_{i,L}(s)$, where $L$ is the support of $\phi$. Let $\cB:V(\R)^{(i)}\to\R$ be a smooth $K$-invariant function with compact support away from the discriminant zero locus.
First note that by Proposition \ref{prop_global_zeta_shintani_G}, the poles of $\xi_{i,L}(s)$ are determined by the poles of the global zeta integral $Z_L(\cB,s)$.
Let $\psi:\R_{>0}\to\R_{\geq 0}$ be a smooth function with compact support away from $0$. Since $Z_L(\cB,s)$ is meromorphic away from a possible simple pole at $s=1$ and possible double poles at $s=5/6$ and $s=3/4$, we have
\begin{equation}\label{eq:gen_count_1}
\int_{2}Z_L(\cB,s)\wt{\psi}(12s)X^sds=C_1X+C'X^{5/6}\log X+C_{5/6}X^{5/6}+O(X^{3/4+\epsilon}),
\end{equation}
where $C_1$, $C'$ and $C_{5/6}$ are constants. Moreover, if $C'=0$ (which we will later prove to be the case), then the $Z_L(\cB,s)$ has at most a simple pole at $5/6$. In this situation, the constants $C_1$ and $C_{5/6}$ are given by
\begin{equation}\label{eq:residues_of_GZI}
\begin{array}{rcccl}
C_1&=&\wt{\psi}(12)\Res_{s=1}Z_L(\cB,s)
&=&A_i\displaystyle\wt{\psi}(12)\Res_{s=1}\xi_{i,L}(s)V_1(\cB),
\\[.15in]
C_{5/6}&=&\wt{\psi}(10)\Res_{s=5/6}Z_L(\cB,s)
&=&A_i\displaystyle\wt{\psi}(10)\Res_{s=5/6}\xi_{i,L}(s)V_{5/6}(\cB),
\end{array}
\end{equation}
where we define for a real number $\kappa$, the quantity $V_\kappa(\cB)$ to be
\begin{equation}
\V_\kappa(\cB):=\int_{x\in V(\R)}|\Delta(x)|^{\kappa}\cB(x)\frac{\nu_V(x)}{|\Delta(x)|}.
\end{equation}
Note in particular that $\V_1(\cB)=\Vol(\cB)$.
Therefore, for the purpose of computing the residues of the Shintani zeta functions, it is enough to evaluate the LHS of \eqref{eq:gen_count_1} up to an error of $o(X^{5/6})$.

It will be convenient for us to replace the measure $\omega_G(g)$ with $dg$. This will only change $Z_L(\cB,s)$ by the constant factor $J$ defined in \eqref{eq:J_def}.
We use Mellin inversion and write
\begin{equation}\label{eq:GZI_to_I}
\begin{array}{rcl}
\displaystyle \int_{2}Z_L(\cB,s)\wt{\psi}(12s)X^sds &=&
\displaystyle J\int_{g\in\FF} \Bigl(\sum_{x\in L}(g\cB)(x)\Bigr)
\Bigl(\int_{2}\wt{\psi}(12s)(\lambda(g))^{-12s}X^sds\Bigr)dg
\\[.2in]&=&\displaystyle
\frac{J}{12}\I(\cB;X),
\end{array}
\end{equation}
where
\begin{equation}\label{eq:I-def}
\I(\cB;X):=\int_{g\in\FF}\sum_{x\in L}(g\cB)(x)\psi\Big(\frac{\lambda(g)}{X^{1/12}}\Big)dg.
\end{equation}

Summarizing, we see:

\begin{equation}\label{eq:Iexpansion}
    \frac{J}{12}\I(\cB;X)=XA_i\wt{\psi}(12)\Res_{s=1}\xi_{i,L}(s)V_1(\cB)+ X^{5/6}A_i\wt{\psi}(10)\Res_{s=5/6}\xi_{i,L}V_{5/6}(\cB) + o(X^{5/6})
\end{equation}
The aim then is to evaluate $\I(\cB;X)$ up to an error term of $o(X^{5/6})$.

\subsection{Breaking up the integral}

Our general process of evaluating global zeta integrals outlined in \S2.2 admits some technical simplifications. Specifically, our fundamental domain $\FF$ has three torus parameters, namely, $t$, $s_1$, and $s_2$. However, for our purposes, we will only need to break $\I(\cB;X)$ into two parts, corresponding to whether $t$ is small or $t$ is large. To this end, let $f_0:\R_{\geq 0}\to \R_{\geq 0}$ be a smooth and compactly supported function such that $f_0(x)=1$ for $x\in[0,2]$. Let $f$ denote the function defined by $f(x)=1-f_0(x)$. Throughout this paper, we fix $\delta>0$, which will be assumed to be small. For a real number $X>0$, define the functions $f^X:\R_{\geq 0}\to\R$ and $f_0^X:\R_{\geq 0}\to \R$ by setting
\begin{equation*}
f_0^X(x):=f_0(x/X^\delta);\quad f^X(x):=f(x/X^\delta).
\end{equation*}
For $g_2\in\FF_2$ written as $g_2=(n,t,k)$ in its Iwasawa decomposition, we define $f_0^X(g_2):=f_0^X(t)$ and $f^X(g_2):=f^X(t)$, and for $g=(\lambda,g_2,g_3)\in\FF$, we define $f_0^X(g):=f_0^X(g_2)$ and $f^X(g):=f^X(g_2)$.
We break up $\I(\cB;X)$ as $\I(\cB;X)=\I^{(1)}(\cB;X)+\I^{(2)}(\cB;X)$, where
\begin{equation}\label{eq:I1-def}
    \I^{(1)}(\cB;X)=\int_{g\in\FF}\sum_{x\in L} (g\cB)(x)\psi\Big(\frac{\lambda(g)}{X^{1/12}}\Big)f_0^X(g)dg.
\end{equation}
and
\begin{equation}\label{eq:I1-def}
    \I^{(2)}(\cB;X)=\int_{g\in\FF}\sum_{x\in L} (g\cB)(x)\psi\Big(\frac{\lambda(g)}{X^{1/12}}\Big)f^X(g)dg.
\end{equation}

In the first and second sections of the next part, we will evaluate 
$\I^{(1)}(\cB;X)$ and $\I^{(2)}(\cB;X)$, respectively, by following our general procedure. Our strategy is as follows:
\begin{enumerate}
\item For $\I^{(1)}(\cB;X)$, we will first prove that it is enough to fiber by $a_{11}$ and $b_{11}$ - the contribution from the regions where $a_{12}$ is forced to be small is negligible. Next we will shift the integral and pick up some main terms with a sufficiently small error.
\item For $\I^{(2)}(\cB;X)$, we will first prove that it is enough to fiber by all the coefficients of $A$ and $b_{11}$. Second, we will show that in fact, it is enough to fiber just by $A$ - the difference between the $A$ and $b_{11}$ fiber and the $A$ fiber can be evaluated precisely, upto a sufficiently small error. Finally, we shift the integral obtained from the $A$ fiber, to pick up main terms (in terms of the Shintani zeta function of ternary quadratic forms) upto small error.
\end{enumerate}

We make the following remark before moving on to Part III
\begin{remark}{\rm 
Our assumption that $L\subset V(\Z)$ is an $S_4$-subset, i.e., every $x\in L$ corresponds to a quartic order inside an $S_4$-field is crucial for the above strategy. Indeed, since $L$ contains no points with $a_{11}=b_{11}=0$ or $\det(A)=0$, only $g\in\FF$ with $w_{b_{11}}(g)\gg 1$, $w_{a_{13}}(g)\gg 1$, and $w_{a_{22}}(g)\gg 1$, give nonzero contributions to $\I^{(1)}(\cB;X)$ and $\I^{(2)}(\cB;X)$. Thus we can get away with this fairly minimal amount of fibering. Without this assumption on $L$, we would have to count points ``deeper in the cusp'', and would need to carry our a more complicated fibering procedure.
}\end{remark}

\part{Computing the Residues}

\section{The region where $t$ is small - evaluating $\I^{(1)}(\cB;X)$}

In this section, we evaluate the value of $\I^{(1)}(\cB;X)$. Our main result is as follows:

\begin{proposition}\label{prop:tsmallpowerexpansion}
We have
$$\I^{(1)}(\cB;X)=X\Vol_{dg_3}(\FF_3)\wt{\psi}(12)\nu(L)\Vol(\cB)\int_{g_2\in\FF_2}f_0^X(g_2)dg_2+cX^{21/24}+O_{\epsilon}(X^{3/4+\epsilon}),$$    
for some constant $c$.
\end{proposition}

\subsection{Fibering the sum}

Fix some $\theta>0$, which will be taken to be small compared to $\delta$. Let $\FF^\bad$ denote the set of $g\in\FF$ with $\lambda(g)\asymp X^{1/12}$ and $f_0^X(g)>0$, such that $w_{a_{12}}(g)\ll X^\theta$ and $w_{b_{11}}(g)\gg 1$.
Then we have the following lemma
\begin{lemma}\label{lem:t_small_regions}
We have
\begin{equation}\label{eq:t_small_regions}
\int_{g\in\FF^\bad}\sum_{x\in L} (g\cB)(x)\psi\Big(\frac{\lambda(g)}{X^{1/12}}\Big)f_0^X(g_2)dg\ll X^{3/4+O(\delta+\theta)}.
\end{equation}
\end{lemma}
\begin{proof}
Let $\lambda\cdot n(t,s_1,s_2)k=g\in\FF^\bad$, with $\lambda\asymp X^{1/12}$, and note that the condition $f_0^X(g)>0$ implies that $w_g(a_{11})$ and $w_g(b_{11})$ are within a factor of $X^{2\delta}$ of each other for all $i,j$. The conditions
$w_g(a_{12})=\lambda t^{-1}s_1^{-1}s_2^{-2}\ll X^\theta$ and $w_g(b_{11})=\lambda t s_1^{-4}s_2^{-2}\gg 1$ imply
\begin{equation*}
X^\theta\gg \frac{w(a_{12})}{w(b_{11})}=t^{-2}s_1^{3}\gg s_1^3X^{-2\delta},
\end{equation*}
which yields $s_1\ll X^{(\theta+2\delta)/3}$. The condition $w(a_{12})\ll X^\theta$ now yields $s_2\gg \lambda^{1/2} X^{-(4\theta+5\delta)/3}$.
Therefore, for $g\in\FF^\bad$, we have $w_g(b_{ij})\gg 1$ (since $b_{11}$ is minimal among them), $w_g(a_{13}),w_g(a_{23}),w_g(a_{33})\gg 1$ (since they have positive powers of $s_2$ in them), and $w(a_{11}),w(a_{12}),w(a_{22})\gg X^{-2\delta}$ (since $a_{11}$ is minimal among them). This implies that every projection of the set $g\cB$ is bounded by $O(\lambda^{12}X^{6\delta})$.
Therefore, the LHS in the equation \eqref{eq:t_small_regions} is $\ll$
\begin{equation*}
X^{6\delta}\int_{\lambda\asymp X^{1/12},s_2\gg \lambda^{1/2}X^{-(4\theta+5\delta)/3}}\lambda^{12}s_2^{-6}d^\times s_2d^\times\lambda
\ll X^{3/4+16\delta+8\theta},
\end{equation*}
as necessary.
\end{proof}

We derive two consequences of the above lemma. The first implies that to estimate $\I^{(1)}(\cB;X)$, up to an error term of $o(X^{5/6})$, it is only necessary to fiber by the coefficients $a_{11}$ and $b_{11}$.

\begin{corollary}
Set $S:=\{a_{11},b_{11}\}$. Then we have
\begin{equation*}
\begin{array}{ll}
&\displaystyle
\left |\int_{g\in\FF}\sum_{x\in L} (g\cB)(x)\psi\Big(\frac{\lambda(g)}{X^{1/12}}\Big)f_0^X(g_2)dg
-\int_{g\in\FF}\sum_{a_{11},b_{11}}\nu(L|_{\{a_{11},b_{11}\}})(g\cB)_S(a_{11},b_{11})\psi\Big(\frac{\lambda(g)}{X^{1/12}}\Big)f_0^X(g_2)dg\right |
\\[.2in]
&\ll  X^{3/4+O(\delta)}.
\end{array}
\end{equation*}
\end{corollary}
\begin{proof}
Write $\{g\in \FF:f_0^X>0,\lambda(g)\asymp X^{1/12}\}=\FF^\bad\cup\FF^\good$ as the union of two disjoint sets.
For $g\in\FF^\good$, the difference between the two sums
\begin{equation*}
\sum_{x\in L}(g\cB)(x)
-\sum_{a_{11},b_{11}}(g\cB)_{a_{11},b_{11}}(a_{11},b_{11})
\end{equation*}
is super-polynomially small by Theorem \ref{thm:projectlargecoords}, since the range $w_g(\alpha)$ is larger than $X^\theta$ for every coefficient $\alpha\not\in\{a_{11},b_{11}\}$.
For $g\in\FF^\bad$, the necessary bound on the LHS is a direct consequence of Lemma \ref{lem:t_small_regions} together with the choice $\theta=O(\delta)$. For the RHS the same analysis as in  Lemma \ref{lem:t_small_regions} works verbatim. 
\end{proof}

As a consequence of the above lemma, we have

\begin{equation}\label{eq:cor_t_small_1}
\begin{array}{ll}
\I^{(1)}(\cB;X)=&\displaystyle
\int_{\lambda>0}\int_{s_1,s_2}\int_{g_2\in\FF_2}\sum_{a_{11},b_{11}}\nu(L|_{\{a_{11},b_{11}\}})(\lambda\cdot (s_1,s_2)g_2)\cB)_S(a_{11},b_{11})
\\[.2in]
&\displaystyle\quad\quad \psi\Big(\frac{\lambda}{X^{1/12}}\Big)f_0^X(g_2)\delta_{\FF_3}(s_1,s_2)\frac{d^\times sd^\times\lambda dg_2}{s_1^6s_2^6} + O(X^{3/4 + O(\delta)}).
\end{array}
\end{equation}

\subsection{Shifting the integral}

Recall that we set $S=\{a_{11},b_{11}\}$. For $g=\lambda\cdot (s_1,s_2)g_2\in\FF$, we perform a Mellin transform to write
\begin{equation*}
\sum_{\substack{a_{11},b_{11}\\a_{11}b_{11}\neq 0}}\nu(L|_{\{a_{11},b_{11}\}})(g\cB)_{\emptyset;S}(a_{11},b_{11})
=\int_{\substack{\Re(v_{11})=1+\epsilon\\\Re(w_{11})=1+\epsilon}}\wt{g\cB}_{\emptyset;S}(v_{11},w_{11})\vec{\zeta}_L(v_{11},w_{11})dv_{11}dw_{11}.
\end{equation*}
Now note that that the action of $s_1$, $s_2$, and $\lambda$ on $(g\cB)_S$ is quite simple, namely, we have 
\begin{equation*}
(\lambda(s_1,s_2)g_2\cB)_S(a_{11},b_{11})=\lambda^{10}s_1^8s_2^4(g_2\cB)_S(\lambda^{-1}s_1^{4}s_2^{2}a_{11},\lambda^{-1}s_1^{4}s_2^{2}b_{11}).
\end{equation*}
Hence for fixed $g_2$ and $\lambda$, we can use Theorem \ref{thm:mellinsum} and integrate over $s_1$ and $s_2$, obtaining:
\begin{equation}\label{eq:t_small_shift_main}
\begin{array}{ll}
&\displaystyle \int_{s_1,s_2}\sum_{\substack{a_{11},b_{11}\\a_{11}b_{11}\neq 0}}\nu(L|_{\{a_{11},b_{11}\}})(g\cB)_{\emptyset;S}(a_{11},b_{11})
\delta_{\FF_3}(s_1,s_2)s_1^{-6}s_2^{-6}d^\times s
\\[.2in]
= & \displaystyle\int_{s}\int_{1+\epsilon} \wt{(g_2\cB)_{\emptyset;S}}
(v_{11},w_{11})\vec{\zeta}_L(v_{11},w_{11})s_1^{2-4v_{11}-4w_{11}}s_2^{-2-2v_{11}-2w_{11}}\lambda^{10+v_{11}+w_{11}}
\displaystyle \delta_{\FF_3}(s_1,s_2) dv_{11}dw_{11}d^\times s
\\[.2in]
= &\displaystyle
\int_{1+\epsilon}\widetilde{(g_2B)_{\emptyset;S}}(v_{11},w_{11})\cdot\vec{\zeta}_L(v_{11},w_{11})\tilde{\delta}_{\FF_3}(2-4v_{11}-4w_{11},-2-2v_{11}-2w_{11})\lambda^{10+v_{11}+w_{11}}dv_{11}dw_{11}.
\end{array}
\end{equation}
Similarly, the top line of the above displayed equation, with the condition $a_{11}b_{11}\neq 0$ replaced with $a_{11}=0,b_{11}\neq 0$, and $a_{11}\neq 0,b_{11}= 0$ respectively give the contributions
\begin{equation}\label{eq:t_small_a11b11_0_contribs}
\begin{array}{rcl}
&\displaystyle\int_{1+\epsilon}\widetilde{(g_2B)}_{a_{11};b_{11}}(w_{11})\cdot\vec{\zeta}_L(a_{11}=0,w_{11})\wt{\delta}_{\FF_3}(2-4w_{11},-2-2w_{11})\lambda^{10+w_{11}}dw_{11},
\\[.2in]
&\displaystyle\int_{1+\epsilon}\widetilde{(g_2B)}_{b_{11};a_{11}}(v_{11})\cdot\vec{\zeta}_L(b_{11}=0,v_{11})\wt{\delta}_{\FF_3}(2-4v_{11},-2-2v_{11})\lambda^{10+v_{11}}dv_{11},
\end{array}
\end{equation}
respectively. Note that $a_{11}=b_{11}=0$ does not contribute to the sum in the RHS of \eqref{eq:cor_t_small_1}, since we have assumed that $L$ contains no irreducible elements.

We consider the final equation in \eqref{eq:t_small_shift_main}. The polar divisors of the function being integrated are at $v_{11}=1$, $w_{11}=1$ (from $\vec{\zeta}_L$), at $v_{11}, w_{11}\in2\Z_{\leq 0}$ (from $\widetilde{(g_2B)}_{\emptyset;S}$, using Lemma \ref{lem:symlatticenopole}), and at $v_{11}+w_{11}\in\{1/2,-1\}$ (from $\wt{\delta}_{\FF_3}$). We start our integral at $(\Re(v_{11}),\Re(w_{11}))=(1+\epsilon,1+\epsilon)$, and shift to $(1+\epsilon,-2+\epsilon)$. While doing so, we pick up poles at $w_{11}=1$, $w_{11}=0$, and $w_{11}=1/2-v_{11}$. Once $w_{11}$ is at $-2+\epsilon$, the integral gives a contribution of size $O(X^{3/4+\epsilon})$ since the line of integration over $v_{11}$ starts at real part $1+\epsilon$. Thus the final line of \eqref{eq:t_small_shift_main} can be written, up to an error of $O(X^{3/4+\epsilon})$, as
\begin{equation}\label{eq_t_small_temp1}
\begin{array}{ll}
&\displaystyle\int_{1+\epsilon}\wt{g_2\cB}_{\emptyset;a_{11}}(v_{11})\vec{\zeta}_L(v_{11})\wt{\delta}_{\FF_3}(-2-4v_{11},-4-2v_{11})\lambda^{11+v_{11}}dv_{11}
\\[.2in]
-&\displaystyle\int_{1+\epsilon}\widetilde{g_2\cB}_{b_{11};a_{11}}(v_{11})\cdot\vec{\zeta}_L(b_{11}=0;v_{11})\wt{\delta}_{\FF_3}(2-4v_{11},-2-2v_{11})\lambda^{10+v_{11}}
\\[.2in]
+&\displaystyle \lambda^{21/2}\Res_{s=0}\wt{\delta}_{\FF_3}(s,-3)\int_{1+\epsilon}\widetilde{g_2\cB}_{\emptyset;S}(v_{11},1/2-v_{11})\vec{\zeta}_L(v_{11},1/2-v_{11})dv_{11}.
\end{array}
\end{equation}
The second term above cancels exactly with the second line of \eqref{eq:t_small_a11b11_0_contribs}. We ignore the third term for now. Take the first term, shift the integral over $v_{11}$ to $\Re(v_{11})=-2+\epsilon$, picking up poles at $1$, $0$, and $-1/2$, and obtaining up to an error term of $O(X^{3/4+o(1)})$
\begin{equation}\label{eq_t_small_temp2}
\begin{array}{ll}
&\displaystyle\lambda^{12}\wt{g_2\cB}_{\emptyset;\emptyset}\Res_{v_{11}=1}\vec{\zeta}(v_{11})\wt{\delta}_{\FF_3}(-6,-6)-\lambda^{11}\wt{g_2\cB}_{a_{11}}\vec{\zeta}_L(a_{11}=0;\emptyset)\wt{\delta}_{\FF_3}(-2,-4)
\\[.1in]
&\displaystyle +\lambda^{21/2}\wt{g_2\cB}(-1/2)\vec{\zeta}_L(-1/2)\Res_{s=0}\wt{\delta}_{\FF_3}(s,-3).
\end{array}
\end{equation}
Finally, we take the first line of \eqref{eq:t_small_a11b11_0_contribs} and shift the integral over $w_{11}$ to $\Re(w_{11})=-1+\epsilon$, picking up poles at $1$, $1/2$, and $0$, and obtaining up to an error term of $O(X^{3/4+o(1)})$
\begin{equation}\label{eq_t_small_temp3}
\begin{array}{ll}
&\displaystyle
\lambda^{11}\wt{g_2\cB}_{a_{11};b_{11}}(1)\Res_{w_{11}=1}\vec{\zeta}_L(a_{11}=0,w_{11})\wt{\delta}_{\FF_3}(-2,-4)
\\[.1in]&\displaystyle
+\lambda^{21/2}\Res_{s=0}(s,-3)\wt{\delta}_{\FF_3}\cdot\wt{g_2\cB}_{a_{11};b_{11}}(1/2)\vec{\zeta}_L(a_{11}=0;1/2)
\\[.1in]&\displaystyle
-\lambda^{10}\wt{g_2\cB}_{\{a_{11},b_{11}\}}(0)\vec{\zeta}_L(\{a_{11},b_{11}\}=0)\wt{\delta}_{\FF_3}(2,-2).
\end{array}
\end{equation}

The final line of \eqref{eq:t_small_shift_main} added to the sum of the two terms in \eqref{eq:t_small_a11b11_0_contribs} is thus equal, up to an error of $O(X^{3/4+o(1)})$, to the sum of \eqref{eq_t_small_temp2}, \eqref{eq_t_small_temp3}, and the third summand of \eqref{eq_t_small_temp1}.
Note that the first term of $\eqref{eq_t_small_temp2}$ is equal to $\nu(L)\Vol(\cB)\Vol(\FF_3)\lambda^{12}$. Meanwhile, the second term of \eqref{eq_t_small_temp2} exactly cancels the first term of~\eqref{eq_t_small_temp3} by Lemma \ref{lem:oneres}. Also, the final term of \eqref{eq_t_small_temp3} is $0$ since $L$ has no points with $a_{11}=b_{11}=0$. Therefore, combining the results of this subsection with \eqref{eq:cor_t_small_1}, we obtain
\begin{equation}\label{eq:I1uptoerror}
\I^{(1)}(\cB;X)=X\Vol_{dg_3}(\FF_3)\wt{\psi}(12)\nu(L)\Vol(\cB)\int_{g_2\in\FF_2}f_0^X(g_2)dg_2 + E_{21/2} +O(X^{3/4+\epsilon})
\end{equation}

where 
$$E_{21/2}=X^{7/8}\wt{\psi}(21/2)\cdot \left(c_1+c_2\int_{g_2\in\FF_2}\int_{1+\epsilon}\widetilde{g_2B}_{\emptyset;S}(v_{11},1/2-v_{11})\vec{\zeta}_L(v_{11},1/2-v_{11})f_0^X(g_2)dv_{11}dg\right)$$

We will show that the $E_{21/2}$ terms will vanish for `formal reasons' due to a lack of a pole for the Shintani zeta function. To do this, we prove the following:

\begin{proposition}\label{prop:formalpowerexpansion}

We have \begin{align*}
&\int_{g_2\in\FF_2}\int_{1+\epsilon}\widetilde{g_2\cB}_{\emptyset;S}(v_{11},1/2-v_{11})\vec{\zeta}_L(v_{11},1/2-v_{11})f_0^X(g_2)dv_{11}dg_2\\ &= \int_{g_2\in\FF_2}\int_{1+\epsilon}\widetilde{g_2\cB}_{\emptyset;S}(v_{11},1/2-v_{11})\vec{\zeta}_L(v_{11},1/2-v_{11})\delta_{\FF_2}(g_2)dv_{11}dg_2 +O_A(X^{-A}).\\
\end{align*}    
\end{proposition}

\begin{proof}
We write $n_u=\left(\begin{smallmatrix}
    1 & 0\\u & 1
\end{smallmatrix}\right), d_t=\left(\begin{smallmatrix}
    t^{-1} & 0\\0 & t
\end{smallmatrix}\right)$. We define $\cB_n:=\left(\begin{smallmatrix}
    1 & 0\\n & 1
\end{smallmatrix}\right)\cdot\cB_{\infty,S}$. Note that $g_2=n_ud_t=d_tn_{ut^{-2}}$, and define $\cB_u:=n_u\cdot\cB_{\infty,S}$ so that 

$$\wt{g_2\cB}_{\infty,S}(v_{11},w_{11}) =\wt{\cB}_{ut^{-2}}(v_{11},w_{11})t^{1/2-2v_{11}}.$$
 Since we have uniform super-polynomial decay along vertical strips for $\wt{\cB}_u$ the integrals in the statement converge uniformly as long as the exponent of $t$ is negative, so as long as $v_{11}>-\frac34$. Hence it is enough to obtain a super polynomially small bound on
$$\int_{t>0}\int_{|u|<1/2}\int_{1+\epsilon}t^{-3/2-2v_{11}}\wt{\cB}_{ut^{-2}}(v_{11},1/2-v_{11})\vec{\zeta}_L(v_{11},1/2-v_{11})f(tX^{-\delta})dv_{11}d^\times tdu.$$
Now, we simply shift  $\Re v_{11}$ to $A$, obtaining an error of $O_A(X^{-\delta(3/2+2A})$, which completes the proof. 
\end{proof}

Putting together \eqref{eq:I1uptoerror} and Proposition \ref{prop:formalpowerexpansion} completes the proof of Proposition \ref{prop:tsmallpowerexpansion}.

\section{The region where $t$ is big - evaluating $\I^{(2)}(\cB;X)$}
In this section, we evaluate the value of $\I^{(2)}(\cB;X)$.

\begin{proposition}\label{prop:section6_t_big_main}
We have
\begin{equation*}
\begin{array}{rcl}
\I^{(2)}(\cB;X)&=&
\displaystyle C\wt{\psi}(12)\Res_{s=2}Z((\cB)_{\emptyset;{A}},\nu(L_A);s)\wt{f^X}(-2)X+C\wt{\psi}(10)Z\bigl((\cB)_{\emptyset;\{A\}},\nu(L_A);4/3\bigr)X^{5/6}
\\[.15in]&&\displaystyle
+c_0X^{21/24+5\delta/2}+c_1X^{21/24-3\delta}+c_2X^{21/24-3\delta/2}+O(X^{4/5+\epsilon}),
\end{array}
\end{equation*}
for some constants $C$, $c_0$, $c_1$, and $c_2$, where $Z((\cB)_{\emptyset;{A}}\nu(L_A);s)$ is the global zeta integral associated to the space of ternary quadratic forms (see \eqref{eq:tq_SZF}).
\end{proposition}

\subsection{Fibering the sum}

Fix some $\theta>0$ which we shall take to be small compared to $\delta$. Let $\FF^{\bad}$
denote the set of $g\in\FF$ such that $w_{b_{12}}(g)\leq X^{\theta}$.

\begin{lemma}\label{lem:t_big_regions}
We have
\begin{equation}
\int_{g\in\FF^\bad}\sum_{x\in L} (g\cB)(x)\psi\Big(\frac{\lambda(g)}{X^{1/12}}\Big)f^X(g_2)dg\ll X^{3/4+\theta}.
\end{equation}
\end{lemma}
\begin{proof}
Let $g=n(s_1,s_2,t)k$ be a fixed element in $\FF'$. The factor $f^X(g_2)$ in the integral implies that we can assume $t\gg X^\delta$. Let $x\in L$ such that $(gB)x\neq 0$.
Then $w_{a_{12}}(g)\ll O(X^{\theta-2\delta})$, and by assuming that $\theta<2\delta$, we see that $a_{12}(x)=0$ and hence $a_{11}(x)=0$. Now for our choice of $L$ there are no points $x\in L$ with $a_{11}(x)=a_{12}(x)=a_{22}(x)=0$, and the weights of all other co-ordinates are $\gg 1$. Hence, we may estimate the number of points (all of which have $a_{11}=a_{12}=0$), as being $\ll$ the volume $\lambda^{10}s_1^5s_2^4t^2$ of the $a_{11}=a_{12}=0$ slice. Finally, note that the condition on $w_{b_{12}}(g)$ implies that $\lambda t X^{-\theta}\ll s_1s_2^2$. We may thus bound the integral in the statement by
$$\int_{\lambda\asymp X^{1/12}}\int_{\substack{t,s_1,s_2\gg 1\\ s_1s_2^2\gg\lambda t X^{\theta}}} \lambda^{10}s_1^{-1}s_2^{-2}2d^{\times}sd^{\times}t \ll X^{\theta}\int_{\lambda\asymp X^{1/12}} \int_{t,s_1,s_2\gg 1} \lambda^9d^{\times}sd^{\times}t \ll X^{3/4+\theta}, $$ as desired.
\end{proof}

We now take $T=\{a_{11},a_{12},a_{22},a_{13},a_{23},a_{33},b_{11}\}$ which we write as $\{A,b_{11}\}$ for brevity. Exactly as in the previous section, we obtain that

\begin{equation*}
\begin{array}{rcl}
\displaystyle
\I^{(2)}(\cB;X)
&=&\displaystyle
\int_{\lambda>0}\int_{s_1,s_2}\int_{g_2\in\FF_2}\sum_{A,b_{11}}\nu(L|_{\{A,b_{11}\}})(\lambda\cdot (s_1,s_2)g_2)\cB)_{\emptyset;T}(A,b_{11})
\\[.2in]
&&\displaystyle\quad\quad \psi\Big(\frac{\lambda}{X^{1/12}}\Big)f^X(g_2)\delta_{\FF_3}(s_1,s_2)\frac{d^\times sd^\times\lambda dg_2}{s_1^6s_2^6} + O(X^{3/4 + O(\delta)}).
\end{array}
\end{equation*}
Moreover, we now invoke Theorem \ref{thm:projectlargecoords} to obtain
\begin{equation}\label{eq:cor_t_big_1}
\I^{(2)}(\cB;X)=\int_{g\in\FF}\sum_{A,b_{11}}\nu(L|_{\{A,b_{11}\}})(g\cB)_{\emptyset;\{A,b_{11}\}}(A,b_{11})
\psi\Big(\frac{\lambda(g)}{X^{1/12}}\Big)f^X(g)dg + O(X^{3/4 + O(\delta)}).
\end{equation}

We may further simplify the above integral by noting that the action of the unipotent of $\SL_2$ does not interfere with our count.
Essentially, given $s_1$, $s_2$, and $t\gg X^\delta$, either the range of $a_{11}$ is $\gg 0$, in which case the range of $b_{11}$ is large enough that we may just estimate by the volume. Or the range of $a_{11}$ is forced to be $0$, in which case the $\SL_2$-unipotent doesn't change the value of $b_{11}$, as it is only adding multiples of $a_{11}$ to it. In either case, the unipotent element in the $\SL_2$ element is irrelevant. In the next subsection, we make this observation precise.

\subsection{Simplifying the fibered sum}

In this subsection, our goal is to prove the following result:
\begin{proposition}\label{prop:t_large_Ab11_to_A}
We have
\begin{equation*}
\begin{array}{rcl}
\displaystyle\I^{(2)}(\cB;X)
&=&\displaystyle
\int_{g\in\FF}\sum_{A}\nu(L_{A})(g\cB)_{\emptyset;\{A\}}(A)\psi\Big(\frac{\lambda(g)}{X^{1/12}}\Big)f^X(g)dg
+c_1X^{21/24-3\delta}+c_2X^{21/24-3\delta/2}
\\[.2in]&&\displaystyle
+O(X^{4/5+o(1)}).
\end{array}
\end{equation*}
\end{proposition}

\noindent We will prove Proposition \ref{prop:t_large_Ab11_to_A} in two steps. For $g\in\FF$, consider the differences
\begin{equation*}
\begin{array}{rcl}
E^{(1)}_g&:=&\displaystyle\sum_{A,b_{11}}\nu(L|_{A,b_{11}})(g\cB)_{\emptyset;\{A,b_{11}\}}(A,b_{11})-\sum_{A}\nu(L_{A})(g\cB)_{\emptyset;\{A\}}(A);
\\[.15in]
E^{(2)}_g&:=&\displaystyle\sum_{b_{11}}\nu(L|_{a_{11}=0;b_{11}})(g\cB)_{\{a_{11}\};\{b_{11}\}}(b_{11})-\nu(L_{a_{11}=0})(g\cB)_{\{a_{11}\}}.
\end{array}
\end{equation*}
First, we will first prove that the integral over $g\in\FF$ of
$E_g^{(1)}-E_g^{(2)}$ is small. Second, we will evaluate the integral
of $E_g^{(2)}$ to complete the proof.

\begin{lemma}\label{lem:Ab11_to_A_fiber}
We have
\begin{equation*}
\int_{g\in\FF}\bigl(E_g^{(1)}-E_g^{(2)}\bigr)
\psi\Big(\frac{\lambda(g)}{X^{1/12}}\Big)f^X(g)dg
=O(X^{4/5+o(1)}).
\end{equation*}
\end{lemma}
\begin{proof}
  Let $\theta<\delta$ be a small constant to be picked later. Note that if $g\in\FF_3$ with $w_g(b_{11})\gg X^\theta$, then both $E_g^{(1)}$ and $E_g^{(2)}$ are superpolynomially small by Theorem \ref{thm:projectlargecoords}. Hence, we may restrict the integral above to $g\in\FF^\bad$, where $\FF^\bad$ consists of $g\in\FF$ with $f^X(g)>0$ and $w_g(b_{11})\ll X^\theta$.


Next, we note that $E_g^{(1)}-E_g^{(2)}=F_g^{(1)}-F_g^{(2)}$, where 
\begin{equation*}
\begin{array}{rcl}
F^{(1)}_g&:=&\displaystyle\sum_{A,b_{11}}\nu(L|_{A,b_{11}})(g\cB)_{\emptyset;\{A,b_{11}\}}(A,b_{11})-\sum_{b_{11}}\nu(L|_{a_{11}=0;b_{11}})(g\cB)_{\{a_{11}\};\{b_{11}\}}(b_{11});
\\[.15in]
F^{(2)}_g&:=&\displaystyle\sum_{A}\nu(L_{A})(g\cB)_{\emptyset;\{A\}}(A)-\nu(L_{a_{11}=0})(g\cB)_{\{a_{11}\}}.
\end{array}
\end{equation*}
Define $\FF^{\bad,1}$ to be the set of $g\in\FF^{\bad}$ satisfying $w_g(a_{12})\ll X^\theta$. We claim that for $g\in\FF^\bad\backslash\FF^{\bad,1}$, we have $F_g^{(1)}-F_g^{(2)}$ is superpolynomially small: 
Indeed, when $w_g(a_{12})>X^\theta$, we have upto superpolynomially small error,
\begin{equation*}
F_g^{(1)}-F_g^{(2)}=\sum_{a_{11}\neq 0,b_{11}}
\nu(L|_{a_{11},b_{11}})(g\cB)_{\emptyset;\{a_{11},b_{11}\}}
-\sum_{a_{11}\neq 0}
\nu(L|_{a_{11}})(g\cB)_{\emptyset;\{a_{11}\}}.
\end{equation*}
Since $a_{11}\neq 0$ and $t\gg X^\delta$ forces the range of $b_{11}$ to be large, the above is superpolynomially small.
Thus, up to superpolynomially small error, we may restrict the integral over $g\in\FF$ in the displayed equation of the lemma to $g\in\FF^{\bad,1}$.

We now show that the integral over $g\in\FF^{\bad,1}$ of all four
terms constituting $E_g^{(1)}$ and $E_g^{(2)}$ are small.
Let us begin
with the first term of $E_g^{(1)}$: Note that if any of $w_g(b_{11})$, $w_g(a_{12})$, or $w_g(a_{22})$ are less than a sufficiently small positive constant $c$, then
$\nu(L_{A,b_{11}}))(g\cB_{\phi;\{A,b_{11}\}})$ is $0$. This is because $L$ contains no elements with $\det(A)=0$ or $(a_{11},b_{11})=(0,0)$.
Note also that since $g\in\FF^{\bad,1}$ is such that $w_g(b_{11})<X^\theta$ and $f^X(g)>0$, it follows that $g\cB$ is nonzero only on integral pairs $(A,B)$ with $a_{11}=0$. Hence we have
\begin{equation}\label{eq:badbound1}
\begin{array}{rcl}
  \displaystyle  \int_{\substack{g\in\FF^{\bad,1}}}\sum_{A,b_{11}}\nu(L|_{A,b_{11}})(g\cB)_{\emptyset;\{A,b_{11}\}}\psi\Bigl(
\frac{\lambda(g)}{X^{1/12}}
\Bigr)dg&\ll& \displaystyle\int_{\substack{\FF^{\bad,1}\\\lambda(g)\asymp X^{1/12}}}\max(1,w_g(a_{12}))\prod_{\alpha\not\in\{a_{11},a_{12}\}}w_g(\alpha)
\\[.2in]&\ll&\displaystyle
X^{5/6+\theta}\int_{\substack{(t,s_1,s_2)\in\FF^{\bad,1}}}s_1^{-1}s_2^{-2}d^\times td^\times s_1 d^\times s_2.
\end{array}
\end{equation}
In fact, the same bound is true for each of the four terms
constituting $E_g^{(1)}$ and $E_g^{(2)}$. For example, for the second
term of $E_g^{(1)}$, though we can no longer assume that
$w_g(b_{11})>c$, the estimate remains the same since only the volume of the $b_{11}$-projection is being integrated (and not the number of integral choices for $b_{11}$).

The conditions $w_g(b_{11})\ll X^\theta$ and $w_g(a_{12})\ll X^\theta$
respectively imply that we have
\begin{equation*}
Y:=\frac{s_1^4s_2^2 X^\theta}{t\lambda}\gg 1;\quad
Y':=\frac{ts_1s_2^2 X^\theta}{\lambda}\gg 1.
\end{equation*}
Consider the integral in the final line of
\eqref{eq:badbound1}. Multiplying the integrand
by $Y^{1/5+\epsilon}Y'^{1/5-5\epsilon}$, yields
\begin{equation*}
  X^{5/6+\theta}\int_{\substack{(t,s_1,s_2)\in\FF^{\bad,1}}}s_1^{-1}s_2^{-2}d^\times td^\times s_1 d^\times s_2
\ll X^{4/5+O(\theta+o(1))}.
\end{equation*}
Since $\theta$ and $\epsilon$ can be taken to be arbitrarily small, the result follows.
\end{proof}

Combining \eqref{eq:cor_t_big_1} and Lemma \ref{lem:Ab11_to_A_fiber}, we obtain
\begin{equation}\label{eq:temp_t_big_11}
\begin{array}{rcl}
\displaystyle \I^{(2)}(\cB;X)&=&
\displaystyle\int_{g\in\FF}\sum_{A}\nu(L_{A})(g\cB)_{\emptyset;\{A\}}(A)
\psi\Big(\frac{\lambda(g)}{X^{1/12}}\Big)f^X(g)dg
+\displaystyle\int_{g\in\FF} E_g^{(2)}
\psi\Big(\frac{\lambda(g)}{X^{1/12}}\Big)f^X(g)dg
\\[.2in]&&+\displaystyle
O(X^{4/5+o(1)}).
\end{array}
\end{equation}

Next we evaluate the integral of $E_g^{(2)}$ 
\begin{lemma}\label{lem:E2}
We have
\begin{equation*}
\int_{g\in\FF} E_g^{(2)}
\psi\Big(\frac{\lambda(g)}{X^{1/12}}\Big)f^X(g)dg=c_1X^{21/24-3\delta}+O(X^{3/4}),
\end{equation*}
for some constant $c_1$.
\end{lemma}
\begin{proof}
We begin by writing, for fixed $g\in\FF$:
\begin{equation*}
\begin{array}{rcl}
E_g^{(2)}&=&\displaystyle
-\nu(L_{a_{11}=0})(g\cB)_{\{a_{11}\}}
+\sum_{b_{11}}\nu(L|_{a_{11}=0;b_{11}})(g\cB)_{\{a_{11}\};\{b_{11}\}}
\\[.15in]&=&\displaystyle
-\nu(L_{a_{11}=0})(g\cB)_{\{a_{11}\}}+\int_{1+\epsilon}\wt{(g\cB)}_{\{a_{11}\};\{b_{11}\}}(w_{11})\vec\zeta_L(a_{11}=0;w_{11})dw_{11}
\\[.15in]&=&\displaystyle
\int_{1-\epsilon}\wt{(g\cB)}_{\{a_{11}\};\{b_{11}\}}(w_{11})\vec\zeta_L(a_{11}=0;w_{11})dw_{11},
\end{array}
\end{equation*}
since the pole at $w_{11}=1$ of the integrand above has residue exactly $\nu(L_{a_{11}=0})(g\cB)_{\{a_{11}\}}$.
Fixing $\lambda$, writing $s=(s_1,s_2)$, $d^\times s=d^\times s_1 d^\times s_2$, and integrating over $\FF_2\FF_3$ now yields
\begin{equation*}
\begin{array}{cl}
&\!\!\displaystyle
\int_{g\in\lambda\FF_2\FF_3}E_g^{(2)}f^X(g)dg
\\[.2in]=&\!\!
\displaystyle
\int_{g\in\lambda\FF_2\FF_3}
\int_{1-\epsilon}\wt{(g\cB)}_{\{a_{11}\};\{b_{11}\}}(w_{11})\vec\zeta_L(a_{11}=0;w_{11})f^X(g)dw_{11}dg
\\[.2in]=&\!\!\displaystyle
\int_{s,t}\int_{1-\epsilon}\wt{((s,t)\lambda\cB)}_{\{a_{11}\};\{b_{11}\}}(w_{11})\vec\zeta_L(a_{11}=0;w_{11})f^X(t)\delta_{\FF_3}(s)dw_{11}{t^{-2}s_1^{-6}s_2^{-6}}d^\times td^\times s
\\[.2in]=&\!\!\displaystyle
\int_{s,t}\int_{1-\epsilon}\wt{(\cB)}_{\{a_{11}\};\{b_{11}\}}(w_{11})\vec\zeta_L(a_{11}=0;w_{11})f^X(t)\delta_{\FF_3}(s)\lambda^{10+w_{11}}{t^{w_{11}-2}s_1^{-4w_{11}+2}s_2^{-2w_{11}-2}}d^\times td^\times s dw_{11}
\\[.2in]=&\!\!\displaystyle
\int_{1-\epsilon}\wt{(\cB)}_{\{a_{11}\};\{b_{11}\}}(w_{11})\vec\zeta_L(a_{11}=0;w_{11})X^{\delta(w_{11}-2)}\wt{f}(w_{11}-2)\wt{\delta_{\FF_3}}(2-4w_{11},-2-2w_{11})\lambda^{10+w_{11}}dw_{11},
\end{array}
\end{equation*}
where in the last step, we recall that $f^X(t)=f(t/X^\delta)$, and integrate over $t$, $s_1$, and $s_2$ as usual. We shift the integral left, where the other two poles are at $w_{11}\in\{1/2,-1\}$. The pole at $w_{11}=1/2$ gives a contribution of some constant times $\lambda^{21/2}X^{-3\delta/2}$, and the other pole gives a contribution of $O(X^{3/4})$. Integrating over $\lambda$ now yields the result.
\end{proof}

Proposition \ref{prop:t_large_Ab11_to_A} follows from \eqref{eq:temp_t_big_11} and Lemma \ref{lem:E2}.

\subsection{Shifting the integral to reduce to a simpler Shintani zeta function}

Applying Proposition \ref{prop:t_large_Ab11_to_A}, we see that for the purposes of evaluating $\I^{(2)}(\cB;X)$, it only remains to compute
\begin{equation}\label{eq:final_before_simpleShintani}
\int_{g\in\FF}\sum_{A}\nu(L_{A})(g\cB)_{\emptyset;\{A\}}(A)\psi\Big(\frac{\lambda(g)}{X^{1/12}}\Big)f^X(g)dg.
\end{equation}
In this section, we relate this quantity to Shintani zeta functions associated to ternary quadratic forms.

Let $\GL_3(\R)^+$ denote the set of elements in $\GL_3(\R)$ having positive determinant. Let $S_3$ denote the space of symmetric (half-integral) $3\times 3$ matrices, with the action of $\GL_3(\R)^+$ given by $g_3A=g_3Ag_3^t$. The set of elements in $S_3$ having nonzero determinant splits up into four $\GL_3(\R)^+$-orbits, one for each possible signature. Denote the orbit with signature $\sigma$ by $S_3(\R)^\sigma$. We let $dA$ be Euclidean measure on $S_3(\R)$, normalized so that $S_3(\Z)$ has co-volume $1$. Let $A\in S_3(\R)$ be an element with nonzero determinant and signature $\sigma$. The stabilizer of $A$ in $\GL_3(\R)^+$ is $\SO_A(\R)$. The Haar measure $\nu_{\GL_3}$ on $\GL_3(\R)^+$, along with a fixed Haar-measure on $\SO_A(\R)$ gives a measure, respecting the $\GL_3(\R)^+$-action, on $\GL_3(\R)^+/\SO_A(\R)$ which is naturally identified with $S_3(\R)^\sigma$. We choose $dh$ to be the Haar-measure on $\SO_3(\R)$ such that the corresponding measure on $S_3(\R)^\sigma$ is $(\det A)^{-2}dA$. Following Kimura \cite[p.163]{Kimura_book}, we then define $\mu(A)$ to be $\Vol(\SO_A(\R)/\SO_A(\Z))$, where the volume is computed with respect to $dh$.

Let $\theta: S_3(\Z)\to\R$ be a periodic $\GL_3(\Z)$-invariant function. 
For a signature $\sigma$, we define the Shintani zeta function $\xi_\sigma^{tq}(s,\theta)$ to be
\begin{equation}\label{eq:TQ_Shintani_zeta}
\xi_\sigma^{tq}(s,\theta):=\sum_{A\in \SL_3(\Z)\backslash S_3(\Z)^\sigma} \frac{\mu(A)\phi(A)}{|\det(A)|^s}.
\end{equation}
This Shintani zeta function is closely related to the {\it global zeta integral}: Let $\Phi:S_3(\R)\to\R$ be a smooth and super-polynomially decaying function. We define
\begin{equation}\label{eq:tq_SZF}
Z(\Phi,\theta;s):= \int_{g\in\GL_3(\Z)\backslash\GL_3(\R)^+}\det(g)^s\sum_{\substack{A\in S_3(\Z)\\\det(A)\neq 0}}\theta(A)(g\Phi)(A)\nu_{\GL_3}(g).
\end{equation}
Then we have the following result. (See Kimura's book \cite[Section 5]{Kimura_book} for a clear exposition.)
\begin{theorem}\label{thm:zeta_integral_Shintani}
The zeta integral $Z^\sigma(\Phi,\theta;s)$ converges absolutely for $\Re(s)$ large enough. Moreover, we have
\begin{equation*}
Z(\Phi,\theta;s)=\sum_{\sigma}\xi_\sigma^{tc}(s,\theta)\int_{A\in S_3(\R)^\sigma}|\det(A)|^{s-2}\Phi(A)dA,
\end{equation*}
where $\sigma$ goes over all signatures.
\end{theorem}

We can write $\FF$ as $\R^\times\cdot\FF_2\FF_3$. For fixed $\lambda$, we write 
\begin{equation*}
\begin{array}{rcl}
\displaystyle \int_{\FF_2\FF_3}\sum_{A}\nu(L_{A})(g\cB)_{\emptyset;\{A\}}(A)f^X(g)dg
&=&\displaystyle
\int_{\FF_2\FF_3}
\sum_{A}\nu(L_{A})((\lambda g_2g_3)\cB)_{\emptyset;\{A\}}(A)f^X(g_2) dg_2dg_3
\\[.2in]&=&\displaystyle
\lambda^6\int_{t>0}\int_{\FF_3}
\sum_{A}\nu(L_{A})\Bigl(\frac{\lambda}{t} g_3\cB\Bigr)_{\emptyset;\{A\}}(A)t^4f^X(t)d^\times tdg_3 .
\\[.2in]&=&\displaystyle
C\lambda^{6+3s}\int_{3+\epsilon}Z\bigl((\cB)_{\emptyset;\{A\}},\nu(L_A);s\bigr)\wt{f^X}(4-3s)ds.
\end{array}
\end{equation*}
Above, the constant $C\in\R^\times$ is chosen to account for the measure changes implicit in the determination of the final line. We will later be able to determine $C$ from purely formal arguments.

By the general theory, the global zeta integral has analytic continuation to the entire complex plane, with poles at most at negative the roots of the associated Bernstein--Sato polynomial $(x+1)(x+3/2)(x+2)$. The function $\wt{f^X}(s)=\wt{f}(s)X^{\delta s}$ has a simple pole at $s=0$ with residue $1$. Therefore, we may shift the integral over $s$ above to $\Re(s)=1+\epsilon$, picking up poles at $2$, $3/2$, and $4/3$, and obtaining for some constants $C$ and $c$:
\begin{equation*}
\begin{array}{rcl}
\displaystyle\int_{g\in\FF_2\FF_3}\sum_{A}\nu(L_{A})(g\cB)_{\emptyset;\{A\}}(A)f^X(g)dg&=&C\lambda^{12}\Res_{s=2}Z\bigl((\cB)_{\emptyset;\{A\}},\nu(L_A);s\bigr)\wt{f^X}(-2)
\\[.2in]&&+\displaystyle 
c\lambda^{21/2}X^{5\delta/2}\wt{f}(5/2)+C\lambda^{10}Z\bigl((\cB)_{\emptyset;\{A\}},\nu(L_A);4/3\bigr)
\\[.2in]&&+\displaystyle 
O(\lambda^{9+O(\delta)+o(1)}).
\end{array}
\end{equation*}
Integrating over $\lambda$, and combining this with Proposition \ref{prop:t_large_Ab11_to_A}, we obtain
\begin{equation*}
\begin{array}{rcl}
\I^{(2)}(\cB;X)&=&
\displaystyle C\wt{\psi}(12)\Res_{s=2}Z((\cB)_{\emptyset;{A}}\nu(L_A);s)\wt{f^X}(-2)X+C\wt{\psi}(10)Z\bigl((\cB)_{\emptyset;\{A\}},\nu(L_A);4/3\bigr)X^{5/6}
\\[.15in]&&\displaystyle
+c_0X^{21/24+5\delta/2}+c_1X^{21/24-3\delta}+c_2X^{21/24-3\delta/2}+O(X^{4/5+\epsilon})
\end{array}
\end{equation*}
completing the proof of Proposition \ref{prop:section6_t_big_main}.

\section{Computing the first two residues}
Let $L\subset V(\Z)$ be a $G(\Z)$-invariant set, defined by finitely many congruence conditions, such that every triple $(Q,C,r)$ corresponding to a $G(\Z)$-orbit on $L$ is an $S_4$-triple. In this section, we determine the residues of the poles at $1$ and $5/6$ of the Shintani zeta functions corresponding to $L$.

\subsection{The work of Ibukiyama-Saito}

 We summarize (and generalize very slightly) in this section some of the main results of Ibukiyama--Saito \cite{ibukiyamasaitoI}. Following their notation, for a ring $R$, let $S_3(R)$ denote the set of ternary quadratic forms with coefficients in $R$ (represented as $3\times 3$
 symmetric matrices), and for $d\in R$, let $S_3(R,d)$ denote the subset of such matrices of determinant $d$. Let $(\omega_p)_p$ be a collection of functions $\omega_p:S_3(\Z_p)\to\R$, for every prime $p$, satisfying the following two conditions.
 \begin{itemize}
     \item[{\rm (1)}] For every prime $p$, the function $\omega_p$ is invariant under the actions of $\GL_3(\Z_p)$ and $\Z_p^{\times}$;
     \item[{\rm (2)}] For sufficiently large primes $p$, the function $\omega_p$ is the characteristic function of $S_n(\Z_p)$.
  \end{itemize}
Let $\omega:S_3(\Z)\to\R$ be defined by $\omega(x)=\prod_p\omega_p(x)$. Fix $i$ with $0\leq i\leq 3$, and abusing notation, let $i$ also denote the signature $(i,3-i)$. 
Consider the Shintani zeta function $\xi_i^{tq}(s,\omega)$, defined in \eqref{eq:TQ_Shintani_zeta}, corresponding to  the function $\omega$ and the signature $i$.\footnote{Note that Ibukiyama-Saito have an extra factor of $c_3$ (defined below) in their definition of the Zeta function so as to cancel out some irrational factors, compared to Kimura.} We define the constants
\begin{equation*}
\delta_i:=(-1)^{3-i};\quad \epsilon_i:=(-1)^{\binom{4-i}{2}}; \quad c_3:=2\displaystyle\prod_{i=1}^3\Gamma(i/2)\pi^{-3}=\pi^{-2}.
\end{equation*}
For a prime $p$, and an element $x\in S_3(\Z_p)$ with nonzero determinant, we define the quantity $\alpha_p(x)$ as in \cite[pp.1104]{ibukiyamasaitoI}:
 \begin{equation}\label{eq:def_alpha_p}
 \alpha_p(x):=\frac12\lim_{k\to\infty}p^{-3k}|O(x_k)|,
 \end{equation}
 where $x_k$ is the reduction of $x$ modulo $p^k$ and $O(x_k)$ is the orthogonal subgroup of $\GL_3(\Z/p^k\Z)$ preserving $x_k$.
 Finally, let $\epsilon_p(x)$ denote the Hasse invariant of $x$, and let $\omega'_p:=\omega_p\epsilon_p$. We define our Hasse-invariants as in \cite{ibukiyamasaitoI}, so that if $x$ can be diagonalized over $\Q_p$ with entries $(a_1,a_2,a_3)$, then $\epsilon_p(x)=\prod_{i\leq j} (a_i,a_j)_p$ where $(,)_p$ is the Hilbert symbol.

Let $f=(f_p)_p$ be a collection, over all primes $p$, of $\GL_3(\Z_p)$-invariant functions $f_p:S_3(\Z_p)\to\R$. For an integer $d\neq 0$, we define the quantity

 $$\lambda_p(d,f)=\sum_{x_p\in S_n(\Z_p,d)/\sim} f_p(x_p)\alpha_p(x_p)^{-1}$$
where $\sim$ denotes $\GL_3(\Z_p)$ equivalence, and $\alpha_p(x_p)$ is the normalized stabilizer in \eqref{eq:def_alpha_p}.
Finally, define 

$$\lambda(d,f):=\prod_p \lambda_p(d,f)\quad \mbox{  and  } \quad\zeta_{f,i}(s):=\sum_{\delta_i d>0} \lambda(d,f)2^{3s} |d|^{2-s}.$$
Then we have the following result

 \begin{proposition}\label{prop:symsumformula}
 Fix $i\in\{0,1,2,3\}$ and let $\omega=\prod_p\omega_p:S_3(\Z)\to\R$ be the function from the start of this subsection. Then we have
 $$\xi^{tc}_{i}(s,\omega)=\zeta_{\omega,i}(s)+\epsilon_i \zeta_{\omega',i}(s).$$
 \end{proposition}

\begin{proof}
    This is proved in \cite[prop 2.2]{ibukiyamasaitoI} when $\omega_p$ is the characteristic function of $S_3(\Z_p)$, but the proof extends verbatim to our slightly more general setting of finitely many congruence conditions. Note that is essential that our functions $\omega_p$ be $\GL_3(\Z_p)$-equivalent in order to apply Siegel's formula.
\end{proof}

We next prove that the zeta functions $\zeta_{\omega,i}(s)$ and $\zeta_{\omega',i}(s)$ satisfy product formulas.

\begin{lemma}\label{lem:prodformulasym}
We have
\begin{equation*}
\zeta_{\omega,i}(s)=2^{3s}\prod_p\Bigl(\sum_{k\geq 0}
\lambda_p(p^k,\omega)p^{k(2-s)}
\Bigr);\quad
\zeta_{\omega',i}(s)=2^{3s}\prod_p\Bigl(\sum_{k\geq 0}
\lambda_p(p^k,\omega')p^{k(2-s)}
\Bigr).
\end{equation*}
\end{lemma}

\begin{proof}
We first claim that $\lambda_p(d,\omega)$ and $\lambda_p(d,\omega')$ depends only on $v_p(d)$. Note first that this only depends on the square class of $d$ in $\Z_p$, as can be seen by acting by an element of $\GL_3(\Z_p)$ with unit determinant on the set $S_n(\Z_p,d)/\sim$. 
Moreover, it is invariant by the cubes of units, as can be seen by the action of scaling by a unit on the set $S_n(\Z_p,d)/\sim$, which preserves the Hasse invariant. This completes the proof. 

Set $\lambda'_p(d,f):=\frac{\lambda_p(d,f)}{\lambda_p(1,f)}$. Then,  for $f=\omega$ or $\omega'$, we may write $\zeta_{f,i}(s)$ as
\begin{align*}
\sum_{\delta d>0} \lambda(d,f)2^{3s} |d|^{2-s}
&=\sum_{d>0} \prod_p\lambda(p^{v_p(d)},f)2^{3s} |d|^{2-s}\\
&=\lambda(1,f)\sum_{d>0}  \prod_p\lambda'_p(p^{v_p(d)},f)2^{3s} |d|^{2-s}\\
&=2^{3s}\lambda(1,f)\prod_p\sum_{k\geq 0}\lambda'_p(p^k,f) p^{k(2-s)}\\
&=2^{3s}\prod_p\sum_{k\geq 0}\lambda_p(p^k,f) p^{k(2-s)},
\end{align*}  
as desired.
\end{proof}

In the above proof we establish that $\lambda_p(d,f)$ depends only on the valuation of $d$. This motivates the following definition. For a function $\phi:S_3(\Z_p)\to\R$, we define $$\zeta_p(\phi,s):= \int_{S_3^{ss}(\Z_p)}\phi(x) |\det(x)|_p^s dx$$ where $dx$ assigns measure $1$ to $S_3(\Z_p)$. Unwinding the definition of $\alpha_p(x_p)$ and using the fact that $\lambda_p(d,\omega)$ depends only on the valuation of $d$, we immediately obtain (as in \cite[p.1109]{ibukiyamasaitoI}) that $\sum_{k\geq 0}\lambda_p(p^k,f) p^{k(2-s)}$is proportional to $\zeta_p(f,s-2)$. It follows that 
\begin{equation}
\zeta_{\omega,i}(s)=\alpha_i2^{3s}\prod_p \frac{\zeta_p(\omega_p,s-2)}{1-p^{-1}}, \quad
\zeta_{\omega',i}(s)=\alpha_i2^{3s}\prod_p \frac{\zeta_p(\omega'_p,s-2)}{1-p^{-1}}
\end{equation}
for some fixed constant $\alpha_i$, and similarly for $\omega'$, as in the bottom of \cite[P. 1108]{ibukiyamasaitoI}.
Let $\iota_p$ denote the characteristic function of $S_3(\Z_p)$. Comparing with \cite[Thm 1.2]{ibukiyamasaitoI} gives the following: 

\begin{equation}\label{eqref:3x32ndmainterms}
\zeta_p(\iota_p,s)=2^{-s\delta_{p=2}}\frac{(1-p^{-1})(1-p^{-3})}{(1-p^{-1-s})(1-p^{-3-2s})} 
\end{equation}
and
\begin{equation}\label{eqref:3x32ndmainterms}
\zeta_p(\iota'_p,s)=2^{-s\delta_{p=2}}\frac{(1-p^{-1})(1-p^{-3})}{(1-p^{-2-s})(1-p^{-2-2s})} 
\end{equation}
and therefore that $\alpha_i=\frac{\zeta(3)}{24c_3} = \frac{\zeta(2)\zeta(3)}{4}$, independent of $i$. We denote $\zeta_{\omega,i}$ and $\zeta_{\omega',i}$ by $\zeta_\omega$ and $\zeta_{\omega'}$, respectively. Summarizing, we obtain the following result.

\begin{theorem}\label{thm:tczetaformula}
Let $\omega=\prod_p \omega_p$ be as above, and $0\leq i\leq 3$. Recall that $\epsilon_i=(-1)^{\binom{4-i}{2}}$. Then
\begin{equation*}
\begin{array}{rl}
    \xi^{tc}_i(s,\omega)=\displaystyle\frac{\zeta(2)\zeta(3)}{4}\cdot 2^{3s}&\!\!\!
    \displaystyle\Bigl( \prod_{p}(1-p^{-1})^{-1}\cdot\int_{S_3(\Z_p)}\omega_p(x) |\det(x)|_p^{s-2} dx 
\\[.2in]
  +&\displaystyle  \epsilon_i\prod_{p}(1-p^{-1})^{-1}\cdot\int_{S_3(\Z_p)}\omega_p(x)\epsilon_p(x) |\det(x)|_p^{s-2} dx\Bigr).
\end{array}
\end{equation*}
\end{theorem}

For a place $v$ of $\Q$ we define the function $\kappa_v:S_3(\Q_v)\backslash\{\det=0\}\to\{0,1\}$ by setting $\kappa_v(A)=1$ when $A$ is isotropic and $\kappa_v(A)=-1$ otherwise.
For a real number $\kappa$, we define the quantities
\begin{equation*}
\mathcal W_{\kappa}(\cB):=\int_{ V(\R)}\cB(A,B)|\det(A)|^{-\kappa}dAdB;\quad
\mathcal W_{\kappa}'(\cB):=\int_{V(\R)}\cB(A,B)\kappa_\infty(A)|\det(A)|^{-\kappa}dAdB.
\end{equation*}
As a simple consequence of the Theorem \ref{thm:tczetaformula}, we compute the residue at $s=2$ and special value at $4/3$ of the global zeta integral $Z(\cB_{\emptyset;A},\nu(L_A);s)$ arising in the power series expansion of $\I^{(2)}(\cB,X)$. 

\begin{corollary}\label{cor:special_value}
We have
\begin{equation*}
\begin{array}{rcl}
\Res_{s=2}Z(\cB_{\emptyset;A},\nu(L_A);s)&=&\displaystyle
16\zeta(2)\zeta(3)\cdot\Vol(\cB)\nu(L);
\\[.15in]
Z(\cB_{\emptyset;A},\nu(L_A);4/3)&=&\displaystyle 
\zeta_{\nu(L_A)}(4/3)\mathcal W_{2/3}(\cB)
-
\zeta_{\nu(L_A)'}(4/3)\mathcal W_{2/3}'(\cB).
\end{array}
\end{equation*}
\end{corollary}
\begin{proof}
We apply Theorem \ref{thm:zeta_integral_Shintani} to write
\begin{equation*}
\Res_{s=2}Z(\cB_{\emptyset;A},\nu(L_A);s)=\sum_{i=0}^3\Res_{s=2}\xi_i^{tq}(s,\nu(L_A)\int_{A\in S_3(\R)^{(i,3-i)}}\cB_{\emptyset;A}(A)dA.
\end{equation*}
From Theorem \ref{thm:tczetaformula}, we see that the residue at $s=2$ of $\xi_i^{tq}(s,\omega)$ (for any acceptable $\omega$) is in fact independent of $i$ (since the product in the second summand converges absolutely at $s=2$). Therefore, we obtain
\begin{equation*}
\Res_{s=2}Z(\cB_{\emptyset;A},\nu(L_A);s)=
16\zeta(2)\zeta(3)\Bigl(\int_{A\in S_3(\R)}\cB_{\emptyset;A}(A)dA\Bigr)\prod_p\int_{A\in S_3(\Z_p)}\nu(L_A)dA.
\end{equation*}
The first claim above follows immediately since the integral of $\cB_{\emptyset;A}(A)$ over $A$ is $\Vol(\cB)$ and the integral of $\nu(L_A)$ over $A$ is $\nu(L)$. The second claim follows similarly by noting that for $A\in S_3(\R)$ having signature $(i,3-i)$, we have $\kappa_\infty(A)=-\epsilon_i$.
\end{proof}

\subsection{Computing the residues of the quartic Shintani zeta functions}

We are now ready to return to our quartic Shintani zeta functions $\xi_{i,L}(s)$. For $i\in\{0,1,2\}$, let $\cB:V(\R)^{(i)}\to\R_{\geq 0}$ be a smooth and compactly supported function. We define the ratios
\begin{equation*}
\mathcal Q_i:=\frac{\mathcal W_{2/3}(\cB)}{\V_{5/6}(\cB)},
\quad\quad\quad
\mathcal Q_i':=\frac{\mathcal W_{2/3}'(\cB)}{\V_{5/6}(\cB)},
\end{equation*}
which we will subsequently prove are independent of $\cB$. Then we have the following result.
\begin{theorem}\label{th:quartic_szf_residues}
Let $L\subset V(\Z)$ be an $S_4$-set defined by finitely many congruence conditions, and let $i\in\{0,1,2\}$. Then $\xi_{i,L}$ has simple poles at $s=1$ and $s=5/6$ with residues given by
\begin{equation*}
\begin{array}{rcl}
\Res_{s=1}\xi_{i,L}(s)&=&
\displaystyle\frac{1}{2A_i}\zeta(2)^2\zeta(3)\nu(L);
\\[.15in]
\Res_{s=5/6}\xi_{i,L}(s)
&=&\displaystyle
\frac{\pi}{32A_i}
\Bigl(
\mathcal Q_i \zeta_{\nu(L_A)}(4/3)
-
\mathcal Q_i' \zeta_{\nu(L_A)'}(4/3)
\Bigr).
\end{array}
\end{equation*}
\end{theorem}
\begin{proof}
Recall that we write $\I(\cB;X)=\I^{(1)}(\cB;X)+\I^{(2)}(\cB;X)$, and evaluate the latter two terms in Propositions \ref{prop:tsmallpowerexpansion} and \ref{prop:section6_t_big_main}. Combining those two results gives us a power series expansion for $\I(\cB;X)$, up to an error term of $O(X^{4/5+o(1)})$. Turn now to \eqref{eq:gen_count_1}, \eqref{eq:residues_of_GZI}, and \eqref{eq:GZI_to_I}: taken together, they also yield a (different) power series expansion for $\I(\cB;X)$ in terms of the residues of the Shintani zeta function of $\xi_{i,L}(s)$. Comparing the exponents of the terms that arise in these two power series expansions, we see that all the terms in both power series expansions other than the $X$- and $X^{5/6}$- terms must vanish for formal reasons! This immediately implies the first claim of the result, namely that $\xi_{i,L}(s)$ has at most simple poles at $1$ and $5/6$. 

Next, equating the coefficient of the $X$-term in $\I(\cB;X)$ from \eqref{eq:Iexpansion} with the sums of the coefficients of the $X$-terms in $\I^{(1)}(\cB;X)$ and $\I^{(1)}(\cB;X)$, we obtain
\begin{equation*}
\frac{12}{J}A_i\wt{\psi}(12)\Res_{s=1}\xi_{i,L}(s)
= 
\Vol_{dg_3}(\FF_3)\Bigl(\int_{g_2\in\FF_2}f_0^X(g_2)dg_2\Bigr)\nu(L)
+
(16\zeta(2)\zeta(3))C\wt{f^X}(-2)\nu(L).
\end{equation*}
The left hand side above is patently independent of $f$ and $\delta$. The right hand side must be so too. Since we have 
$$\wt{f^X}(-2)=\int_{g_2\in\FF_2}(1-f_0^{X})(g_2)dg_2,$$
it follows that $C=\Vol_{dg_3}(\FF_3)/(16\zeta(2)\zeta(3))$. Therefore, using the definition of $J$ from \eqref{eq:J_def}, we see that the residue at $s=1$ of $\xi_{i,L}(s)$ is given by
\begin{equation*}
\begin{array}{rcl}
\Res_{s=1}\xi_{i,L}(s)&=&
\displaystyle\frac{J}{12A_i}\Vol_{dg_2}(\FF_2)\Vol_{dg_3}(\FF_3)\nu(L)
\\[.15in]
&=&
\displaystyle\frac{1}{2A_i}\Vol_{\nu_{\SL_2}}(\FF_2)\Vol_{\nu_{\SL_3}}(\FF_3)\nu(L)
\\[.15in]
&=&
\displaystyle\frac{1}{2A_i}\zeta(2)^2\zeta(3)\nu(L).
\end{array}
\end{equation*}
as claimed by the theorem.

Note that $\I^{(1)}(\cB;X)$ has no $X^{5/6}$-term. We equate the coefficients of the $X^{5/6}$-terms of $\I(\cB;X)$ and $\I^{(2)}(\cB;X)$ from \eqref{eq:Iexpansion} and Proposition \ref{prop:section6_t_big_main}. Using Corollary \ref{cor:special_value} for the special value of $\xi_i^{tq}$ and the recently computed value of $C$, we obtain
\begin{equation*}
\begin{array}{rcl}
\displaystyle\frac{12}{J}A_i\Res_{s=5/6}\xi_{i,L}V_{5/6}(\cB)
&=&\displaystyle
CZ\bigl((\cB)_{\emptyset;\{A\}},\nu(L_A);4/3\bigr)
\\[.15in]&=&\displaystyle
\frac{\Vol_{dg_3}(\FF_3)}{16\zeta(2)\zeta(3)}\bigl(
\zeta_{\omega}(\nu(L_A),4/3)\mathcal W_{2/3}(\cB)
-
\zeta_{\omega'}(\nu(L_A),4/3)\mathcal W_{2/3}'(\cB)
\bigr).
\end{array}
\end{equation*}
Now the volume $\Vol_{dg_2}(\FF_2)$ can be computed by the Gauss--Bonnet theorem to be $\pi/6$. Hence we have 
$$\frac{J}{12}\frac{\Vol_{dg_3}(\FF_3)}{16\zeta(2)\zeta(3)}=\frac{J\Vol_{dg_2}(\FF_2)\Vol_{dg_3}(\FF_3)}{32\pi\zeta(2)\zeta(3)}=\frac{3}{16\pi}\zeta(2)=\frac{\pi}{32}.$$
Therefore, the value of $\Res_{s=5/6}\xi_{i,L}(s)$ is
given by
\begin{equation*}
\Res_{s=5/6}\xi_{iL}(s)=\frac{\pi}{32A_i}
\Bigl(
\zeta_{\nu(L_A)}(4/3)
\frac{\mathcal W_{2/3}(\cB)}{\V_{5/6}(\cB)}
-
\zeta_{\nu(L_A)'}(4/3)
\frac{\mathcal W_{2/3}'(\cB)}{\V_{5/6}(\cB)}
\Bigr).
\end{equation*}
However, this value is independent of the choice of the $K$-invariant set $\cB$. Moreover, the quotients $\W_{2/3}(\cB)/\V_{5/6}(\cB)$ and $\W_{2/3}(\cB)/\V_{5/6}(\cB)$ are independent of the choice of $L$. Picking some $L$'s with different ratios $\zeta_{\nu(L_A)}(4/3)/\zeta_{\nu(L_A)'}(4/3)$ (see \S9 for many such examples), we see that the quantities $\mathcal Q_i$ and $\mathcal Q_i'$ are independent of $\cB$. 
This completes the proof of the theorem.
\end{proof}

\subsection{Computing the Archimedean local integrals}

In this subsection, we compute the values of $\mathcal Q_i$ and $\mathcal Q_i'$. Recall the constant $\mathcal M$ from the introduction:
\begin{equation*}
\mathcal M= \frac{2^{5/3}\Gamma(1/6)\Gamma(1/2)}{\sqrt{3}\pi\Gamma(2/3)}.
\end{equation*}
We have the following result.
\begin{proposition}\label{prop:local_int_at_inf}
Let $i\in\{0,1,2\}$. For any $\SO_2(\R)\times\SO_3(\R)$-invariant function $\cB:V(\R)^{(i)}\to\R$, we have
\begin{equation*}
\begin{array}{ccrclcl}
\mathcal Q_i&=
& 
\displaystyle\frac{\W_{2/3}(\cB)}{V_{5/6}(\cB)}
&=&\displaystyle
\left\{
\begin{array}{rcl}
\mathcal M&{\rm if}&i\in\{0,2\};
\\
\sqrt{3}\cdot\mathcal M&{\rm if}&
i=1;
\end{array}
\right. &=& \mathcal M_i;
\\[.2in]
\mathcal Q_i'&=
&
\displaystyle\frac{\W'_{2/3}(\cB)}{ V_{5/6}(\cB)}
&=&\left\{
\begin{array}{rcl}
\mathcal M&{\rm if }& i=0;\\
\sqrt{3}\cdot \mathcal M&{\rm if }& i=1;\\
\displaystyle\frac{\mathcal M}{3}&{\rm if }& i=2.
\end{array}
\right.&=&\mathcal M_i'.
\end{array}
\end{equation*}
\end{proposition}
\begin{proof}
We have already seen that the values of $\mathcal Q_i$ and $\mathcal Q_i'$ are independent of the function $\cB$. Hence we may assume that $\cB$ is of the form $\cB(g\cdot s(f))=\phi(g)R(f)$, for $g\in\SL_3(\R)$ and $f\in U(\R)$, where $\phi:\SL_3(\R)\to\R$ is an $\SO_3(\R)$-invariant function, $R:U(\R)\to\R$ is a $\SO_2(\R)$-invariant function, and $s:U(\R)\to V(\R)^{(i)}$ is a section (i.e., $\Res(s(f))=f$). By the change of measures formula in Proposition \ref{prop:jac_sl_3}, we have
\begin{equation*}
\begin{array}{rcl}
\W_{2/3}(\cB)&=&\displaystyle\int_{(A,B)\in V(\R)}\cB(A,B)|\det(A)|^{-2/3}\nu_V(A,B)
\\[.2in]
&=&\displaystyle\frac1{\sigma_i}
\int_{g\in\SL_3(\R)}\phi(g)\nu_{\SL_3}(g)\int_{f\in U(\R)}R(f)\left(\frac{|a(f)|}{4}\right)^{-2/3}\nu_U(f),
\end{array}
\end{equation*}
where $\sigma_i$ is the (common) size of the stabilizer in $\SL_3(\R)$ of $(A,B)\in V(\R)^{(i)}$.
Similarly, we have
\begin{equation*}
V_{5/6}(\cB)=\int_{x\in\cB}|\Delta(x)|^{-1/6}\nu_V(x) = \frac1{\sigma_i}\int_{g\in\SL_3(\R)}\phi(g)\nu_{\SL_3}(g)\int_{f\in U(\R)}R(f)|\Delta(f)|^{-1/6}\nu_U(f).
\end{equation*}
Denote the rightmost integral by $V_{5/6}(R)$.
Taking quotients, we arrive at
$$\mathcal Q_i=4^{2/3}V_{5/6}(R)^{-1}\displaystyle\int_{f\in U(\R)}R(f)|a(f)|^{-2/3}df.$$
This is independent of $R$ so long as $R$ is $\SO_2(\R)$-invariant. So we may write, for any $f$ whose discriminant has the correct sign:
\begin{equation*}
\mathcal Q_i
=4^{2/3}\Delta(f)^{1/6}\displaystyle\int_{\theta=0}^{2\pi}f(\cos(\theta),\sin(\theta))^{-2/3}\frac{d\theta}{2\pi}
=\displaystyle
\left\{
\begin{array}{rcl}
\displaystyle\frac{2^{5/3}\Gamma(1/6)\Gamma(1/2)}{\pi\sqrt{3}\Gamma(2/3)}&\mbox{if} &\Delta(f)>0,
\\[.1in]
\displaystyle\frac{2^{5/3}\Gamma(1/6)\Gamma(1/2)}{\pi\Gamma(2/3)}&\mbox{if}&
\Delta(f)<0,
\end{array}
\right.
\end{equation*}
where the last line follows from the computation at the end of \cite[\S6.1]{BST}.

For $i\in\{0,1\}$, we have $\mathcal Q_i'=\mathcal Q_i$ since if $(A,B)\in V(\R)^{(i)}$ in these two cases, then both $A$ and $B$ are isotropic (since the conics cut out by them has $4$ or $2$ common zeros, respectively). Hence $\kappa_\infty(A)$ is always $1$ in this case. To handle $i=2$, we proceed as follows. Note that $\mathcal Q_2$ is independent of the $\SO_2(\R)\times\SO_3(\R)$-invariant $\cB$, and also that the values of $\kappa_\infty$ and $\det(A)$ are invariant under the action of $\SL_3(\R)$. Hence, denoting $\cos(\theta)A+\sin(\theta)B$ by $(A,B)_\theta$, we have 
\begin{equation*}
\begin{array}{rcl}
\mathcal Q_2&=&\displaystyle\Delta(A,B)^{1/6}\int_{\theta=0}^{2\pi}|\det((A,B)_\theta)|^{-2/3}\frac{d\theta}{2\pi},
\\[.2in]
\mathcal Q_2'&=&\displaystyle\Delta(A,B)^{1/6}\int_{\theta=0}^{2\pi}\kappa_\infty((A,B)_\theta)|\det((A,B)_\theta)|^{-2/3}\frac{d\theta}{2\pi},
\end{array}
\end{equation*}
for every $(A,B)\in V(\R)^{(2)}$.
We choose $A=\diag(1,1,1)$ and $B=\diag(0,\sqrt{3},-\sqrt{3})$, in which case
$$\det(A,B)_\theta=\cos(\theta)(\cos(\theta)+\sqrt{3}\sin(\theta))(\cos(\theta)-\sqrt{3}\sin(\theta))=\cos(3\theta).$$
It then follows that $\kappa_\infty((A,B)_\theta)$ is $-1$ exactly when all three factors above have the same sign, which happens for 
\begin{equation*}
\theta\in [0,\pi/6]\cup[5\pi/6,7\pi/6]\cup[11\pi/6,2\pi].
\end{equation*}
Therefore it follows that
$\mathcal Q_2'=\frac13 \mathcal Q_2$, since the integral of $\cos(3\theta)d\theta$ over the interval $[\frac{m\pi}{6},\frac{(m+1)\pi}{6}]$ does not depend on $m$.
\end{proof}

\subsection{Interpreting the non-Archimedean integrals in the language of rings}

In this subsection, we reinterpret our non-Archimedean integrals in the language of rings. For a $\Z_p$ triple $(Q,C,r)$ of rings, let $\Sigma_{(Q,C,r)}\subset V(\Z_p)$ denote the open subset (consisting of a single $G(\Z_p)$-orbit) corresponding to $(Q,C,r)$. For a free module $M$ over $\Z_p$, let $\Bas(M)\subset M^{\dim M}$ denote the subset of ordered bases of $M$, equipped with the natural measured induced from the Haar measure on $M^{\dim M}$. There is a natural surjection $$\phi:\Bas(Q/\Z_p)\times \Bas(C/\Z_p)\ra \Sigma_{(Q,C,r)}$$ given by simply expressing the quadratic resolvent map in the chosen co-ordinates. 

There is a natural action of $u\in \Z_p^{\times}$ on $\Bas(Q/\Z)\times\Bas(C/\Z)$ which scales the basis of $Q$ by $u$ and that of $\Bas(C/\Z)$ by $u^2$. Then $\phi$ induces a bijection up to the action of $\Z_p^{\times}\times \Aut(Q,C,r)$. We therefore define $$\Bas(Q,C):=\left(\Bas(Q/\Z_p)\times\Bas(C/\Z_p)\right)/\Z_p^{\times}.$$ We associate to $\Bas(Q,C)$ the quotient measure normalized so that $$\mu(\Bas(Q,C))\mu(\Z_p^{\times}) = \mu(\Bas(Q/\Z_p)\times\Bas(C/\Z_p)).$$
In fact, if we consider the variety $\PBas(Q,C):=\bG_m\backslash\Bas(Q,C)$ where the action is as described above, then the measure we are defining is the $\nu_{\PBas}$ measure as defined in \S\ref{sec:measures}. We obtain an induced map

$$P\phi:\PBas(Q,C)\ra \Sigma_{(Q,C,r)}.$$ Now since $\PBas$ is a trivial $G$-torsor, it follows that the measure on $\PBas$ is just $\nu_G$ measure pushed forward. We thus obtain the following from Proposition \ref{prop:first_change_of_measures}:

\begin{lemma}\label{lem:ringstoformsmeasure}
We have
$$\frac{\nu_V\mid_{\Sigma_{(Q,C,r)}}}{|\Delta(Q)|_p}=\frac{(P\phi)_*\nu_{\PBas}}{\#\Aut(Q,C,r)}$$
\end{lemma}
Using this lemma, we may re-express the densities and local integrals appearing as factors of the ternary quadratic Shintani zeta function in the language of rings.
First, note the following equality.
\begin{equation}\label{eq:VVol_to_PBasVol}
\begin{array}{rcl}
\nu(\Sigma_{(Q,C,r)})&=&\displaystyle\frac{|\Delta(Q)|_p}{\#\Aut(Q,C,r)}\nu_{\PBas}(\PBas(Q,C))
\\[.2in]
&=&\displaystyle (1-p^{-1})\cdot\frac{|\Delta(Q)|_p}{\#\Aut(Q,C,r)}\cdot\Vol(\SL_2(\Z_p))\Vol(\SL_3(\Z_p)).
\end{array}
\end{equation}
Let $\omega:S_3(\Z_p)\to\R$ be the function $\nu(\Sigma_{(Q,C,r),A})$, sending $A$ to the volume of the fiber over $A$.
Next note that the following equalities follow from the above discussion.
\begin{equation}\label{eq:ringintegralnoep}
\begin{array}{rcl}
\zeta_p(\omega,s)&=&\displaystyle(1-p^{-1})\cdot \frac{|\Delta(Q)|_p}{\#\Aut(Q,C,r)}\cdot \Vol(\SL_3(\Z_p))\int_{x\in (C/\Z_p)^{\vee}_{\prim}}  |\det{r_x}|_p^{s}dx.
\\[.2in]
\zeta_p(\omega',s)&=&\displaystyle(1-p^{-1})\cdot\frac{|\Delta(Q)|_p}{\#\Aut(Q,C,r)}\cdot \Vol(\SL_3(\Z_p))\int_{x\in (C/\Z_p)^{\vee}_{\prim}}  \epsilon_p(r_x)|\det{r_x}|_p^{s}dx.
\end{array}
\end{equation}
Denote the two integrals above by $I_p((Q,C,r),s)$ and $I_p'((Q,C,r),s)$

We are ready to prove Theorem \ref{thm:mainS4rings}. 
For a $G(\Z)$-invariant function $\phi:V(\Z)\to\R$, let $\phi_p:V(\Z_p)\to\R$ denote the ($G(\Z_p)$-invariant) completion of $\phi$ at $p$. Given a triple $(Q,C,r)$ over $\Z_p$, let $\phi_p(Q,C,r)$ denote $\phi_p(A,B)$ for (any) $(A,B)\in \Sigma_{(Q,C,r)}$. We begin with the following result.
\begin{theorem}\label{thm:residues_final}
Let $\phi:V(\Z)\to\R$ be a $G(\Z)$-invariant $S_4$-function defined by finitely many congruence conditions. Then for $i\in\{0,1,2\}$, we have
\begin{equation*}
\begin{array}{rcl}
\Res_{s=1}\xi_{i}(\phi;s)&=&\displaystyle \frac{1}{2A_i}\prod_p
\Bigl(1-\frac{1}{p}\Bigr)\Bigl(\sum_{(Q,C,r)}\frac{\phi(Q,C,r)|\Delta(Q)|_p}{\#\Aut((Q,C,r)}\Bigr),
\\[.2in]
\Res_{s=5/6}\xi_{i}(\phi;s)&=&\displaystyle\frac{\pi}{8A_i}
\left(
\mathcal M_i\zeta(1/3)\prod_p\Bigl(1-\frac{1}{p^{1/3}}\Bigr)
\Bigl(\sum_{(Q,C,r)}\frac{\phi(Q,C,r)|\Delta(Q)|_pI_p((Q,C,r),-2/3)}{\#\Aut((Q,C,r)}\Bigr)\right.
\\[.2in]
&&\displaystyle \left. +\;
\mathcal M_i'\zeta(2/3)\prod_p\Bigl(1-\frac{1}{p^{2/3}}\Bigr)
\Bigl(\sum_{(Q,C,r)}\frac{\phi(Q,C,r)|\Delta(Q)|_pI'_p((Q,C,r),-2/3)}{\#\Aut((Q,C,r)}\Bigr)
\right),
\end{array}two
\end{equation*}
where the sum is over all triples $(Q,C,r)$ over $\Z_p$.
\end{theorem}
\begin{proof}
Recall the computation of these two residues from Theorem \ref{th:quartic_szf_residues}. To evaluate the residue at $s=1$, we simplify the expression in Theorem \ref{th:quartic_szf_residues} using \eqref{eq:VVol_to_PBasVol} to write
\begin{equation*}
\begin{array}{rcl}
\Res_{s=1}\xi_{i}(\phi;s)&=&\displaystyle \frac{1}{2A_i}\zeta(2)^2\zeta(3)\prod_p\sum_{(Q,C,r)}\phi_p(Q,C,r)\nu(\Sigma_{(Q,C,r)})
\\[.2in]
&=&\displaystyle
\frac{1}{2A_i}\tau(\SL_2(\Q)\times\SL_3(\Q))\prod_p
\Bigl(1-\frac{1}{p}\Bigr)\Bigl(\sum_{(Q,C,r)}\frac{\phi_p(Q,C,r)|\Delta(Q)|_p}{\#\Aut((Q,C,r)}\Bigr),
\end{array}
\end{equation*}
where $\tau(\SL_2(\Q)\times\SL_3(\Q)$ denotes the Tamagawa number of $\SL_2\times\SL_3$ over $\Q$. The second equality follows because $\zeta(2)^2\zeta(3)$ is the volume of the fundamental domain for the action of $(\SL_2\times\SL_3)(\Z)$ on $(\SL_2\times\SL_3)(\Z)$. Since this Tamagawa number is $1$, the first claim of the theorem has been proved.

The second claim follows similarly. We begin by applying \ref{th:quartic_szf_residues}, obtaining
\begin{equation}\label{eq:final_res_comp_1}
\begin{array}{rcl}
\Res_{s=5/6}\xi_{i}(\phi;s)
&=&\displaystyle
\frac{\pi}{32A_i}
\Bigl(
\mathcal Q_i \zeta_{\nu(\phi|_A)}(4/3)
-
\mathcal Q_i' \zeta_{\nu(\phi|_A)'}(4/3)
\Bigr),
\end{array}
\end{equation}
where $\nu(\phi|_A)=\prod_p\nu(\phi_p|_A)$ for the functions $\nu(\phi_p|_A):S_3(\Z_p)\to\R$ sending $A$ to the integral of $\phi(A,B)$ over $B\in S_3(\Z_p)$. We use \eqref{eqref:3x32ndmainterms}, the subsequent computation of $\alpha_i$, and \eqref{eq:ringintegralnoep} to write
\begin{equation}\label{eq:final_res_comp_2}
\begin{array}{rcl}
\zeta_{\nu(\phi|_A)}(4/3)&=&\displaystyle 4\zeta(2)\zeta(3)\zeta(1/3)\prod_p\Bigl(1-\frac1{p}\Bigr)^{-1}\Bigl(1-\frac1{p^{1/3}}\Bigr)\zeta_p(\nu(\phi_p|_A),-2/3)
\\[.2in]&=&\displaystyle
4\tau(\SL_3(\Q))\zeta(1/3)\prod_p\Bigl(1-\frac1{p^{1/3}}\Bigr)\sum_{(Q,C,r)}\frac{\phi(Q,C,r)|\Delta(Q)|_p}{\#\Aut(Q,C,r)}I_p((Q,C,r),-2/3),
\end{array}
\end{equation}
where we have multiplied and divided by $\zeta(1/3)$ to make the local product converge. Similarly, we have
\begin{equation}\label{eq:final_res_comp_3}
\begin{array}{rcl}
\zeta_{\nu(\phi|_A)'}(4/3)&=&\displaystyle 4\zeta(2)\zeta(3)\zeta(2/3)\prod_p\Bigl(1-\frac1{p}\Bigr)^{-1}\Bigl(1-\frac1{p^{2/3}}\Bigr)\zeta_p(\nu(\phi_p|_A)',-2/3)
\\[.2in]&=&\displaystyle
4\tau(\SL_3(\Q))\zeta(2/3)\prod_p\Bigl(1-\frac1{p^{2/3}}\Bigr)\sum_{(Q,C,r)}\frac{\phi(Q,C,r)|\Delta(Q)|_p}{\#\Aut(Q,C,r)}I_p'((Q,C,r),-2/3).
\end{array}
\end{equation}
From Proposition \ref{prop:local_int_at_inf}, we have $\mathcal Q_i=\mathcal M_i$ and $\mathcal Q_i'=\mathcal M_i'$. Therefore, the second claim of the theorem follows from \eqref{eq:final_res_comp_1}, \eqref{eq:final_res_comp_2}, and \eqref{eq:final_res_comp_3}, along with the fact that the Tamagawa number of $\SL_3(\Q)$ is $1$.
\end{proof}

Finally, we prove Theorem \ref{thm:mainS4rings}.

\medskip

\noindent{\bf Proof of Theorem \ref{thm:mainS4rings}:} Let $\Lambda$ be a finite $S_4$-collection of local specifications for quartic rings, and assume without loss of generality that $\Lambda_\infty$ is a singleton set containing the algebra $\R^{4-2i}\times \C^i$ for some $i\in\{0,1,2\}$. Due to Bhargava's parametrization, we have a bijection between the set $R(\Lambda)$ and $G(\Z)$-orbits on the following subset $L$ of $V(\Z)$:
\begin{equation*}
L:=V(\Z)^{(i)}\bigcap \cap_{p\in S} L_p,
\end{equation*}
where $L_p\subset V(\Z_p)$ is the set of elements (defined modulo some power of $p$) corresponding to some triple in $\Lambda_p$. Let $\chi_L$ denote the characteristic function of $L$ and note the formal equality $N_\Lambda(\psi,X)=N_\psi(\chi_L,X)$. We know from Theorem \ref{th:quartic_szf_residues} that  $\xi_i(\chi_L,s)=\xi_{i,L}(s)$ have simple poles at $s=1$ and $s=5/6$. The result now follows by applying Theorems \ref{th:countingbyzeta} and \ref{thm:residues_final}, and noting that
$A_i=|\Aut(\R^{4-2i}\times\C^i)|$.

\section{Explicit computation for fields}

We will compute the integrals $I_p((Q,C,r),-2/3)$ and $I_p((Q,C,r),-2/3)$, summed up over triples with each splitting type. We note that $I_p((Q,C,r),-2/3)$ (without the $\epsilon_p$ term) only depends on the cubic resolvent $C$, and nothing more! So we begin with that.

\subsection{Counting Norms in Cubic Resolvents}

The following result will be convenient for us:

\begin{lemma}\label{lem:isotropicifhasse}
    Let $p>2$ be prime, and let $x_p\in S_3(\Z_p)$ be an element with nonzero discriminant. Then $\epsilon_p(x_p)=1$ if and only if $x_p$ is isotropic.
\end{lemma}

\begin{proof}
    Since $p$ is odd we may diagonalize $x_p$ as $(a_1,a_2,a_3)$. Now scaling the form by a constant $c$ leaves $\epsilon_p(x_p)$ unaltered, as $(ac,ac)_p=(a,a)_p(c,c)_p$ and $(c,c)_p^6=1$. We may therefore assume that $a_1\in \Z_p^{\times}$. Moreover since scaling $a_i$ by squares doesn't affect either the Hasse symbol or the isotropicity, we may assume either that all the $a_i$ are units, or that exactly 2 of them are and the third has valuation 1.

    If they are all units then all the Hilbert symbols are 1, so the Hasse symbol is 1, and likewise the form is isotropic since the form is smooth and any 3-variable form over $\F_p$ is isotropic.

    Assume now that $a_1,a_2$ are units and $a_3$ has valuation 1. Now $x_p$ is isotropic if and only if $-a_1a_2$ is a square. This is equivalent to $(p,-a_1a_2)_p=1$. On the other hand, we write $a_3=pb_3$ and use the fact the the Hilbert symbols of two units vanishes to write
    $$\epsilon_p(x_p)=(a_3,a_1a_2a_3)_p=(pb_3,pa_1a_2b_3)_p=(p,pa_1a_2)_p=(p,-a_1a_2)_p$$ where the last equality follows from $(p,-p)_p=1$. The claim is proven.
\end{proof}

For a cubic ring $C$ let $f_C:C/\Z_p\ra \Z_p$ denote the corresponding binary cubic form. For such an $f$ and an integer $m\geq 0$, we let $D(f,m)$ denote the measure of $\ol{x}$ in $(C/\Z_p)^{\vee}_{\prim}$ of entries such that $|f(x)|=p^{-m}$ for any lift $x$ of $\ol{x}$. 

For the below, we break up the $(2^2)$ splitting type by the Galois groups of the quartic field, and the $(1^21^2)$ splitting type by whether the two quadratic fields are isomorphic or not.

\begin{theorem}\label{thm:sptypeptcounts}
    Assume $p>3$. For each splitting type of quartic ring, the cubic resolvent rings and their associated cubic forms that show up are as follows. Here $\pi$ denotes a uniformizer of $\Z_p$, $C_{\max}$ denotes the maximal order containing $C$, and $\epsilon\in\{1,\dots,p-1\}$ is a quadratic non-residue.
\vspace{0.2in}
\begin{table}[h]
\centering
\hspace*{-0.5cm}
 \begin{tabular}{|c|c|c|c|c|}
  \hline
  \textbf{\textrm{Splitting Type}} & \textbf{ $C_{\max}$} & \textbf{f(a,b)} & \textbf{$D(f,0)$}&  \textbf{$D(f,m), m\geq 1$ } \\
  \hline
  (1111),(22) & (111) & $ab(a+b)$ & $\frac{(p-2)(p-1)}{p^2}$ & $3(p-1)^2p^{-m-2}$ \\
  \hline
  (112), (4) & (12) & $a(a^2-\epsilon b^2)$  & $\frac{p-1}{p}$ &  $(p-1)^2p^{-m-2}$\\
  \hline
   (13) & (3)&  $\bullet$   & $\frac{p^2-1}{p^2}$ &  0\\
   \hline
   $(1^211),(1^22)$ & $(1^21)$ & $a(b^2-\pi a^2)$ & $\frac{(p-1)^2}{p^2}$ & $(p-1)^2p^{-m-2} +\frac{p-1}{p^2}\cdot\delta_{m=1}$\\
   \hline
   $(1^31)$ & $(1^3)$ & $a^3-\pi$ & $\frac{p-1}{p}$ & $\frac{p-1}{p^2}\cdot\delta_{m=1}$\\
   \hline
   $(2^2)_{C_2^2},(1^21^2)_{=}$ & (111) & $ab(a+pb)$ & $\frac{(p-1)^2}{p^2}$ & $(p-1)^2p^{-m-2} + \frac{(p-1)(p-2)}{p^3}\delta_{m=2}+2(p-1)^2p^{-m-1}\delta_{m\geq 3}$\\
   \hline
      $(2^2)_{C_4},(1^21^2)_{\neq}$ & $(12)$ & $a(p^2a^2-\epsilon b^2)$ & $\frac{(p-1)^2}{p^2}$ & $(p-1)^2p^{-m-2} + \frac{p-1}{p^2}\delta_{m=2}$ \\
  \hline
   $(1^4)$ & $(1^21)$ & $a(a^2-\pi b^2)$ & $\frac{p-1}{p}$ & $(p-1)^2p^{-m-2}\delta_{m\geq 2}$ \\
   \hline
\end{tabular}
\end{table}
\end{theorem}

\begin{proof}
We first verify the resolvent cubic algebras are as stated. Indeed, the group theory suffices to figure out at least the dimensions of the fields involved: recall that the cubic algebra is composed of the three points obtained as the pairwise intersection of the lines connecting the 4 points giving the quartic algebra.

For the unramified splitting types the proof is straightforward as there is a unique unramified field of each degree.
For the case of $(1^211),(1^22)$ types note that the discriminant being $p$ uniquely determines that $C=C_{\max}$ is of type $(1^21)$. For the $(1^31)$ case the cubic resolvent algebra is easily seen to be a field. Moreover the discriminant being $p^2$ implies that $C$ is of index $p$ in $C_{\max}$. This implies the field is $(1^3)$ as opposed to $(3)$. 
The cases of $(2^2)_{C_2^2}$ and $(1^21^2)_{=}$ are likewise to give cubic algebras of type $(111)$. The cases of  $(2^2)_{C_4}$ and $(1^21^2)_{\neq}$ give a cubic algebra $\Q_p\oplus K$, but the discriminant is $p^2$ and so $K/\Q_p$ must be unramified. 
Finally, $(1^4)$ has discriminant $p^3$, which means $[C_{\max}:C]=p$ an $C_{\max}$ has discriminant $p$, forcing the algebra to be $(1^21)$.
\vspace{0.1in}

Next, we determine $C$ itself. Now only in the overramified cases - the last 3 rows - does $C\neq C_{\max}$. In all these cases we have $[C_{\max}:C]=p$, and the maximal orders of $(12)$ and $(111)$ have a unique subring of index $p$ up to isomorphism, since the same is true for $\F_p^3$ and $\F_p\oplus\F_{p^2}$. 
It remains to consider the case of $(1^4)$. Now in this case the reduction mod $p$ corresponds to $\F_p[t]/(t^4)$. Embedded as a subscheme of $\P^2$, the quadrics cutting this out must be a (unique) double line $2\ell$, and we claim no other singular conics. Indeed, these would be a pair of lines, and any linear function besides that defining $\ell$ generates the ideal $(t)$. 
Thus, the cubic form mod $p$ must have a single root. 
Finally, note that $\Z_p[\sqrt{\pi}]\oplus \Z_p$ has two subrings, only one of which corresponds to a binary cubic form with a single root.  

This determines the rings $C$ and a straightforward computation yields $f$. Finally, the computations of $D(f,m)$ are a straightforward  exercise using Hensel's lemma. 
\end{proof}

Now the above table reduces computing the integral without $\epsilon_p$ to a simple geometric sum. To deal with the $\epsilon_p$-integral we subdivide our cases and we have to analyze when $r_x$ is isotropic versus not. We deal with the cases separately.

\subsection{Non-overramified cases}

We handle the first 6 cases of the table above. In these cases the cubic resolvent ring is maximal. It follows that the $\F_p$-rank of any $r_x$ is at least 2.  If the $\F_p$-rank of $A$ is $3$ (equivalently, $p\nmid\det(A)$), then $A$ is automatically isotropic, and we have $\epsilon_p(A)=1$. If the $\F_p$-rank of $A$ is $2$, then the conic cut out by $A$ factors into two lines either defined over $\F_p$ or conjugate over $\F_p$. In the former case, we say that $A$ is {\it residually hyperbolic} and define $\kappa_p(A)=1$. In the latter case, we say that $A$ is not residually hyperbolic and define $\kappa_p(A)=-1$. Note that $\kappa(A)$ only depends on $A$ mod $p$. The Hasse symbol $\epsilon_p(A)$ depends only on $\kappa(A)$ and $\det(A)$. Specifically, we have $\epsilon_p(A)=-1$ if and only if $\kappa_p(A)=-1$ and $p^k\parallel\det(A)$ for $k$ odd.
Next we consider the possibilities for how many roots of the cubic resolvent give $A$ that are residually hyperbolic and how many give $A$ that are not:\footnote{This analysis is present, though in slightly different language in \cite{BSHMC}.}
\begin{itemize}
    \item (1111): three roots, all residually hyperbolic;
    \item (22): one residually hyperbolic, two not;
    \item (112): one residually hyperbolic;
    \item (4): one not residually hyperbolic;
    \item (13): no roots;
    \item $(1^211)$: two residually hyperbolic (the tangent vector of the double point is defined over $\F_p)$;
    \item $(1^22)$: one residually hyperbolic, one not. The root that is residually hyperbolic is the single root;
    \item $(1^31)$: one residually hyperbolic (the line going through the two points, and the tangent line to the multiplicity point).
\end{itemize}

The integrals are now trivial to read off from the table. Here they are:

\begin{table}[h]
\centering
\hspace*{-0.5cm}
 \begin{tabular}{|c|c|c|c|}
  \hline
  \textbf{\textrm{Type}} & $\displaystyle\sum_{(Q,C,r)\in\Sigma^{\max}_p}\frac{|\Delta(Q)|_p}{\#\Aut(Q,C,r)}$ &   $\int_{x\in (C/\Z_p)^{\vee}_{\prim}}  |\det{r_x}|_p^{s}$ & $\int_{x\in (C/\Z_p)^{\vee}_{\prim}}  \epsilon_p(r_x)|\det{r_x}|_p^{s}$  \\
  \hline
  (1111)&$\frac{1}{24}$&$\frac{(p-2)(p-1)}{p^2}+\frac{3(p-1)^2}{p^2}\cdot\frac{p^{-1-s}}{1-p^{-1-s}}$&$\frac{(p-2)(p-1)}{p^2}+\frac{3(p-1)^2}{p^2}\cdot\frac{p^{-1-s}}{1-p^{-1-s}}$ \\
    \hline
    (22)&$\frac18$&$\frac{(p-2)(p-1)}{p^2}+\frac{3(p-1)^2}{p^2}\cdot\frac{p^{-1-s}}{1-p^{-1-s}}$&$\frac{(p-2)(p-1)}{p^2}+\frac{(p-1)^2}{p^2}\cdot\left(\frac{p^{-1-s}}{1-p^{-1-s}}-\frac{2p^{-1-s}}{1+p^{-1-s}}\right)$\\
\hline
  (112)&$\frac14$&$\frac{p-1}{p} + \frac{(p-1)^2}{p^2}\cdot\frac{p^{-1-s}}{1-p^{-1-s}}$&$\frac{p-1}{p} + \frac{(p-1)^2}{p^2}\cdot\frac{p^{-1-s}}{1-p^{-1-s}}$\\
    \hline
   (4)&$\frac14$&$\frac{p-1}{p} + \frac{(p-1)^2}{p^2}\cdot\frac{p^{-1-s}}{1-p^{-1-s}}$&$\frac{p-1}{p} - \frac{(p-1)^2}{p^2}\cdot\frac{p^{-1-s}}{1+p^{-1-s}}$\\
   \hline
   (13)&$\frac13$&$\frac{p^2-1}{p^2}$&$\frac{p^2-1}{p^2}$\\
   \hline
   $(1^211)$&$\frac{1}{2p}$ &$\frac{(p-1)^2}{p^2}+\frac{(p-1)^2}{p^2}\cdot\frac{p^{-1-s}}{1-p^{-1-s}}+p^{-s}\cdot\frac{p-1}{p^2}$&$\frac{(p-1)^2}{p^2}+\frac{(p-1)^2}{p^2}\cdot\frac{p^{-1-s}}{1-p^{-1-s}}+p^{-s}\cdot\frac{p-1}{p^2}$\\
   \hline
   $(1^22)$&$\frac{1}{2p}$ &$\frac{(p-1)^2}{p^2}+\frac{(p-1)^2}{p^2}\cdot\frac{p^{-1-s}}{1-p^{-1-s}}+p^{-s}\cdot\frac{p-1}{p^2}$&$\frac{(p-1)^2}{p^2}+\frac{(p-1)^2}{p^2}\cdot\frac{p^{-1-s}}{1-p^{-1-s}}-p^{-s}\cdot\frac{p-1}{p^2}$\\
   \hline
   $(1^31)$& $\frac{1}{p^2}$& $\frac{p-1}{p}+p^{-s}\cdot \frac{p-1}{p^2}$&$\frac{p-1}{p}+p^{-s}\cdot \frac{p-1}{p^2}$ \\
   \hline
\end{tabular}
\end{table}

\subsection{The over-ramified case: $(2^2)$}

Note that we will handle both the $C_2^2$ and the $C_4$ cases at once. In both cases the corresponding cubic.
The discriminant of these rings is $p^2$ and the automorphism groups are of size 4. In both cases, the form $\det{r_x}$ mod $p$ has a double root $\alpha_2$ and a single root $\alpha_1$. 
From Theorem \ref{thm:sptypeptcounts} we see that adding up over the two cases we get

$$\frac12\sum_{(2^2)}\int_{x\in (C/\Z_p)^{\vee}_{\prim}}  |\det{r_x}|_p^{s} = \frac{(p-1)^2}{p^2} + \frac{(p-1)^2}{p^2}\cdot \frac{p^{-1-s}}{1-p^{-1-s}} + \frac{(p-1)^2}{p^3}p^{-2s}+ \frac{(p-1)^2}{p^4}\cdot\frac{p^{-3s}}{1-p^{-1-s}}$$
where the first term comes from the non-roots, the second term comes from the pre-image of $\alpha_1$, and the next 2 comes from the pre-image of $\alpha_2$.

We now understand the $\epsilon_p$ portion. If $x$ does not  reduce to either of these, then $\epsilon(r_x)=1$. 
Consider next consider the case where $x$ reduces to $\alpha_1$.  In this case $\epsilon_p(r_x)=1$ iff either $v_p(|\det r_x|)$ is even, or $r_x$ is residually hyperbolic.  In this case we claim that $r_x$ is never residually hyperbolic. This is because the quadratic $\phi_{\alpha_1}$ must define a double line, and $\phi_{\alpha_2}$ must define a product of lines, which are defined over $\F_{p^2}$ and conjugate. So for this contribution we change 
$$\frac{(p-1)^2}{p^2}\cdot \frac{p^{-1-s}}{1-p^{-1-s}}\longrightarrow \frac{(p-1)^2}{p^2}\cdot \frac{-p^{-1-s}}{1+p^{-1-s}}$$

We finally deal with the case where $x$ reduces to a double root $\alpha_2$. To handle this case we go back to working with $V(\Z_p)$. Now in the notation of \cite[p.1358]{BHCL3} this contribution amounts to restricting to $(A,B)\in U_p(2^2)$ such that $A$ modulo $p$ defines a double line. By the argument at the end of that page the $p$-adic measure of possible $B$'s is uniform among all such $A$. 

\begin{lemma}\label{lem:type2^2isotropy}
\begin{enumerate}
    \item Among $A$'s belonging to a $U_p(2^2)$ pair with $p^2 || \det A$, the ratio of isotropic $A$'s to non-isotropic $A$'s is $p+1 : p-1$.
    \item Among $A$'s belonging to a $U_p(2^2)$ pair with $p^3 | \det A$, the ratio of isotropic $A$'s to non-isotropic $A$'s is 1:1
\end{enumerate} 
\end{lemma}

\begin{proof}
Note that the condition on $A$ is the following: let $L=\ker A_{\F_p}$. Note that $\ell\ra A[\tilde{\ell}]/p\in \F_p$ gives a well-defined quadratic form on $L$, where $\tilde{\ell}$ denotes any lift of $\ell$. We call this form $A'$. We must have $A'\neq 0$. 

Now if $p^2|| A,$ then $A'$ is non-degenerate and $A$ being isotropic is equivalent to $A'$ being isotropic. The proportion of degenerate isotropic to non-isotropic $2\times 2$ matrices is $p+1:p-1$ , proving part (1).

If $p^{k+1}|| A$ with $k>1$, then the diagonal coefficients of $A$ must have valuation $0,1,k$. We may pick a basis $v_1,v_2,v_3$ on which $A$ is diagonal with valuations $0,1,k$. Now $\det(A+Xp^k) \equiv \det(A)+p^{k}A[v_1]A[v_2] X[v_3]$. Therefore within each $p^k$ coset the p-adic densities of determinants whose square class is $p^k$ vs $\epsilon p^k$ is the same. This proves part (2). 
\end{proof}

By the above, we see that

$$\frac12\sum_{(2^2)}\int_{x\in (C/\Z_p)^{\vee}_{\prim}}  \epsilon_p(r_x)|\det{r_x}|_p^{s} = \frac{(p-1)^2}{p^2} - \frac{(p-1)^2}{p^2}\cdot \frac{p^{-1-s}}{1+p^{-1-s}} + \frac{(p-1)^2}{p^4}p^{-2s}$$

Hence, we see that

\begin{align*}
&\displaystyle\sum_{(Q,C,r)\in\Sigma_p^{(2^2)}}\frac{|\Delta(Q)|_p}{\#\Aut(Q,C,r)}\int_{x\in (C/\Z_p)^{\vee}_{\prim}}  |\det{r_x}|_p^{s}\\&=\frac{1}{4p^2}\left(\frac{(p-1)^2}{p^2} + \frac{(p-1)^2}{p^2}\cdot \frac{p^{-1-s}}{1-p^{-1-s}} + \frac{(p-1)^2}{p^3}p^{-2s}+ \frac{(p-1)^2}{p^4}\cdot\frac{p^{-3s}}{1-p^{-1-s}}\right)\\
&\\ &\displaystyle\sum_{(Q,C,r)\in\Sigma_p^{(2^2)}}\frac{|\Delta(Q)|_p}{\#\Aut(Q,C,r)}\int_{x\in (C/\Z_p)^{\vee}_{\prim}}  \epsilon_p(x)|\det{r_x}|_p^{s}\\ &= \frac{1}{4p^2}\left(\frac{(p-1)^2}{p^2} + \frac{(p-1)^2}{p^2}\cdot \frac{p^{-1-s}}{1-p^{-1-s}} + \frac{(p-1)^2}{p^4}p^{-2s}\right).
\end{align*}

\subsection{The over-ramified case: $(1^21^2)$}

We again deal with the two cases $(1^21^2)_=,(1^21^2)_{\neq}$ at once. Note there are 2 fields corresponding to $(1^21^2)_=$ but the Automorphism group is twice as big, hence the measures of the two cases are the same. 

Everything proceeds analogulsy to the $(2^2)$ case, except for the following:

\begin{lemma}\label{lem:type1^21^2isotropy}
\begin{enumerate}
    \item Among $A$'s belonging to a $U_p(1^21^2)$ pair with $p^2 || \det A$, the ratio of isotropic $A$'s to non-isotropic $A$'s is $p-2 : p$.
    \item Among $A$'s belonging to a $U_p(1^21^2)$ pair with $p^3 | \det A$, the ratio of isotropic $A$'s to non-isotropic $A$'s is 1:1
\end{enumerate} 
\end{lemma}

\begin{proof}
    We let $L=\ker A_{\F_p}$ and $A'$ the induced quadratic form on $L$. 
    In the case where $p^2 || \det A$ we must have $A'$ be non-degenerate. 
    Now the condition on $B_{\F_p}$ is that it has no non-trivial zeroes on $L$. 
    Note that $B_{\F_p}$ has two zeroes on the zero set of $A_{\F_p}$, which we call $P_1,P_2$. Thus $A'$ must not have zeroes on $P_1,P_2$. The ratio we seek is therefore: within the set of non-degenerate quadratic forms that don't vanish on 2 given points, how many are hyperbolic vs not?

    Note the hyperbolic forms have 2 zeroes, so the porportion of those that don't vanish on $P_1,P_2$ is $\frac{\binom{p-1}{2}}{\binom{p+1}{2}}$, and so our final answer is 

    $$\frac{p+1}{p-1}\cdot \frac{\binom{p-1}{2}}{\binom{p+1}{2}} = \frac{p-2}p$$ proving the first part.

    For the second part, we now assume $p^{k+1}|| \det A$ with $k\geq 2$. Now we once again break up into $p^k$ cosets as before.
\end{proof}

Given this lemma, we can now finish the computation in exactly the same way as the previous section, just taking into account the ratio $p-2:p$ instead of $p+1:p-1$, which means the $\epsilon_p$ average will give $-\frac{1}{p-1}$ instead of $\frac{1}{p}$. Thus we obtain:

\begin{align*}
&\displaystyle\sum_{(Q,C,r)\in\Sigma_p^{(1^21^2)}}\frac{|\Delta(Q)|_p}{\#\Aut(Q,C,r)}\int_{x\in (C/\Z_p)^{\vee}_{\prim}}  |\det{r_x}|_p^{s}\\&=\frac{1}{4p^2}\left(\frac{(p-1)^2}{p^2} + \frac{(p-1)^2}{p^2}\cdot \frac{p^{-1-s}}{1-p^{-1-s}} + \frac{(p-1)^2}{p^3}p^{-2s}+ \frac{(p-1)^2}{p^4}\cdot\frac{p^{-3s}}{1-p^{-1-s}}\right)\times\\
 &\displaystyle\sum_{(Q,C,r)\in\Sigma_p^{(1^21^2)}}\frac{|\Delta(Q)|_p}{\#\Aut(Q,C,r)}\int_{x\in (C/\Z_p)^{\vee}_{\prim}}  \epsilon_p(x)|\det{r_x}|_p^{s}\\ &= \frac{1}{4p^2}\left(\frac{(p-1)^2}{p^2} + \frac{(p-1)^2}{p^2}\cdot \frac{p^{-1-s}}{1-p^{-1-s}} - \frac{p-1}{p^3}p^{-2s}\right).
\end{align*}

\subsection{The over-ramified case: $(1^4)$}

Note that here we have $$\sum_{(Q,C,r)\in\Sigma^{(1^21^2)}_p}\frac{|\Delta(Q)|_p}{|\Aut(Q)|}=\frac{1}{2p^2}=\frac{1}{p^3}.$$
In this case there is only a single triple root mod $p$. 
In this case the cubic resolvent algebra has maximal order $\Z_p[\sqrt{\pi}]\oplus \Z_p$ for some uniformizer $\pi\in\Z_p$. And the resolvent cubic ring $R$ is the pre-image mod the maximal ideal of $\F_p\subset \F_p\oplus \F_p$. Hence $\det r_x$ has only a single triple root mod $p$. Moreover, the density of elements with valuation determinant $p^n$ is $\frac{p}{p+1}$ for $n=0$, $0$ for $n=1$, and $\frac{p^{1-n}(p-1)}{p+1}$ for $n\geq 2$. 

So consider the case where $x$ reduces to the triple root.

\begin{lemma}\label{lem:type1^4isotropy}
\begin{enumerate}
    \item Among $A$'s belonging to a $U_p(1^4)$ pair with $p^2 || \det A$, the ratio of isotropic $A$'s to non-isotropic $A$'s is $1 : 1$.
    \item Among $A$'s belonging to a $U_p(1^4)$ pair with $p^3 | \det A$, the ratio of isotropic $A$'s to non-isotropic $A$'s is 1:1
\end{enumerate} 
\end{lemma}

\begin{proof}
     We let $L=\ker A_{\F_p}$ and $A'$ the induced quadratic form on $L$. 
    In the case where $p^2 || \det A$ we must have $A'$ be non-degenerate. 
    Now the condition on $B_{\F_p}$ is that it has no non-trivial zeroes on $L$. 
    Note that $B_{\F_p}$ has exactly one zero on the zero set of $A_{\F_p}$, which we call $P$. Thus $A'$ must not vanish on $P$. The ratio we seek is therefore: within the set of non-degenerate quadratic forms that don't vanish on 1 given point, how many are hyperbolic vs not?

    Note the hyperbolic forms have 2 zeroes, so the proportion of those that don't vanish on $P$ is $\frac{\binom{p}{2}}{\binom{p+1}{2}}$, and so our final answer is 

    $$\frac{p+1}{p-1}\cdot \frac{\binom{p}{2}}{\binom{p+1}{2}} = 1$$ proving the first part.

    For the second part, we now assume $p^{k+1}|| \det A$ with $k\geq 2$. Now we once again break up into $p^k$ cosets as before.
\end{proof}

Finally, we obtain the following:

\begin{align*}
&\displaystyle\sum_{(Q,C,r)\in\Sigma_p^{(1^4)}}\frac{|\Delta(Q)|_p}{\#\Aut(Q,C,r)}\int_{x\in (C/\Z_p)^{\vee}_{\prim}}  |\det{r_x}|_p^{s}=\frac{1}{p^3}\left(\frac{p-1}{p} + \frac{(p-1)^2}{p^2}\cdot \frac{p^{-2-2s}}{1-p^{-1-s}} \right)\\
&\\ &\displaystyle\sum_{(Q,C,r)\in\Sigma_p^{(1^4)}}\frac{|\Delta(Q)|_p}{\#\Aut(Q,C,r)}\int_{x\in (C/\Z_p)^{\vee}_{\prim}}  \epsilon_p(x)|\det{r_x}|_p^{s}= \frac{p-1}{p^4}
\end{align*}

\subsection{Summary of secondary densities}\label{sec:2ndmaintermvalues}

\resizebox{\textwidth}{!}{
 \begin{tabular}{|c|c|c|}
  \hline
  \textbf{\textrm{Type}} & $\displaystyle\sum_{(Q,C,r)\in\Sigma^{\max}_p}\frac{|\Delta(Q)|_p}{\#\Aut(Q,C,r)}\cdot\int_{x\in (C/\Z_p)^{\vee}_{\prim}}  |\det{r_x}|_p^{s}$ & $\displaystyle\sum_{(Q,C,r)\in\Sigma^{\max}_p}\frac{|\Delta(Q)|_p}{\#\Aut(Q,C,r)}\int_{x\in (C/\Z_p)^{\vee}_{\prim}}  \epsilon_p(r_x)|\det{r_x}|_p^{s}$  \\
  \hline
  (1111)&$\frac{1}{24}\left(\frac{(p-2)(p-1)}{p^2}+\frac{3(p-1)^2}{p^2}\cdot\frac{p^{-1-s}}{1-p^{-1-s}}\right)$&$\frac{1}{24}\left(\frac{(p-2)(p-1)}{p^2}+\frac{3(p-1)^2}{p^2}\cdot\frac{p^{-1-s}}{1-p^{-1-s}}\right)$ \\
    \hline
    (22)&$\frac18\cdot\left(\frac{(p-2)(p-1)}{p^2}+\frac{3(p-1)^2}{p^2}\cdot\frac{p^{-1-s}}{1-p^{-1-s}}\right)$&$\frac18\cdot\left(\frac{(p-2)(p-1)}{p^2}+\frac{(p-1)^2}{p^2}\cdot\left(\frac{p^{-1-s}}{1-p^{-1-s}}-\frac{2p^{-1-s}}{1+p^{-1-s}}\right)\right)$\\
\hline
  (112)&$\frac14\cdot\left(\frac{p-1}{p} + \frac{(p-1)^2}{p^2}\cdot\frac{p^{-1-s}}{1-p^{-1-s}}\right)$&$\frac14\cdot\left(\frac{p-1}{p} + \frac{(p-1)^2}{p^2}\cdot\frac{p^{-1-s}}{1-p^{-1-s}}\right)$\\
    \hline
   (4)&$\frac14\cdot\left(\frac{p-1}{p} + \frac{(p-1)^2}{p^2}\cdot\frac{p^{-1-s}}{1-p^{-1-s}}\right)$&$\frac14\cdot\left(\frac{p-1}{p} - \frac{(p-1)^2}{p^2}\cdot\frac{p^{-1-s}}{1+p^{-1-s}}\right)$\\
   \hline
   (13)&$\frac13\cdot\frac{p^2-1}{p^2}$&$\frac13\cdot\frac{p^2-1}{p^2}$\\
   \hline
   $(1^211)$&$\frac{1}{2p}\cdot\left(\frac{(p-1)^2}{p^2}+\frac{(p-1)^2}{p^2}\cdot\frac{p^{-1-s}}{1-p^{-1-s}}+p^{-s}\cdot\frac{p-1}{p^2}\right)$&$\frac{1}{2p}\cdot\left(\frac{(p-1)^2}{p^2}+\frac{(p-1)^2}{p^2}\cdot\frac{p^{-1-s}}{1-p^{-1-s}}+p^{-s}\cdot\frac{p-1}{p^2}\right)$\\
   \hline
   $(1^22)$&$\frac{1}{2p}\cdot\left(\frac{(p-1)^2}{p^2}+\frac{(p-1)^2}{p^2}\cdot\frac{p^{-1-s}}{1-p^{-1-s}}+p^{-s}\cdot\frac{p-1}{p^2}\right)$&$\frac{1}{2p}\cdot\left(\frac{(p-1)^2}{p^2}+\frac{(p-1)^2}{p^2}\cdot\frac{p^{-1-s}}{1-p^{-1-s}}-p^{-s}\cdot\frac{p-1}{p^2}\right)$\\
   \hline
   $(1^31)$& $\frac{1}{p^2}\cdot\left(\frac{p-1}{p}+p^{-s}\cdot \frac{p-1}{p^2}\right)$&$\frac{1}{p^2}\cdot\left(\frac{p-1}{p}+p^{-s}\cdot \frac{p-1}{p^2}\right)$ \\
   \hline
    $(2^2)$& $\frac{1}{4p^2}\left(\frac{(p-1)^2}{p^2} + \frac{(p-1)^2}{p^2}\cdot \frac{p^{-1-s}}{1-p^{-1-s}} + \frac{(p-1)^2}{p^3}p^{-2s}+ \frac{(p-1)^2}{p^4}\cdot\frac{p^{-3s}}{1-p^{-1-s}}\right)$&$\frac{1}{4p^2}\left(\frac{(p-1)^2}{p^2} + \frac{(p-1)^2}{p^2}\cdot \frac{p^{-1-s}}{1-p^{-1-s}} + \frac{(p-1)^2}{p^4}p^{-2s}\right)$ \\
   \hline
    $(1^21^2)$& $\frac{1}{4p^2}\left(\frac{(p-1)^2}{p^2} + \frac{(p-1)^2}{p^2}\cdot \frac{p^{-1-s}}{1-p^{-1-s}} + \frac{(p-1)^2}{p^3}p^{-2s}+ \frac{(p-1)^2}{p^4}\cdot\frac{p^{-3s}}{1-p^{-1-s}}\right)$&$\frac{1}{4p^2}\left(\frac{(p-1)^2}{p^2} + \frac{(p-1)^2}{p^2}\cdot \frac{p^{-1-s}}{1-p^{-1-s}} - \frac{p-1}{p^3}p^{-2s}\right)$ \\
   \hline
    $(1^4)$& $\frac{1}{p^3}\left(\frac{p-1}{p} + \frac{(p-1)^2}{p^2}\cdot \frac{p^{-2-2s}}{1-p^{-1-s}} \right)$&$\frac{p-1}{p^4}$ \\
   \hline
\end{tabular}
}

\part{The Sieve}

\section{Nonmaximal quartic rings and switching correspondences}\label{sec:switching}

A pair $(Q,C)$, where $Q$ is a quartic ring with nonzero discriminant and $R$ is a cubic resolvent of $Q$, is said to be {\it non-maximal} (resp.\ {\it non-maximal at $p$}) if the index of $Q$ in its maximal order is greater than~$1$ (resp.\ a multiple of $p$). A pair $(A,B)\in V(\Z)$ with $\Delta(A,B)\neq 0$, corresponding to the triple $(Q,C,r)$ is said to be {\it non-maximal} (resp.\ {\it non-maximal at $p$}) if $(Q,C)$ is non-maximal (resp.\ non-maximal at $p$). Otherwise, we say that $(Q,C)$ and $(A,B)$ is {\it maximal} (resp.\ {\it maximal at $p$}). Suppose $(Q,C)$ is nonmaximal, $Q'$ is an overring of $Q$, and $R'$ is a resolvent ring of $Q'$. Then we say that $(Q,C)$ is a {\it subpair} of $(Q',C')$ and that $(Q',C')$ is a {\it overpair} of $(Q,C)$.
In this section, we present Bhargava's description of nonmaximal elements in $V(\Z)$, and then describe a ``switching correspondence'' which allows us relate the $G(\Z)$-orbit of $(A,B)$ corresponding to $(Q,R,r)$ to $G(\Z)$-orbits corresponding to subpairs and overpairs of $(Q,C)$.

\subsection{Nonmaximality in $V(\Z)$}

We have the following result due to Bhargava.
\begin{lemma}\cite[Lemma 22]{BHCL3}\label{lem:nonmax_Manjul}
If $Q$ is any quartic ring that is not maximal at $p$, then there exists a $\Z$-basis $1,\alpha_1,\alpha_2,\alpha_3$ of $Q$ such that at least one of the following is true.
\begin{itemize}
\item[{\rm (i)}] $\Z+\Z\cdot(\alpha_1/p)+\Z\cdot\alpha_2+\Z\cdot\alpha_3$ forms a ring;
\item[{\rm (ii)}] $\Z+\Z\cdot(\alpha_1/p)+\Z\cdot(\alpha_2/p)+\Z\cdot\alpha_3$ forms a ring;
\item[{\rm (iii)}] $\Z+\Z\cdot(\alpha_1/p)+\Z\cdot(\alpha_2/p)+\Z\cdot(\alpha_3/p)$ forms a ring.
\end{itemize}
\end{lemma}

Let $(A,B)$ be an element in $V(\Z)$ which is nonmaximal at $p$. The following points are noted in the discussion following the proof of \cite[Lemma 22]{BHCL3}. Assume that we are not in Case (iii) of the above lemma. Then Case (i) occurs if and only if $(A,B)$ can be transformed via an element of $G(\Z)$ so that its coefficients satisfy the following condition (see \cite[(43)]{BHCL3}:
\begin{equation}\label{eq:SW_Condition_1}
a_{11}\equiv b_{12}\equiv b_{13}\equiv 0\pmod{p},\mbox{ and }
b_{11}\equiv 0 \pmod{p^2}.
\end{equation}
Case (ii) occurs if and only if $(A,B)$ can be transformed by an element of $G(\Z)$ so that one of the following two conditions are satisfied (see (a) and (b) just following \cite[(43)]{BHCL3}):
\begin{equation}\label{eq:SW_Condition_2}
a_{11}\equiv a_{12}\equiv a_{22}\equiv b_{11}\equiv b_{12}\equiv b_{22}
\equiv 0 \pmod{p};
\end{equation}
\begin{equation}\label{eq:SW_Condition_2'}
b_{11}\equiv b_{12}\equiv b_{22}\equiv 0\pmod{p^2},\mbox{ and }
b_{13}\equiv b_{23}\equiv 0 \pmod{p}.
\end{equation}
Finally, Case (iii) occurs if and only if the $\F_p$ span of $A$ and $B$ is 1-dimensional, which implies that $(A,B)$ can be transformed by an element of $G(\Z)$ to ensure that we have
\begin{equation}\label{eq:SW_Condition_3}
B\equiv 0 \pmod{p}.
\end{equation}

\subsection{The switching correspondence}

Let $\overline{V(\Z)}:=G(\Z)\backslash V(\Z)$ denote the set of $G(\Z)$-orbits on $V(\Z)$. Then $\overline{V(\Z)}$ is in bijection with the set of isomorphism classes of quadratic maps $\{(r:W\to U)\}$, where $W$ is a three-dimensional lattice, $U$ is a two-dimensional lattice. In this language, given a sub-pair $(Q',C')$ of $(Q,C)$, and corresponding maps $(r:W\ra U,r':W'\ra U')$ we may find inclusions $W'\subset W, U'\subset U$ inducing $r'$ from $r$, such that $[W:W']=[U:U']$. In this way, we may interpret the correspondences of subpairs and overpairs of a given index purely representation theoretically. We specifically isolate subpairs of index $p$ and apply this observation there. These are particularly important since - as we shall prove below - they account for the majority of non-maximal rings.

\medskip

\noindent {\bf The Type-$1$ switching correspondence}
\smallskip

\noindent Given an element $v:U\to\Z$ in $U^\vee$, we compose to get a quadratic form $q_v:W\to U\to \Z$ on $W$. Let $p$ be a fixed prime.
Let $r:W\to U$ be an element of $\overline{V(\Z)}$. Consider the set of all elements $(w,v)\in \P(W_{\F_p})\times\P(U^\vee_{\F_p})$ such that
for all $(v',w')\in U^\vee_{\F_p}\times W_{\F_p}$
\begin{equation}\label{eq:T1_conditions}
q_{v'}(w,w)\equiv  q_v(w,w')\equiv 0\mod p,\quad q_v(w,w)\equiv 0\mod{p^2}.
\end{equation}
Given such a $(w,v)$ we define the element $r':W'\to U'$, where $W'\subset W$ and $U'\subset U$ are given by $W'=w+pW$ and $U'=p(\ker v)$ and $r'$ is the restriction from $W$ to $W'$.

Conversely, suppose we are given $r':W'\ra U'$ together with a choice of 2-dimensional subspace $L\subset W'_{\F_p}$ and $v\in \P(U'^{\vee}_{\F_p})$ such that $q_v$ vanishes on $L$. 
Then we define $r:W\ra U$ by setting 
$$W=L+pW',U=\ker v$$ where $r$ is the restriction of $r'$.

Let $T_p(1)$ denote the set of pairs $(r:W\to U, r':W'\to U')$ as above, so that $T_p(1)\ra \ol{V(\Z)}\times \ol{V(\Z)}$ is an (asymmetric) correspondence. We denote the projection maps to the $\ol{V(\Z)}$ on the left and to the $\ol{V(\Z)}$ on the right by $\pi_\sub$ and $\pi_\over$, respectively. Then the following is a consequence of Lemma \ref{lem:nonmax_Manjul}:
\begin{lemma}
For $y\in T_p(1)$, we have $\pi_\sub(y)$ is an index-$p$ subpair of $\pi_\over(y)$. Moreover, if $x$ and $x'$ in $\ol{V(\Z)}$ are such that $x$ is an index-$p$ subpair of $x'$, then there exists a unique $y\in T_p(1)$ such that $\pi_\sub(y)=x$ and $\pi_\over(y)=x'$. 
\end{lemma}

We deduce the a number of implications of the above lemma. Given $x\in \ol{V(\Z)}$ corresponding to $r:W\to U$, let $M_p^1(x)$ denote the number of pairs $(w,v)\in \P(W_{\F_p})\times\P(U^\vee_{\F_p})$ satisfying the conditions of~\eqref{eq:T1_conditions}. Given an element $x\in V(\Z)$, we will have to distinguish between some types of divisibility of $\Delta(x)$ by various prime powers. Specifically, for a quartic ring $Q$ and a prime $p$, we say that $p^{2k}\mid\Delta(Q)$  {\it well} if either $Q$ has index at least $p^k$ in its maximal order, or $Q$ has index at least $p^{k-1}$ in the maximal order of an {\it overramified} quartic field (i.e., a quartic field with splitting type $(1^21^2)$, $(2^2)$, or $(1^4)$ at $p$). For a square number $q^2\geq 1$, we say that $q^2\mid\Delta(Q)$ {\it well} if every prime power $p^{2k}$ dividing $q$ divides $\Delta(Q)$ well. Note that if $p^{2k}$ divides $\Delta(Q)$ but does not divide it well, then $Q$ must be an index-$qp^{k-1}$ suborder of a $(1^3)$-maximal order, for some $q$ with $(p,q)=1$. In particular, if $p^{2k+2}\mid\Delta(Q)$, then $p^{2k}\mid\Delta(Q)$ well. Finally, for $x\in V(\Z)$, we say that $q^2\mid\Delta(x)$ well if $q^2\mid\Delta(Q)$ well, where $Q$ is the quartic order corresponding to $x$. It can easily be checked that this is equivalent to the condition that $x$ corresponds to the triple $(Q,C,r)$, where the index of $C$ is at least $q$.

We have the following proposition.
\begin{proposition}\label{prop:Mp1_prelim}
Let $x$ be an element of $\ol{V(\Z)}$. Then the following are true.
\begin{enumerate}
\item If $M_p^1(x)\geq 1$, then $x$ is nonmaximal at $p$.
\item If $M_p^1(x)> 1$, then $p^4\mid\Delta(x)$ well.
\end{enumerate}
\end{proposition}
\begin{proof}
The first claim is immediate since if $M_p^1(x)\geq 1$, then there exists $y\in T_p(1)$ with $\pi_\sub(y)=x$. Then $\pi_\over(y)$ is an overpair of $x$, implying that $x$ is nonmaximal. 
To prove the second claim, let $r:W\to U$ correspond to $x$. Then there exist two pairs $(w,v)$ and $(w',v')$ in $\P(W_{\F_p})\times\P(U^\vee_{\F_p})$ satisfying the conditions of~\eqref{eq:T1_conditions}. We divide into three cases: first, $w=w'$; second, $v=v'$; and third, $w\neq w'$ and $v\neq v'$. For the first case, we complete $w$ to a basis of $W$ and take $\langle v,v'\rangle$ to be a basis of $U$, yielding an element $(A,B)$ of the form
\begin{equation*}
\left[\begin{pmatrix}0^2&{0}&{0}\\{0}&{*}&{*}\\{0}&{*}&{*}\end{pmatrix},\begin{pmatrix} 0^2&0&0\\0&{*}&{*}\\0&{*}&{*}\end{pmatrix}\right],
\end{equation*}
where the $0$ means that the coefficient is divisible by $p$ and the $0^2$ means the coefficient is divisible by $p^2$.
Since the element $(1,\diag(p^{-1},1,1))\in G(\Q)$ leaves the above pair integral, it follows that $p^8\mid\Delta(x)$.
For the third case, we complete $w$ and $w'$ to a basis of $W$ and complete $v$ to a basis of $U$, obtaining an element $(A,B)$ of the form
\begin{equation*}
\left[\begin{pmatrix}0^2&{0}&{0}\\{0}&{0^2}&{0}\\{0}&{0}&{*}\end{pmatrix},\begin{pmatrix} 0&{*}&{*}\\{*}&{0}&{*}\\{*}&{*}&{*}\end{pmatrix}\right].
\end{equation*}
Now either $b_{12}\equiv 0\pmod{p}$, in which case $(A,B)$ satisfies the conditions of \eqref{eq:SW_Condition_2}, or $b_{12}\not\equiv 0\pmod{p}$, in which case $(A,B)$ can be transformed via an action of $\GL_2(\Z)$ to satisfy the conditions of \eqref{eq:SW_Condition_2'}. In either case, we are in Case (ii) of Lemma \ref{lem:nonmax_Manjul} implying that $p^2\mid \ind(x)$ and hence $p^4\mid\Delta(x)$ well. Finally, in the third case, we complete $\langle w,w'\rangle$ to a basis for $W$ and take $\langle v,v'\rangle$ to be a basis for $U$ yielding an element of the form
\begin{equation*}
\left[\begin{pmatrix}0^2&{0}&{0}\\{0}&{0}&{*}\\{0}&{*}&{*}\end{pmatrix},\begin{pmatrix} 0&{0}&{*}\\{0}&{0^2}&{0}\\{*}&{0}&{*}\end{pmatrix}\right].
\end{equation*}
This element satisfies the conditions of \eqref{eq:SW_Condition_2}, again implying that $p^2\mid \ind(x)$ and hence $p^4\mid\Delta(x)$ well.
\end{proof}

\section{Uniformity estimates}

In this section, we prove uniform estimates on the number of $G(\Z)$-orbits on the set of elements in $V(\Z)$ having nonzero bounded discriminant, where each orbit is weighted by certain $G(\Z)$-invariant functions.

\subsection{Uniformity estimates on Rings}

Let $K$ be an \'etale quartic algebra over $\Q$, and $p$ be a prime. We say that $p^2\mid\Delta(K)$ {\it well} if the splitting type of $p$ at $K$ is $(1^21^2)$, $(2^2)$, or $(1^4)$. For a squarefree number $q$, we say that $q^2\mid\Delta(K)$ well if $p^2\mid\Delta(K)$ well for all primes $p$ dividing $q$.
We begin with the following result.
\begin{theorem}\label{prop:unif_fields}
Let $q$ be a squarefree number. The number of \'etale quartic algebras $K$ over $\Q$ with $|\Delta(K)|<X$ and $q^2\mid\Delta(K)$ well is bounded by $O(X^{1+o(1)}/q^2)$.
\end{theorem}
\begin{proof}
    We break up the set of \'etale quartic algebras $K$ over $\Q$ into four subsets: first, the subset of fields $K$ whose Galois closures have Galois group $S_4$ or $A_4$ over $\Q$. These are exactly the algebras $K$ such that both $K$ and the cubic resolvent of $K$ are fields. The required estimates for this subset of fields follows from \cite[Proposition 23]{dodqf}. (Proposition 23 of \cite{dodpf} implies the estimate when $q$ is a prime, but the same proof carries over without change for arbitrary squarefree $q$.) Second, the subset of $D_4$- $V_4$-, and $C_4$-quartic fields $K$. These fields $K$ contain a quadratic subfield $K_2$, and we have $\Delta(K)=\Delta(K_2)^2N_{K_2/\Q}(\Delta(K/K_2))$, where $N_{K_2/\Q}$ denotes the norm function from $K_2$ to $\Q$, and $\Delta(K/K_2)$ denotes the relative discriminant. Given an integer $D$ with $|D|<X$ and $q^2\mid D$, there are only $O(|D|^{o(1)})$ choices for $\Delta(K_2)$, thereby fixing $K_2$ up to $O(|D|^{o(1)})$ choices. From \cite[Theorem 1.1]{Cohen_Diaz_Olivier}, it follows that the number of quadratic extensions of $K_2$ with relative discriminant having norm $D/\Delta(K_2)^2$ is also bounded by $O(|D|^{o(1)})$. Since there are $O(X/q^2)$ possible values of $D$, the result follows for this subset of fields. 
    
    Third, we consider the algebras $K$ which are of the form $\Q\oplus K_3$, for which the claim follows essentially from work of Davenport--Heilbron (see \cite[Lemma 2.7]{BBPEE}). Finally, we consider algebras of the form $K=K_2\oplus K_2'$, where $K_2$ and $K_2'$ are \'etale quadratic algebras over $\Q$. The set of all $K$ is parametrized by pairs of squarefree (away from $2$) integers, and the claim follows immediately by noting that both these integers must be multiples of $q$ (up to a factor of $2$).
\end{proof}

Next, we have the following result, due to Nakagawa, estimating the number of suborders of an \'etale quartic algebra over $\Q$ with fixed index.

\begin{proposition}\label{prop:Nak}
Let $K$ be a quartic \'etale algebra over $\Q$ with ring of integers $\O_K$ and discriminant $D\neq 0$. Then the number $N(\O_K,q)$ of suborders of $\O_K$ having index $q$ is $\ll q^{o(1)} N(q,D)$, where
\begin{equation}\label{eq:NakBound}
N(q,D):=\prod_{\substack{p^2\nmid D\\p^3||q}}p
\prod_{\substack{p^2\nmid D\\p^e||q,\ e\geq 4}}p^{\lfloor e/2\rfloor}
\prod_{\substack{p^2\mid D\\p^e||q,\ e\geq 2}}p^{\lfloor e/2\rfloor}
\end{equation}
\end{proposition}
Note that we have $N(q,D)=N(q,(q,D))$.
Next we prove that Nakagawa's result implies a bound of the same strength when counting sub-pairs $(Q,C)$ of given fixed index inside an \'etale quartic extension of~$\Q$.

\begin{corollary}\label{cor:Nak}
Let $K$ be an \'etale quartic algebra over $\Q$, with ring of integers $\O_K$ and discriminant $D$. Let $N_K(q)$ denote the number of pairs $(Q,C)$, where $Q$ is a quartic suborder of $\O_K$ having index $q$ and $R$ is a cubic resolvent ring of $Q$. Then $N_K(q)\ll q^{o(1)} N(q,D)$.
\end{corollary}
\begin{proof}
Let $Q$ be one of the $N(\O_K,q)$ suborders of $\O_K$ of index $q$. Bhargava proves in \cite[Corollary 4]{BHCL3} that the number of cubic resolvent rings of $Q$ is equal to the sum of the divisors of the content of $Q$, where the {\it content} of $Q$ is the largest integer $c$ such that $Q=\Z+cQ'$ for some quartic ring $Q'$. This quartic ring $Q'$ then has content $1$ and index $q/c^3$ in $\O_K$. The number of pairs $(Q',C')$, where $Q'$ is an index $q/c^3$ quartic suborder of $\O_K$ of content $1$, and $R'$ is the (unique) cubic resolvent ring of $Q'$ is bounded by $N(\O_K,q/c^3)$. Therefore, we have
\begin{equation*}
N_K(q)\ll \sum_{c^3\mid q} c^{1+o(1)} N(\O_K,q/c^3)\ll
\sum_{c^3\mid q} c^{1+o(1)} N(q/c^3,D)\ll q^{o(1)} N(q,D)
\end{equation*}
as needed.
\end{proof}

We derive the following consequence of Theorem \ref{prop:unif_fields} and Corollary \ref{cor:Nak}.
\begin{corollary}\label{cor:uniformabestimate}
Let $a$ and $b$ be coprime squarefree positive integers. Then the number of triples $(Q,C,r)$, where $Q$ is a quartic ring, $C$ is a cubic resolvent of $Q$, $0<|\Delta(Q)|<X$, and $a^2b^4\mid\Delta(Q)$ well is bounded by $O(X^{1+o(1)}/a^{2}b^{4})$.
\end{corollary}
\begin{proof}
We partition our set of triples $(Q,C,r)$ into subsets with fixed value of the index $\ind(Q,C)=\ind(Q)$. Let $n$ be a positive integer occurring as an index of some such triple $(Q,C,r)$. Then we must have $b\mid n$ (up to a possible factor of $6$, which we will ignore) since the discriminant of a quartic maximal order cannot be divisible by $p^4$ for a prime $p\geq 5$. Write $n=ba_1b_1n_1$, where $a_1\mid a$, $b_1\mid b$, and $(n_1,ab)=1$. Let $(Q,C)$ be a pair with index $n$, and denote the maximal order of $Q\otimes\Q$ by $\O$. Since $a^2b^4\mid\Delta(Q)$ well, it follows that $(a_2b_2)^2\mid\Delta(\O)$, where $a_2=a/a_1$ and $b_2=b/b_1$. Let $b_1'$ be the largest integer dividing $b_1$ such that $b_1'^2\mid\Delta(\O)$. 
Since we have $|\Delta(\O)|<X/n^2$ and $a_2^2b_2^2b_1'^2\mid\Delta(\O)$ well, Proposition~\ref{prop:unif_fields} implies that the number of choices for $\O$ is $\ll X/(n^2(a_2b_2b_1')^{2-o(1)})\ll ((ab)^{o(1)})X/(a^2b^4b_1'^2n_1^2)$.
By Corollary \ref{cor:Nak}, it follows that the number of suborders of any such $\O$ with index $n$ is $\ll b_1'n^{o(1)} N(n_1,n_1)$. Therefore, the number of pairs $(Q,C)$ satisfying the conditions of the statement of the corollary, such that $\ind(Q,C)=n$ is $\ll ((abn)^{o(1)} N(n_1,n_1)X)/(b_1n_1^2)$. Since the sum over $n_1$ of $N(n_1,n_1)/n_1^2$ converges, and since $\sum_{b_1\mid b}b_1^{-1}\ll b^{o(1)}$, our result follows.
\end{proof}

\subsection{Weighted uniformity estimates}

Recall the function $M_p^1:\ol{V(\Z)}\to\Z_{\geq 0}$ defined in \S8.2. Let $\delta_p^{\nm}$ denote the indicator function on $\ol{V(\Z)}$ of the set of elements which are nonmaximal at $p$. Define the function $M_p^2:\ol{V(\Z)}\to\Z$ by $M_p^2=\delta_p^\nm-M_p^1$. The functions $M_p^1$, $\delta_p^\nm$, and $M_p^2$ are defined modulo $p^2$. We will abuse notation and use $M_p^1$, $\delta_p^\nm$, and $M_p^2$, to also denote the corresponding functions $V(\Z)\backslash\{\Delta = 0\}\to\Z$ and $V(\Z/p^2\Z)\to\Z$.

For a positive real number $X$, let $\ol{V(\Z)}_X$ denote the set of elements $x\in \ol{V(\Z)}$ with $0<|\Delta(x)|<X$. We have the following result.
\begin{proposition}\label{prop:unifMp1}
We have
\begin{equation*}
\begin{array}{rcl}
|\{x\in\ol{V(\Z)}_X:\,M_p^1(x)= 1\}|
&\ll& \displaystyle\frac{X^{1+o(1)}}{p^2}; \\[.15in]
|\{x\in\ol{V(\Z)}_X:\,M_p^1(x)> 1\}|
&\ll& \displaystyle\frac{X^{1+o(1)}}{p^4}; \\[.15in]
|\{x\in\ol{V(\Z)}_X:\,M_p^1(x)\gg_\epsilon p^\epsilon\}|
&\ll& \displaystyle\frac{X^{1+o(1)}}{p^5}; \\[.15in]
|\{x\in\ol{V(\Z)}_X:\,M_p^1(x)\gg_\epsilon p^{1+\epsilon}\}|
&\ll& \displaystyle\frac{X^{1+o(1)}}{p^8}. 
\end{array}
\end{equation*}    
\end{proposition}
\begin{proof}
First, if $M_p^1(x)\neq 0$, then $x$ corresponds to a triple $(Q,C)$ where $Q$ is nonmaximal at $p$. Then Corollary \ref{cor:uniformabestimate} provides the required bound. Second, if $|M_p^1(x)|>1$, then $p^4\mid\Delta(x)$ well by Proposition~\ref{prop:Mp1_prelim}. The required bound now follows from Corollary \ref{cor:uniformabestimate}.
Third, assume that $|M_p^1(x)|\gg_\epsilon p^\epsilon$, where $x\in\ol{V(\Z)}_X$ corresponds to $(Q,C,r)$. We claim that the index of $Q$ in its maximal order $\cO$ is at least $p^3$. Indeed, there are $\gg_\epsilon p^\epsilon$ distinct pairs $(Q',C')$ such that $Q$ has index-$p$ in $Q'$. But the maximal order of each $Q'$ is $\cO$. Hence by Corollary \ref{cor:Nak}, it follows that the index of any such $Q'$ in $\cO$ is at least $p^2$, which implies that the index of $Q$ in $\cO$ is at least $p^3$. The required result now follows from bounding the number of possible choices for $\cO$ by $O(X^{1+o(1)}/p^{6})$ using Theorem \ref{prop:unif_fields}, and the number of choices for $(Q,C)$ given $\cO$ by $O(p)$ using Corollary \ref{cor:Nak}. The fourth and final estimate follows in identical fashion by observing that if $|M_p^1(x)|\gg_\epsilon p^{1+\epsilon}$, where $x\in\ol{V(\Z)}_X$ corresponds to $(Q,C)$, then the index of $Q$ in its maximal order is at least $p^5$.
\end{proof}

We prove the analogous result for $M_p^2$.

\begin{proposition}\label{prop:unifMp2}
We have
\begin{equation*}
\begin{array}{rcl}
|\{x\in \ol{V(\Z)}_X:\,|M_p^2(x)|\geq 1 \}|
&\ll& \displaystyle\frac{X^{1+o(1)}}{p^4}; \\[.15in]
|\{x\in \ol{V(\Z)}_X:\,|M_p^2(x)|\gg_\epsilon p^\epsilon\}|
&\ll& \displaystyle\frac{X^{1+o(1)}}{p^5}; \\[.15in]
|\{x\in \ol{V(\Z)}_X:\,|M_p^2(x)|\gg_\epsilon p^{1+\epsilon}\}|
&\ll& \displaystyle\frac{X^{1+o(1)}}{p^8}.
\end{array}
\end{equation*}    
\end{proposition}
\begin{proof}
Since $\delta_p^{\nm}=M_p^2+M_p^1$ is an indicator function, and only takes on the values $0$ and $1$, only the first displayed equation of the proposition requires justification; the others follow from Proposition \ref{prop:unifMp1}. Let $x\in V(\Z)$ be such that $M_p^2(x)>0$. By Part 1 of Proposition \ref{prop:Mp1_prelim}, it follows that $x$ is nonmaximal at $p$. It follows that $\delta_p^\nm(x)=1$, and so $M_p^1(x)>1$ well. Then by Part 2 of Proposition \ref{prop:Mp1_prelim}, it follows that $p^4\mid\Delta(x)$.
The required bound now follows from Corollary \ref{cor:uniformabestimate}.
\end{proof}

For a squarefree positive integer $q$ and $x\in\ol{V(\Z)}$, we define $M^i_q(x):=\prod_{p\mid q}M^i_p(x)$ for $i\in\{1,2\}$, and for relatively prime positive integers $a,b$ we define $M_{a,b}(x):=M^1_a(x)M^2_b(x)$. For a $G(\Z)$-invariant function $\phi:V(\Z)\to\R$, we let $N(\phi,X)$ denote the sum over elements $x\in\overline{V(\Z)}_X$ of $\phi(x)$. We have the following consequence of the proofs of Propositions \ref{prop:unifMp1} and \ref{prop:unifMp2}.
\begin{corollary}\label{cor:weighted_unif}
Let $a$ and $b$ be positive squarefree relatively prime integers. Then we have
\begin{equation*}
N(|M_{a,b}|,X)\ll \frac{X^{1+o(1)}}{a^2b^4}.
\end{equation*}
\end{corollary}

\medskip

Recall that for a positive integer $n$, and a function $\phi:V(\Z/n\Z)\to\C$, we normalize the Fourier transform $\widehat{\theta}:V^*(\Z/n\Z)\to\C$ as follows:
\begin{equation*}
\widehat{\theta}(y):=\frac{1}{n^{12}}\sum_{x\in V(\Z/n\Z)}e\Big(\frac{[ x,y]}{n}\Big)\theta(x).
\end{equation*}
In what follows, we prove results analogous to Propositions \ref{prop:unifMp1} and \ref{prop:unifMp2} and Corollary \ref{cor:weighted_unif} for the functions $\widehat{M_p^1}$ and $\widehat{M_p^2}$. To this end, recall that we have
\begin{equation*}
M_p^1 = 
\sum_{(w,v)\in \P(W_{\F_p})\times \P(U^{\vee}_{\F_p})^v}\chi_{w,v},
\end{equation*}
where $\chi_{w,v}$ denotes the characteristic function of the corresponding sublattice of $V(\Z)$ satisfying \eqref{eq:T1_conditions}.
Dualizing, we obtain
\begin{equation}\label{eq:Mp1hateq}
\wh{M_p^1} = 
p^{-5}\sum_{\substack{L\subset \P(W_{\F_p})\\ v\subset \P(U_{\F_p}^\vee)}}\chi'_{L,v}    
\end{equation}
where  $L$ is a line, and $\chi'_{L,v}$ is the characteristic function of the set of all $(r:U\to W)$ satisfying the following conditions:
\begin{itemize}
\item[{\rm (1)}] $r(L)\equiv 0\mod p^2$
\item[{\rm (2)}] $L$ is in the kernel of $r$ modulo $p$
\item[{\rm (3)}] $q_v\equiv 0\mod p$
\item[{\rm (4)}] $L$ is in the kernel of $q_v$ modulo $p^2$.
\end{itemize}
We have the following result.

\begin{proposition}\label{prop:M1phatbound}
The $L^\infty$-norm of $\widehat{M_p^1}$ is bounded by $O(p^{-2})$. Moreover,
\begin{equation*}
\begin{array}{rcl}
|\{x\in\ol{V(\Z)}_X:\,|\widehat{M_p^1}(x)|>0\}|
&\ll& \displaystyle\frac{X^{1+o(1)}}{p^{16}}; \\[.15in]
|\{x\in\ol{V(\Z)}_X:\,|\widehat{M_p^1}(x)|> 1/p^5\}|
&\ll& \displaystyle\frac{X^{1+o(1)}}{p^{17}}; \\[.15in]
|\{x\in\ol{V(\Z)}_X:\,|\widehat{M_p^1}(x)|\gg 1/p^4\}|
&\ll& \displaystyle\frac{X^{1+o(1)}}{p^{20}}. 
\end{array}
\end{equation*}    
\end{proposition}
\begin{proof}
The first assertion is immediate from \eqref{eq:Mp1hateq}.
Our proof of the estimates claimed in the proposition will make extensive use of ``switching tricks''. To illustrate this in the first case, take an element $x\in \ol{V(\Z)}$, corresponding to $r:W\to U$, with $\widehat{M_p^1}(x)\geq 1$. We know by definitition that there exist $L\subset \P(W_{\F_p})$ and $v\in\P(U_{\F_p}^\vee)$, satisfying the above four conditions. By choosing coordinates so that $L$ is spanned by the first two basis vectors of $W$, and $v$ is the second basis vector of $U^\vee$, we obtain an element $(A,B)$ in the orbit of $x$ in the form
\begin{equation*}
\left[\begin{pmatrix}{*}&{0}&{0}\\{0}&{0^2}&{0^2}\\{0}&{0^2}&{0^2}\end{pmatrix},\begin{pmatrix} {0}&0^2&0^2\\0^2&{0^2}&{0^2}\\0^2&{0^2}&{0^2}\end{pmatrix}\right].
\end{equation*}
We may multiply the first row and column of the above pair by $p$, and divide by $p^2$, thereby lowering the discriminant of $(A,B)$ by a factor of $p^{16}$. This can be made to give the required saving.

We now make this rough argument precise. Given $r:W\to U$ along with an $L$ and $v$ as above, we define $W':=L/p+W$ and notice that $r(W')\subset U$. We thus obtain
\begin{equation*}
N(p^5\wh{M_p^1},X)=N(J_p,X/p^{16}) \quad {\rm for}\quad
J_p= \sum_{\substack{w\in \P(W_{\F_p})\\ v\in \P(U^{\vee}_{\F_p})}}\chi^0_{w,v},
\end{equation*}
where $\chi^0_{w,v}$ is the characteristic function of the set of $x\in\ol{V(\Z)}$ corresponding to $r:W\to U$ satisfying the condition that $q_v$ has $w$ in its kernel. To prove the first assertion, it thus suffices to prove that $N(J_p,X)\ll X^{1+o(1)}$.

Given an element $x$ corresponding to $r:W\to U$, we let  $B^1(x)$ denote the variety of pairs $(w,v)$ such that $q_v$ has $w$ in its kernel. This is clearly the $\F_p$ points of an algebraic variety of bounded degree. Hence, the size of $B^1(x)$ is $\Theta(p^k)$ for some $k\geq 0$. We break up $\ol{V(\Z)}$ into the following subsets.

\begin{enumerate}
    \item \textbf{The set of $x\in \ol{V(\Z)}$ such that $B^1(x)$ dominates $\P^1(U^{\vee}_{\F_p})$:}
    
    For such $x$ (corresponding to $r:W\to U$), every $v\in \P(U_{\F_p}^\vee)$ is such that $q_v$ has a kernel mod $p$ and hence $p\mid\det(q_v)$. Therefore, the cubic resolvent of $x$ is $0$ modulo $p$, implying that $p^4\mid\Delta(x)$ well. Since $J_p(x)\ll p^3$, by Corollary \ref{cor:uniformabestimate} we see that this set contributes at most $\frac{X^{1+o(1)}}{p}$ to $N(J_p,X)$. 
    
    For the remaining casework we assume $B^1$ does not dominate $\P^1(U^{\vee})$we assume this is not the case.

    \item \textbf{The set of $x\in \ol{V(\Z)}$ such that $B^1(x)$ does not dominate $\P^1(U^{\vee}_{\F_p})$ and $J_p(x)\asymp p^2$:}
    
    Let $x$, corresponding to $r:W\to U$, be an element of this subset. Since $B^1(x)$ does not dominate $\P^1(U^{\vee}_{\F_p})$, this means we have a single $v\in \P^1(U^{\vee}_{\F_p})$ such that $q_v=0$. We may then define $U'=\ker v$ to obtain an element $x'$ corresponding to $(r:W\ra U')$ of discriminant $\Delta(x)/p^6$. The map sending $x$ to $x'$ is at most $p$-to-1. Hence the contribution of this subset to $N(J_p,X)$ is $\ll p \cdot p^2\cdot X/p^6=X/p^3$. 
    
    \item \textbf{The set of $x\in \ol{V(\Z)}$ such that $B^1(x)$ does not dominate $\P^1(U^{\vee}_{\F_p})$ and $J_p(x)\asymp p$:}

    Let $x$, corresponding to $r:W\to U$, be an element of this subset. This time, since $B^1(x)$ does not dominate $\P^1(U^{\vee}_{\F_p})$, we have a single $v\in \P^1(U^{\vee}_{\F_p})$ such that the kernel of $q_v$ contains a 2-dimensional subspace $L\subset W_{\F_p}$. As a consequence, the cubic resolvent of $x$ must be non-maximal at $p$, implying that $p^2\mid\Delta(x)$ well. Therefore, by Corollary \ref{cor:uniformabestimate}, we see that this set contributes at most  $\frac{X^{1+o(1)}}{p}$ to $N(J_p,X).$ 
\end{enumerate}
Since the contribution of the set of $x\in\ol{V(\Z)}$ with $J_p(x)=O(1)$ contributes at most $O(X)$ to $N(J_p,X)$, the first claim of the proposition follows from these three cases.

The proof of the other two cases are similar. We shall only provide the rough argument of the proof, leaving the linear algebraic details to the reader. We turn to the second claim of the proposition. If $x\in\ol{V(\Z)}$, corresponding to $r:W\to U$, with $\widehat{M_p^1}>1/p^5$, then there must be two pairs $(L,v)\neq (L',v')$ satisfying the four conditions listed above the proposition. We consider the three cases $L=L'$, $v=v'$, and $L\neq L'$, $v\neq v'$. In each of these three cases, we pick bases for $W$ and $U$ such that: the last two basis vectors of $W$ span $L$ and, if $L'\neq L$, the first two basis vectors of $W$ span $L'$; and the second basis vector of $U$ is $v$ and, if $v\neq v'$, the first basis vector of $U$ is $v'$. Picking this basis gives us a pair $(A,B)$ in the orbit of $x$ in each of the three cases. Note that in the third case, we have $(A,B)\in p^2V(\Z)$, and the corresponding bound of $O(X/p^{24})$ is more than sufficient. Below we list the $(A,B)$ obtained in the first and second cases:
\begin{equation*}
\left[\begin{pmatrix} {0}&0^2&0^2\\0^2&{0^2}&{0^2}\\0^2&{0^2}&{0^2}\end{pmatrix},
\begin{pmatrix} {0}&0^2&0^2\\0^2&{0^2}&{0^2}\\0^2&{0^2}&{0^2}\end{pmatrix}\right],\quad\quad
\left[\begin{pmatrix}{0^2}&{0^2}&{0}\\{0^2}&{0^2}&{0^2}\\{0}&{0^2}&{0^2}\end{pmatrix},
\begin{pmatrix} {0^2}&0^2&0^2\\0^2&{0^2}&{0^2}\\0^2&{0^2}&{0^2}\end{pmatrix}\right].
\end{equation*}
For the first case, we multiply the first row and column by $p$, and divide the resulting pair by $p^2$ (similar to the previous reduction to $J_p$). We are left with a pair $(A',B')\in V(\Z)$ whose cubic resolvent is identically $0$. This gives a bound of $O(X/p^{20})$, which is sufficient. In the second case, we divide $(A,B)$ by $p$ (reducing the discriminant by $p^{12}$ via an injective map) and then further divide the second matrix by $p$ (reducing the discriminant by $p^6$ via a map which is at most $p$-to-$1$). The contribution from this case is therefore bounded by $O(X/p^{17})$, which is sufficient.

We turn to the third assertion of the proposition. Given $x\in \ol{V(\Z)}$, corresponding to $r:W\to U$, with $M_p^1(x)\gg 1/p^4$, denote the set of pairs $(L,v)$ by $B^2(x)$. We have $|B^2(x)|\gg p$ from \eqref{eq:Mp1hateq}. We have already noted above that if $B^2(x)$ contains $(L,v)$ and $(L',v')$ with $L\neq L'$ and $v\neq v'$ then $x$ is an orbit in $p^2V(\Z)$, and so there are at most $O(X/p^{24})$ possibilities for $x$. Similarly, if $B^2(x)$ contains $(L,v)$ for a fixed $v$ and $p$ different $L$'s, then $x$ is easily seen to be an orbit in $p^2V(\Z)$. Finally, as noted above, if $B^2(x)$ contains $(L,v)$ and $(L,v')$ for $v\neq v'$, then we have a bound of $O(X/p^{20})$, which is sufficient.
\end{proof}

We also need the analogous result for the function $\widehat{M_p^2}$. To do this, we rely heavily on Hough's work \cite{hough}, in which the Fourier transform of $1-\delta^{\nm}_p$ is computed.
We have the following result.

\begin{proposition}\label{prop:M2phatbound}
The $L^\infty$-norm of $\widehat{M_p^2}$ is bounded by $O(p^{-2})$. Moreover,
\begin{equation*}
\begin{array}{rcl}
|\{x\in\ol{V(\Z)}_X:\,|\widehat{M_p^2}(x)|>0\}|
&\ll& \displaystyle\frac{X^{1+o(1)}}{p^{5}}; \\[.15in]
|\{x\in\ol{V(\Z)}_X:\,|\widehat{M_p^2}(x)|\gg_\epsilon p^{-7+\epsilon}\}|
&\ll& \displaystyle\frac{X^{1+o(1)}}{p^{10}}; \\[.15in]
|\{x\in\ol{V(\Z)}_X:\,|\widehat{M_p^2}(x)|\gg_\epsilon p^{-5+\epsilon}\}|
&\ll& \displaystyle\frac{X^{1+o(1)}}{p^{12}}. \\[.15in]
\end{array}
\end{equation*}    
\end{proposition}
\begin{proof}
Since we have $M_p^2=\delta^{\nm}_p-M_p^1$, it suffices from Proposition \ref{prop:M1phatbound} to instead obtain the stated bounds on $\widehat{\delta^{\nm}_p}$. The first assertion follows directly since the density of elements that are nonmaximal at $p$ is $\ll p^{-2}$.
To prove the remaining estimates, we use the computation of $\widehat{1-\delta^{\nm}_p}$ in \cite[Theorem 2]{hough}\footnote{Note that Hough's normalization of the Fourier transform differs from ours by a factor of $p^{24}$.}.
Define $A_p\subset \ol{V(\Z)}$ to be the subset of elements $x$ corresponding to pairs $(r:W\to U)$, such that there exists $v\in \P^1(U^{\vee}_{\F_p})$ for which $q_v\equiv 0\mod p$. (Equivalently, every pair $(A,B)$ in the orbit of $x$ is such that $\langle A,B\rangle$ generates a dimension $\leq 1$ set over $\F_p$.) Likewise, define $A_{p^2}\subset \ol{V(\Z)}$ to be the subset of pairs $x$ corresponding to $(r:W\to U)$ such that there exists a $v\in \P^1(U^{\vee}_{\Z/p^2\Z})$ for which $q_v\equiv 0\mod p^2$. We have seen in the proof of the previous proposition that the number of orbits in $A_p$ with discriminant less than $X$ is bounded by $O(X/p^5)$. Similarly, the number of orbits in $A_{p^2}$ with discriminant less than $X$ is bounded by $O(X/p^{10})$.

We first note from \cite[Theorem 2]{hough}, that 
the support of $\widehat{\delta^{\nm}_p}$ is contained in $A_p$. This immediately implies the first displayed equation of the proposition. Moreover, by \cite[Theorem 2]{hough}, the function $\widehat{\delta^{\nm}_p}$ is bounded by $O(p^{-7})$ outside of Orbit 1 of Case $\cO_{D1^2}$ (in which the Fourier transform is bounded by $p^{-5}$) and Case $\cO_0$. (The notation of these cases are following that of $\cite{hough}$.) The support of Orbit 1 of Case $\cO_{D1^2}$ is contained in $A_p$, which implies the second displayed equation of the proposition. The support of all the orbits in Case $\cO_0$ are contained in $p\ol{V(\Z)}$. The third displayed equation then follows from the bound $(p\ol{V(\Z)})_X$.
\end{proof}

For a squarefree positive integer $q$ and $x\in\ol{V(\Z)}$, we have $\widehat{M^i_q}(x):=\prod_{p\mid q}\widehat{M^i_p}(x)$ for $i\in\{1,2\}$, and for relatively prime positive integers $a,b$ we have $\widehat{M_{a,b}}(x):=\widehat{M^1_a}(x)\widehat{M^2_b}(x)$. 
Finally, we may combine the arguments of the proofs of the above two propositions, obtaining the following result for composite numbers:

\begin{corollary}\label{cor:boundsftcomposite}
Let $a,b$ be relatively prime positive integers. 
$$N(|\widehat{M_{a,b}}|,X)\ll \frac{X^{1+o(1)}}{a^{21}b^{12}}.$$

\end{corollary}

\section{Executing the Sieve}
In this section we prove Theorem \ref{thm:mainallfields}, which we restate here for convenience:

\begin{theorem}
Let $F(\Sigma)$ be a family of quartic fields, and let $\psi:\R_{\geq 0}\to\R_{\geq 0}$ be a smooth function with compact support. Then 
\begin{equation*}
N_{\Sigma}(\psi,X)=C_1(\Sigma,\psi)\cdot X+
C'_{5/6}(\Sigma,\psi)\cdot X^{5/6}\log X + C_{5/6}(\Sigma,\psi)\cdot X^{5/6} + O(X^{13/16+o(1)}).
\end{equation*}
\end{theorem}
Though we do not give explicit descriptions of the constants $C_1(\psi)$, $C'_{5/6}(\psi)$, and $C_{5/6}(\psi)$, our proof will in fact express them as sums of residues (at $1$ and $5/6$) of certain Shintani zeta functions; our constants are inexplicit because these residues are inexplicit in general. However, when $\Sigma$ is an $S_4$-family, our results in the previous part of the paper give explicit descriptions for the residues of the relevant Shintani zeta functions. Hence, combing those results with our proof of Theorem \ref{thm:mainallfields} will yield Theorem \ref{thm:mainS4fields}.

\subsection{The inclusion exclusion sieve}
Recall that $\Sigma=(\Sigma_v)_{v\in S}$ consists of finite sets $\Sigma_v$ of \'etale quartic extensions of $\Q_v$, for $v$ in a finite set $S$ of places of $\Q$. For each finite place $p\in S$, let $\Lambda_p\subset V(\Z_p)$ be the set of elements whose associated quartic ring is the maximal order of an algebra belonging to $\Sigma_p$. Then $\Lambda_p$ is a subset defined by congruence conditions modulo $p^2$ (up to some possible bounded powers of $2$).
For a finite prime $p\not\in S$, we define $\Lambda_p$ to be the set of elements in $V(\Z_p)$ which correspond to triples $(Q,R,r)$, where $Q$ is maximal at $p$. We set $\Lambda_\infty$ to be the set of elements in $V(\R)$ whose corresponding quartic extension of $\R$ belongs to $\Sigma_\infty$. Let $L(\Lambda)$ denote the set of elements $x\in V(\Z)$, such that $x\in\Lambda_v$ for every place $v$. Then $L(\Lambda)$ is $G(\Z)$-invariant, and $\ol{L(\Lambda})$ is in bijection with $F(\Sigma)$.

To carry out a smoothed count of the number of elements in $\ol{L(\Lambda})$, we use an inclusion exclusion sieve. To set this up, we let $P$ denote the product of the finite primes in $S$, and let $\delta_{\Sigma}$ denote the characteristic function of the 
set of elements $x\in V(\Z)$, such that $x\in\Lambda_p$ for all $p\mid S$. Then $\delta_{\Sigma}$ is defined modulo $P^2$ (again, up to some bounded powers of $2$ which will not effect the error term).
For each prime $p\not\in S$ we recall that $\delta^{\nm}_p\subset V(\Z)$ is the characteristic function of set of elements in $V(\Z)$ corresponding to quartic rings which are non-maximal at $p$. For each positive integer $m$ coprime to $P$, we set $\delta^{\nm}_m:=\prod_{p\mid m} \delta^{\nm}_p$. Without loss of generality, we will assume that $\Sigma_\infty=\{\R^{4-2i}\times\C^i\}$ is a singleton set. 
For a finite collection $\Sigma$ of local specifications of quartic fields, let $F'(\Sigma)$ denote the family of quartic \'etale algebras $K$ with $K\otimes\Q_v\in\Sigma_v$ for all $v\in S$.
Let $N'_\Sigma(\psi;X)$ be the smoothed sum of the quartic \'etale algebras defined by
\begin{equation}\label{eq:fields_count}
N'_{\Sigma}(\psi,X):=\sum_{K\in F'(\Sigma)}\psi\Bigl(\frac{|\Delta(K)|}{X}\Bigr).
\end{equation}
Then by the inclusion-exclusion sieve, we have
$$N'_{\Sigma}(\psi,X)=\sum_{\substack{m\geq 1\\ (m,P)=1}}\mu(m)N^{(i)}(\delta_\Sigma\delta^{\nm}_m;\psi;X).$$
In fact, we will split this up via through the decomposition studied earlier: $\delta^{\nm}_p=M^1_p+M^2_p$. We obtain

\begin{equation}\label{eq:mainsievesetup}
    N'_{\Sigma}(\psi,X)=\sum_{\substack{a,b\geq 1\\ (ab,P)=1}}\mu(ab)N^{(i)}(M^{\Sigma}_{a,b};\psi;X),
\end{equation}
where $M^{\Sigma}_{a,b}:=\delta_\Sigma M_{a,b}$. In the next subsections we shall analyze the power-series expansions for $N(M^{\Sigma}_{a,b};\phi;X)$ given by the theory of Shintani zeta functions, being careful about uniformity with respect to $a$ and $b$.

\subsection{The power series expansion}

Let $i\in\{0,1,2\}$ be fixed. Recall the functional equation of the Shintani zeta functions due to Sato-Shintani, and stated in Theorem \ref{th:SZF_v2}. We restate it here for $\xi_i(M_{a,b}^\Sigma;s)$: we have
\begin{equation}\label{eq:funceq_M}
\xi_i(M^{\Sigma}_{a,b};1-s)=(abP)^{24s}\gamma(s-1)\sum_{j\in\{0,1,2\}}c_{ji}(s)\xi^*_j(\widehat{M^{\Sigma}_{a,b}};s).
\end{equation}
Recall also that the functions $\xi(\phi;s)$ have (at most) a double pole at $1$, $5/6$, and $3/4$. We let the expansion of $\xi_i(M^{\Sigma}_{a,b};s)$ around $c=1,\frac56,\frac34$ be given by $\xi_i(s,M_{a,b})=\frac{r_2(a,b,c)}{(s-c)^2}+\frac{r(a,b,c)}{s-c} + O(1)$. 

The purpose of this subsection is to prove the following result.
\begin{proposition}\label{prop:residueexpansion}
Let $\psi:\R_{>0}\to\R_{\geq 0}$ be a smooth and compactly supported function. Then we have
\begin{equation*}
\begin{array}{rcl}
N^{(i)}(M^\Sigma_{a,b};\psi;X)&=&\displaystyle\sum_{c\in \{1,\frac34,\frac56\}}X^c\left(\log X\tilde{\phi}(c)r_2(a,b,c)+\tilde{\phi}(c)r(a,b,c)+(\tilde{\phi})'(c)r_2(a,b,c) \right)
\\[.2in]&&\displaystyle
+ O(a^{3+o(1)}b^{12+o(1)}).
\end{array}
\end{equation*}
\end{proposition}

\begin{proof}
We begin by invoking Mellin inversion to write
\begin{equation*}
N^{(i)}(M^\Sigma_{a,b},X;\phi)=\int_{2}\xi_i(M^{\Sigma}_{a,b};s) X^s\wt{\psi}(s)ds.
\end{equation*}
Pulling the integral to $\Re(s)=-1$, we pick up possible terms at the possible poles of $\xi$ at $1$, $5/6$, and $3/4$, obtaining a main term contribution to $N^{(i)}(M_{a,b},X;\phi)$ of
\begin{equation*}
\sum_{c\in \{1,\frac34,\frac56\}}X^c\left(\log X\tilde{\phi}(c)r_2(a,b,c)+\tilde{\phi}(c)r(a,b,c)+(\tilde{\phi})'(c)r_2(a,b,c) \right),
\end{equation*}
along with an ``error term'' of
\begin{equation*}
\int_{-1}\xi_i(M^{\Sigma}_{a,b};s) X^s\wt{\psi}(s)ds=
-\int_{2}\xi_i(M^{\Sigma}_{a,b};1-t) X^{1-t}\wt{\psi}(1-t)dt.
\end{equation*}
Applying the functional equation \eqref{eq:funceq_M}, we see that the error term is $\ll$
\begin{equation*}
X\int_{2}\Big(\frac{(ab)^{24}}{X}\Big)^t\wt{\psi}(1-t)\gamma(t-1)\sum_{j\in\{0,1,2\}}c_{ji}(t)\xi_j^*(\widehat{M^{\Sigma}_{a,b}};t)\ll X N(|\widehat{M^\Sigma_{a,b}}|,(ab)^{24}/X).
\end{equation*}
Since we have $\wh{M^{\Sigma}_{a,b}} = \wh{M_{a,b}}\cdot\wh{\delta_{\Sigma}}$, the result follows from Corollary \ref{cor:boundsftcomposite}.
\end{proof}

\subsection{Estimating the residues at $1$, $5/6$, and $3/4$}
In this section we derive estimates for $r(a,b,c)$ and $r_2(a,b,c)$ for $c\in\{1,5/6,3/4\}$. Pick a smooth and compactly supported function $\psi:\R_{>0}\to\R_{\geq 0}$ whose Mellin transform is non-zero at $1$, $5/6$, and $3/4$. From Corollary \ref{cor:uniformabestimate} and Proposition \ref{prop:residueexpansion}, we obtain
$$\sum_{c\in \{1,\frac34,\frac56\}}X^c\left(\log X\tilde{\psi}(c)r_2(a,b,c)+\tilde{\psi}(c)r(a,b,c)+(\tilde{\psi})'(c)r_2(a,b,c) \right) \ll \frac{X^{1+o(1)}}{a^{2+o(1)}b^{4+o(1)}}+a^{3+o(1)}b^{12+o(1)}.$$
We claim that the above implies bounds of the same strength for each of the six individual terms 
$$\{X^cr_2(a,b,c),X^cr(a,b,c):c\in\{1,5/6,3/4\}\}.$$ 
This is implied by the following simple lemma.

\begin{lemma}\label{lem:convex_bound}
Let $a_1,\dots,a_n$, $b_1,\dots,b_n$, and $e_1>e_2>\dots>e_n>0$, be real numbers satisfying the inequality 
\begin{equation}\label{eqlem:convex_bound}
\sum_{i=1}^n \bigl(a_iX^{e_i}\log X + b_iX^{e_i}\bigr) = O(M_1X^{1+o(1)}+M_2),
\end{equation}
for all $X>1$, for some fixed positive real numbers $M_1$ and $M_2$.
Then we also have 
$$\max_{i=1}^n\left(\max(|a_iX^{e_i}|,|b_iX^{e_i}|\right) \ll X^{o(1)}\cdot (M_1X+M_2).$$
\end{lemma}

\begin{proof}
For a real number $1<r<2$, replacing $X$ with $rX$ in \eqref{eqlem:convex_bound}, yields 
$$\sum_{i=1}^n \Big[ r^{e_i}\Bigl(1+\frac{\log r}{\log X}\Bigr) a_iX^{e_i}\log X + r^{e_i}b_iX^{e_i} \Big]= O(M_1X^{1+o(1)}+M_2).$$ 
The claim would now follow immediately if we find $2n$ such numbers $r_i$, such that the matrix $E_{\vec{r}}$, whose $j$'th row is 
$$\left(r_j^{e_1}\Big(1+\frac{\log r_j}{\log X}\Big),r_j^{e_1},\dots,r_j^{e_n}\Big(1+\frac{\log r_j}{\log X}\Big),r_j^{e_n}\right),$$ is invertible with determinant $X^{o(1)}$. By subtracting appropriately scaled even columns from odd ones, and multiplying the odd columns by $\log X$, we see that the determinant of $E_{\vec{r}}$ is equal to $(\log X)^{-2n}$ times the determinant of $F_{\vec{r}}$, whose $j$'th row is 
$$\left(r_j^{e_1}\log r_j,r_j^{e_1},\dots,r_j^{e_n}\log r_j,r_j^{e_n}\right).$$
It is elementary that the functions  $r^t\log r,r^t$ of the variable $r$ are independent as $t$ varies over $\R_+$, so a general choice of $\vec{r}$ will do. This completes the proof.
\end{proof}

We derive the following consequence of the above lemma:
\begin{proposition}\label{prop:residue_bounds}
We have
\begin{equation*}
\begin{array}{rcl}
\max\bigl(r_2(a,b,1),r(a,b,1))&\ll& a^{-2+o(1)}b^{-4+o(1)};
\\[.12in]
\max\bigl(r_2(a,b,5/6),r(a,b,5/6))&\ll& a^{-7/6+o(1)}b^{-4/3+o(1)};
\\[.12in]
\max\bigl(r_2(a,b,3/4),r(a,b,3/4))&\ll& a^{-3/4+o(1)}b^{o(1)}.
\end{array}
\end{equation*}
\end{proposition}
Note in particular that sum over $a$ and $b$ of the residues at $5/6$ converges.

\vspace{.1in}
\begin{proof}
From Lemma \ref{lem:convex_bound}, it follows that
$$ \max_{c\in\{\frac34,\frac56,1\}}\max(r(a,b,c),r_2(a,b,c))X^{c+o(1)} = O\left(\frac{X^{1+o(1)}}{a^{2+o(1)}b^{4+o(1)}}+O(a^{3+o(1)}b^{12+o(1)}\right).$$
We optimize by setting $X=a^5b^{16}$. This yields the bound
 \begin{equation}\label{eq:allresbound}
    \max_{c\in\{\frac34,\frac56,1\}}\max(r(a,b,c),r_2(a,b,c))=O(a^{3-5c+o(1)}b^{12-16c+o(1)}),
 \end{equation}  
proving the result.
\end{proof}

\subsection{Proof of Theorem \ref{thm:mainallfields}: Sieving down to maximal orders in quartic algebras}\label{subsection:sieveexecution}

We finally prove our main Theorem \ref{thm:mainallfields}. Let $\psi:\R_{\geq 0}\to\R_{\geq 0}$ be a smooth function with compact support. For $c\in\{1,5/6\}$, define
$$C_2(c;\psi):=\sum_{(ab,P)=1}\mu(ab) r_2(a,b,c)\tilde{\psi}(c),$$
and
$$C_1(c;\psi):=\sum_{(ab,P)=1}\mu(ab)\left( r_2(a,b,c)(\tilde{\psi})'(c) + r(a,b,c)\tilde{\psi}(c)\right)$$
From \eqref{eq:mainsievesetup} and Corollary \ref{cor:weighted_unif}, we have:
\begin{equation*}
\begin{array}{rcl}
N'_\Sigma(\psi;X) &=& \displaystyle\sum_{\substack{a,b\geq 1\\(ab,P)=1}}\mu(ab)N^{(i)}(M_{a,b}^\Sigma;\psi;X)
\\[.2in]&=&\displaystyle
\sum_{\substack{(ab,P)=1\\a^5b^{16}\leq X}}\mu(ab)N^{(i)}(M_{a,b}^\Sigma;\psi;X)+
\sum_{\substack{(ab,P)=1\\a^5b^{16}> X}}\mu(ab)N^{(i)}(M_{a,b}^\Sigma;\psi;X)
\\[.2in]&=&\displaystyle
\sum_{\substack{(ab,P)=1\\a^5b^{16}\leq X}}\mu(ab)N^{(i)}(M_{a,b}^\Sigma;\psi;X)+
O\Big(
\sum_{\substack{(ab,P)=1\\a^5b^{16}> X}}\frac{X^{1+o(1)}}{a^2b^4}
\Big).
\end{array}
\end{equation*}
We estimate the first sum above using Proposition \ref{prop:residueexpansion}:
\begin{equation*}
\begin{array}{rl}
&\displaystyle
\sum_{\substack{(ab,P)=1\\a^5b^{16}\leq X}}\mu(ab)N^{(i)}(M_{a,b}^\Sigma;\psi;X) \\[.12in]
=& \displaystyle
\sum_{\substack{(ab,P)=1\\a^5b^{16}\leq X}}\mu(ab)
\sum_{c\in \{1,\frac34,\frac56\}}X^c\left(\log X\tilde{\psi}(c)r_2(a,b,c)+\tilde{\psi}(c)r(a,b,c)+(\tilde{\psi})'(c)r_2(a,b,c) \right)
\\[.25in]
&\displaystyle +O\Big(
\sum_{\substack{(ab,P)=1\\a^5b^{16}\leq X}}a^{3+o(1)}b^{12+o(1)}
\Big).
\end{array}
\end{equation*}
Next, we expand the sum over $a$ and $b$ and $c\in\{1,5/6\}$ in the main term of the RHS above to all $a$, $b$ with $(ab,P)=1$, bounding the residues of the tail using Proposition \ref{prop:residue_bounds}. This yields
\begin{equation}\label{eq:inclusion_exclusion_final}
\begin{array}{rcl}
N'_\Sigma(\psi;X)&=&\displaystyle
\sum_{c\in\{1,5/6\}} \left( C_2(c;\psi)X^c\log X +C_1(c;\psi)X^c  \right)
\\[.2in]&&\displaystyle +
O\Bigl(X^{o(1)}
\sum_{a^5b^{16}\leq X}\Bigl(a^{3}b^{12}+\frac{X^{3/4}}{a^{3/4}}\Bigr)
\Bigr) +O\Bigl(
X^{o(1)}\sum_{a^5b^{16}> X}\Big(\frac{X}{a^2b^4}+\frac{X^{5/6}}{a^{7/6}b^{4/3}}
\Bigr)\Big)
\\[.2in]&=&\displaystyle
\sum_{c\in\{1,5/6\}} \left( C_2(c;\psi)X^c\log X +C_1(c;\psi)X^c  \right)+O(X^{13/16+o(1)}).
\end{array}
\end{equation}

To recover $N_\Sigma(\psi;X)$ from $N'_\Sigma(\psi;X)$, note that $F'(\Sigma)\backslash F(\Sigma)$ consists exactly of algebras of the form $\Q\oplus K_3$, $K_2\oplus K_2'$, $\Q\oplus\Q\oplus K_2$, and $\Q^4$, where $K_3$ is a cubic field and $K_2$ and $K_2'$ are quadratic fields. Smoothed counts for the cubic fields are carried out in \cite{SST2}, with two main terms of magnitude $X$ and $X^{5/6}$, with an error term of size $O(X^{2/3+o(1)})$. Smoothed counts for sums of quadratic fields follow from an elementary application of Dirichlet's hyperbola method, combined with a standard squarefree sieve. This yields main terms of size $X\log X$ and $X$, with an error term of size $oO(X^{3/4+o(1)})$. Smoothed counts for quadratic fields follow from an elementary squarefree sieve, giving a main term of size $X$ with an error of $O(X^{1/2+o(1)})$. Subtracting this from the power series expansion of $N_\Sigma'(\psi;X)$ gives
\begin{equation*}
N_{\Sigma}(\psi,X)=C_1(\Sigma,\psi)\cdot X+
C'_{5/6}(\Sigma,\psi)\cdot X^{5/6}\log X + C_{5/6}(\Sigma,\psi)\cdot X^{5/6} + O(X^{13/16+o(1)}),
\end{equation*}
for some constants $C_1(\Sigma,\psi)$, $C_{5/6}'(\Sigma,\psi)$, and $C_{5/6}(\Sigma,\psi)$. Note that the leading constant of the $X\log X$ term much be zero after the subtraction, since the number of quartic fields with discriminant bounded by $X$ is known to be $O(X)$. This completes the proof of Theorem \ref{thm:mainallfields}. $\Box$

\begin{remark}{\rm 
    The argument above is set up to be dependent of where the potential poles actually are, and would give a power series expansion with an error of  $X^{13/16+o(1)}$ taking into account all the poles strictly larger than $13/16$.
}\end{remark}

\subsection{Proof of Theorem \ref{thm:mainS4fields}:}

Let $\Sigma=(\Sigma_v)_{v\in S}$ be an $S_4$-family as in the theorem statement. Let $P=\prod_{p\in S} p$.  For a prime $p\in S$ let $L_p\subset V(\Z_p)$ denote the open subset corresponding to $\Sigma_p$. For a squarefree $n$ coprime to $P$ we define $D_n:=\prod_{p\in S} \chi_{L_p}\prod_{p\mid n} \delta_p^{\nm}$, and let
$D_{a,b}:=\prod_{p\in S} \chi_{L_p}\prod_{p\mid a} M_p^1\prod_{p\mid b} M_p^2$.

By Theorem \ref{th:quartic_szf_residues}, the corresponding zeta functions $\xi_{i}(D_n;s)$ have only simple poles at $s=1$ and $s=5/6$. Moreover, since $\Sigma$ is an $S_4$ family, it follows from \eqref{eq:inclusion_exclusion_final} that we have
$$
N_\Sigma(\psi;X)=N'_\Sigma(\psi;X) = C_1(1;\psi)X+C_1(5/6;\psi)X^{5/6}+O(X^{13/16+o(1)}),
$$
where for $c\in\{1,5/6\}$, we have
$$
C_1(c;\psi)=\sum_{(ab,P)=1}\mu(ab)\wt{\psi}(c)r(a,b,c).
$$
By the relation $\delta_p^{\nm}=M_p^1+M_p^2$, we may rewrite this as 
$$C_1(c;\psi)=\tilde{\psi}(c)\sum_{(n,P)=1}\mu(n)\Res_{s=c}\xi_{i}(D_{n};s).$$
Now Theorem \ref{thm:mainS4fields} follows directly from the computation of these residues in Theorem \ref{thm:residues_final}. $\Box$

\bibliographystyle{abbrv} \bibliography{Main}
\end{document}